\documentclass[11pt]{amsart}
\usepackage{amscd,amsmath,amssymb,amsfonts,verbatim}
\usepackage[cmtip, all]{xy}
\usepackage{MnSymbol}
\usepackage{comment}

\newtheorem{thm}{Theorem}[section]
\newtheorem{prop}[thm]{Proposition}
\newtheorem{lem}[thm]{Lemma}
\newtheorem{cor}[thm]{Corollary}

\theoremstyle{definition}

\newtheorem{defn}[thm]{Definition}
\theoremstyle{remark}
\newtheorem{remk}[thm]{Remark}
\newtheorem{remks}[thm]{Remarks}

\newtheorem{exm}[thm]{Example}
\newtheorem{exms}[thm]{Examples}
\newtheorem{notat}[thm]{Notation}
\numberwithin{equation}{section}

{\hfill$\square$\end{defn}}
{\hfill$\square$\end{remk}}
{\hfill$\square$\end{remks}}
{\hfill$\square$\end{exm}}
{\hfill$\square$\end{exms}}
{\hfill$\square$\end{notat}}

\newcommand{\thmref}{Theorem~\ref}
\newcommand{\propref}{Proposition~\ref}
\newcommand{\corref}{Corollary~\ref}

\newcommand{\lemref}{Lemma~\ref}

\newcommand{\refone}{{\ensuremath 1}}
\newcommand{\reftwo}{{\ensuremath 2}}
\newcommand{\refthree}{{\ensuremath 3}}
\newcommand{\reffour}{{\ensuremath 4}}
\newcommand{\reffive}{{\ensuremath 5}}
\newcommand{\refsix}{{\ensuremath 6}}
\newcommand{\refseven}{{\ensuremath 7}}

\newcommand{\cdh}{{\rm cdh}}
\newcommand{\h}{{\rm H}}
\newcommand{\dm}{{\mathbf{DM}}}
\newcommand{\dr}{{\mathbf{D}}}

\newcommand{\sC}{{\mathcal C}}

\newcommand{\sF}{{\mathcal F}}

\newcommand{\sH}{{\mathcal H}}

\newcommand{\sK}{{\mathcal K}}
\newcommand{\sL}{{\mathcal L}}
\newcommand{\sM}{{\mathcal M}}

\newcommand{\sO}{{\mathcal O}}

\newcommand{\sU}{{\mathcal U}}

\newcommand{\sX}{{\mathcal X}}

\newcommand{\A}{{\mathbb A}}

\newcommand{\F}{{\mathbb F}}
\newcommand{\G}{{\mathbb G}}

\renewcommand{\L}{{\mathbb L}}
\newcommand{\M}{{\mathbb M}}
\newcommand{\N}{{\mathbb N}}
\renewcommand{\P}{{\mathbb P}}
\newcommand{\Q}{{\mathbb Q}}

\newcommand{\T}{{\mathbb T}}

\newcommand{\Z}{{\mathbb Z}}
\renewcommand{\1}{{\mathbb{S}}}

\newcommand{\fm}{{\mathfrak m}}

\newcommand{\ff}{{\mathfrak f}}

\newcommand{\hs}{\heartsuit}

\newcommand{\Ker}{{\rm Ker}}

\newcommand{\Alb}{{\rm Alb}}
\newcommand{\CH}{{\rm CH}}

\newcommand{\surj}{\twoheadrightarrow}
\newcommand{\inj}{\hookrightarrow}
\newcommand{\red}{{\rm red}}

\newcommand{\Hom}{{\rm Hom}}

\newcommand{\Spec}{{\rm Spec \,}}
\newcommand{\sing}{{\rm sing}}
\newcommand{\Char}{{\rm char}}

\newcommand{\ab}{\rm ab}
\newcommand{\nr}{\rm nr}
\newcommand{\Gal}{{\rm Gal}}

\newcommand{\Sch}{{\operatorname{\mathbf{Sch}}}} 
\newcommand{\Reg}{{\operatorname{\mathbf{Reg}}}}

\newcommand{\<}{\langle}
\renewcommand{\>}{\rangle}

\newcommand{\Sm}{{\mathbf{Sm}}}

\newcommand{\Ab}{{\mathbf{Ab}}}

\newcommand{\et}{{\textnormal{\'et}}}

\newcommand{\ds}{{/\kern-3pt/}}

\newcommand{\Nsch}{{\operatorname{\mathbf{NSch}}}}

\newcommand{\tr}{{\operatorname{tr}}}

\newcommand{\un}{\underline}
\newcommand{\ov}{\overline}

\renewcommand{\dim}{\text{\rm dim}}

\newcommand{\tuborg}{\left\{\begin{array}{ll}}
\newcommand{\sluttuborg}{\end{array}\right.}

\newcommand{\zar}{{\rm zar}}
\newcommand{\nis}{{\rm nis}}

\newcommand{\reg}{{\rm reg}}

\newcommand{\tor}{{\rm tor}}

\newcommand{\cf}{{\rm cf}}

\newcommand{\tm}{{\rm t}}
\newcommand{\sh}{{\mathbf{SH}}}

\newcommand{\wt}{\widetilde}
\newcommand{\wh}{\widehat}

\newcommand{\coker}{{\rm Coker}}
\newcommand{\Fil}{{\rm fil}}

\newcommand{\Tab}{{\mathbf {Tab}}}

\newcounter{elno}

\newcounter{elno-abc}   

\newcounter{elno-abc-prime}   

\begin{document}
\title{Tame class field theory over local fields}
\author{Rahul Gupta, Amalendu Krishna, Jitendra Rathore}
\address{Institute of Mathematical Sciences, 4th Cross St, CIT Campus, Tharamani, Chennai,
	600113, India.} 
\email{rahulgupta@imsc.res.in}
%\address{Fakult\"at f\"ur Mathematik, Universit\"at Regensburg, 
%93040, Regensburg, Germany.}
%\email{Rahul.Gupta@mathematik.uni-regensburg.de}
\address{Department of Mathematics, South Hall, Room 6607, University of California
  Santa Barbara, CA, 93106-3080, USA.}
\email{amalenduk@math.ucsb.edu}
\address{Department of Mathematics, South Hall, Room 6607, University of California
  Santa Barbara, CA, 93106-3080, USA.} 
\email{jitendra@math.ucsb.edu}

%\thanks{The first author is supported the SFB 1085 Higher Invariants, Universit ̈at Regensburg.}



\maketitle

\begin{quote}\emph{Abstract.}
	In this paper, we extend the hitherto known tame class field theory for smooth
	curves over local fields to all quasi-projective schemes admitting smooth compactifications over such fields.
  \end{quote}
%\end{abstract}
\setcounter{tocdepth}{1}
%\maketitle
\tableofcontents

%Journal version
\section{Introduction}\label{sec:Intro}
\subsection{Background}\label{sec:Background}
The class field theory for smooth varieties over finite fields is now well
understood. However, this problem over a local field (i.e., a complete discrete
valuation field with finite residue field) is significantly more
challenging and not yet fully understood.
The objective of this paper is to establish the class field theory for
tamely ramified coverings of smooth quasi-projective schemes over such fields. 

We fix a local field $k$. The class field theory for smooth projective curves over
$k$ is now complete, thanks to the works of  Bloch \cite{Bloch-cft}, Saito
\cite{Saito-JNT}, Kato-Saito \cite{Kato-Saito-1} and Yoshida \cite{Yoshida03}.
The tame class field theory for open subschemes of a smooth projective curve over
$k$ is also now completely understood by the works of Hiranouchi
(cf. \cite{Hiranouchi-1}, \cite{Hiranouchi-2}). We shall therefore focus on
the case of higher dimensional schemes.

Jannsen-Saito (cf. \cite{JS-Doc}, \cite{JS-JAG}) showed that if $X$ is either
a smooth projective scheme over $k$ with good reduction or an arbitrary
smooth projective surface over $k$, 
then the reciprocity map $\rho_X \colon SK_1(X) \to \pi^{\ab}_1(X)$
from the idele class group to the abelian {\'e}tale fundamental group has the
property that its kernel is the direct sum of a finite group and another group
which is divisible by every integer invertible in $k$. 
A generalization of
the results of Jannsen-Saito to higher dimension were proven by Forr{\'e} \cite{Forre-Crelle}. He showed that for any smooth projective
scheme $X$ over $k$, the reciprocity map $\rho_X \colon SK_1(X) \to \pi^{\ab}_1(X)$
has the property that its kernel is the direct sum a finite group and another group
which is divisible by every integer prime to the residue characteristic of $k$. 

Let us now assume that $X$ is an open subscheme of a smooth projective scheme
$\ov{X}$ over $k$. Yamazaki \cite{Yamazaki} introduced a tame
class group for $X$, and constructed a reciprocity map if
$\Char(k)$ is zero. His reciprocity
map coincides with that of Hiranouchi for curves.
Uzun \cite{Uzun} showed that Yamazaki's characteristic zero
reciprocity map is an isomorphism with finite
coefficients if $\ov{X}$ admits a good reduction.
The objective of this paper is to study the tame class field theory
of $X$ in full generality. We define a tame idele class group $C^{\tm}(X)$ and
%We show that $C^\tm(X)$ coincides with the class group of Yamazaki.
show that the results of Jannsen-Saito \cite{JS-Doc}
and Forr{\'e} \cite{Forre-Crelle} extend
verbatim to the setting of tame class field theory.
%Our main results are described below.

\subsection{Main results}\label{sec:MR}
Let $\ff$ be the residue field of $k$
such that $\Char(\ff) = p$. Assume that $X$ is an open subscheme of a smooth
projective scheme $\ov{X}$ over $k$.
The first result of this paper is the following extension of the structure
theorems of Grothendieck \cite{SGA-1} and Yoshida \cite[Thm.~1.1]{Yoshida03} to the tame
{\'e}tale fundamental group (cf. \thmref{thm:J-fin} for
a general statement). Grothendieck proved the
prime-to-$p$ part and Yoshida proved the $p$-part of this result
for smooth projective geometrically connected schemes over $k$.
Let $\pi^{\ab, \tm}_1(X)_0$ denote the kernel of the canonical map 
$\pi_X \colon \pi^{\ab, \tm}_1(X) \to G_k$,
where $G_k$ is the abelianized absolute Galois group of $k$.

\begin{thm}\label{thm:Main-2}
    If $X$ is geometrically connected over $k$, then
    \[
      \pi^{\ab, \tm}_1(X)_0 = F \oplus \wh{\Z}^r,
    \]
    where $F$ is a finite group and $r$ is the $\ff$-rank of the special fiber of
    the N{\'e}ron model of $\Alb(\ov{X})$.
 \end{thm}

 \vskip .2cm

The next set of results establishes a class field theory for the tame abelian {\'e}tale
 fundamental groups of $X$ as follows.

\begin{thm}\label{thm:Main-1}
  There exist reciprocity homomorphisms
\[
  \rho_X \colon C(X) \to \pi^{\ab}_1(X); \ \ \ 
\rho^\tm_X \colon C^{\tm}(X) \to \pi^{\ab, \tm}_1(X)
  \]
such that the following hold.

\begin{enumerate}
\item
  The maps $\rho_X$ and $\rho^\tm_X$ are compatible with the canonical maps
  $C(X) \surj C^\tm(X)$ and $\pi^{\ab}_1(X) \surj \pi^{\ab, \tm}_1(X)$. 
\item
  $\Ker(\rho_X) = F' \oplus D$, where $|F'| < \infty$ and $D$ 
  is divisible
  %{\footnote{Asakura-Saito \cite{Asakura-Saito} have shown that $D$ can be large already for curves.}}
by every integer prime to $p$.
\item
  If $\dim(X) \le 2$, then $D$ is divisible by every integer invertible in $k$.
\item
  (2) and (3) hold also for $\Ker(\rho^\tm_X)$.
\item
  If $\dim(X) \le 1$, then $\Ker(\rho^\tm_X)$ is divisible.
\end{enumerate}
\end{thm}

Note that part (5) resolves the problem left open by Hiranouchi \cite{Hiranouchi-2} for
curves.

\vskip .2cm

We let $E^\tm(X) = \Ker(C^\tm(X) \to C(\ov{X}))$ and $J^\tm(X) = \Ker(\pi^{\ab, \tm}_1(X)
\to \pi^{\ab}_1(\ov{X}))$.

\begin{thm}\label{thm:Main-5}
  The group $J^\tm(X)$ is finite and the reciprocity map induces a surjection
  \[
    \rho^\tm_X \colon E^\tm(X)_\tor \surj J^\tm(X).
  \]
\end{thm}

As a part of proof of \thmref{thm:Main-5}, we show that $J^\tm(X)$
can be written as a quotient of the direct sum of the 0-th ramification subgroups of
the function fields of integral curves lying in $X$.
We also obtain a similar description of $E^\tm(X)$ (cf.
\propref{prop:ML-main-3}) which plays a key role in 
proving \thmref{thm:Main-5}. 

\vskip .2cm

\vskip .2cm

Let $C^\tm(X)_0$ denote the kernel of the norm map $N_{X}\colon C^\tm(X) \to k^\times$.
Let $\rho^\tm_{X,0}$ denote the induced reciprocity map
$C^\tm(X)_0 \to \pi^{\ab, \tm}_1(X)_0$.
We let $\coker_{\rm top}(\rho^\tm_X)$ denote the topological cokernel of $\rho^\tm_X$.
Let $r$ be as in \thmref{thm:Main-2}.

\begin{thm}\label{thm:Main-3}
  The groups $({\rm Image}(\rho^\tm_X))_\tor$ and ${\rm Image}(\rho^\tm_{X,0})$
  satisfy the following.
  \begin{enumerate}
  \item
    They are finite if $\Char(k) = 0$.
  \item
    They have finite exponents and their prime-to-$p$ parts are finite if
    $\Char(k) = p$.
    \end{enumerate}
If $X$ is geometrically connected over $k$, then we have the following.
\begin{enumerate}
\item
${\rm Image}(\rho^\tm_{X,0})$ is finite. 
\item
$\coker_{\rm top}(\rho^\tm_X)$ is direct sum of $\wh{\Z}^r$ and a finite group.
 \end{enumerate}
\end{thm}

The above theorem was proven by
Hiranouchi when $X$ is a (geometrically connected smooth) curve. But the first claim
of the theorem is new even for curves.

\vskip .2cm

\begin{thm}\label{thm:Main-4}
  If $X$ is birational to a smooth projective $k$-scheme which
  admits a good reduction, then the reciprocity map induces an isomorphism
  \[
    \rho_X \colon {C(X)}/m \xrightarrow{\cong} {\pi^{\ab}_1(X)}/m
    \]
    for every integer $m$ invertible in $k$. In particular, $\Ker(\rho_X)$
    is divisible by every integer invertible in $k$.
    The same holds also for $\rho^\tm_X$.
    \end{thm}

The above theorem was shown by Yamazaki \cite[Thm.~6.5]{Yamazaki}
under the assumption that $\Char(k) = 0$ and
$\ov{X}$ is a geometrically connected rational surface over $k$.

\vskip .2cm

If $\Char(k) = p$, then the previous results do not explain the $p$-part of
the tame class field theory of $X$. However, the following
result shows that this part actually
coincides with the $p$-part of the class field theory of $\ov{X}$.

\begin{thm}\label{thm:Main-6}
  Suppose that $\Char(k) = p$ and $m \ge 1$ is an integer. Then the canonical
  maps
  \[
    {C^\tm(X)}/{p^m} \to {C(\ov{X})}/{p^m}; \ \ \ {\pi^{\ab, \tm}_1(X)}/{p^m} \to
    {\pi^{\ab}_1(\ov{X})}/{p^m}
  \]
  are isomorphisms.
\end{thm}

\subsection{An application to motivic cohomology}
\label{sec:HMC}
Let $X$ be as in \S~\ref{sec:MR}. As a byproduct of our proofs, we obtain the
following result about the
motivic cohomology groups $H^{2d+n}_c(X, \Z(d+n))$ for $n \ge 2$.
To motivate it, recall that a result of Akhtar \cite[Thm.~1.1]{Akhtar} says that if
$Y$ is a smooth projective scheme of pure dimension $d$ over a finite field, then
$H^{2d+m}(Y, \Z(d+m))$ is zero for $m \ge 2$. The key point of the proof is the well
known fact that the Milnor $K$-group of a finite field is
zero in degrees more than one. 
Since this is not true for a local field $k$, we can not expect Akhtar's result
over $k$.
On the other hand, Tate \cite{Tate-Kyoto} showed that $\CH^{2}(k,2) \cong K^M_2(k)$ is
a direct sum of a finite and a divisible group. The
following result therefore should be the optimal generalization of
Akhtar's result to varieties over local fields.

\begin{thm}\label{thm:Main-9}
We have the following.
  \begin{enumerate}
    \item
  There is a decomposition $H^{2d+2}_c(X, {\Z}(d+2)) \cong F \bigoplus D$, where
  $F$ is a finite group whose order is invertible in $k$ and $D$ is divisible by
  every integer invertible in $k$.
  \item
$H^{2d+n}_c(X, {\Z}(d+n))$ is divisible by
  every integer invertible in $k$ if $n \ge 3$.
  \item
  $\CH^{d+n}(X, n)$ is divisible if $n \ge 3$ and $X$ is projective over $k$.
\end{enumerate}
\end{thm}

The following result generalizes Tate's theorem.

\begin{cor}\label{cor:Tate-higher-dim}
  If $X$ is smooth projective of pure dimension $d$ over a $p$-adic field, then
  $\CH^{d+2}(X,2) \cong F \bigoplus D$, where
  $F$ is a finite group and $D$ is a divisible group.
\end{cor}

\vskip .3cm

\subsection{Outline of proofs}\label{sec:Outline} 
We briefly outline the proofs.
The first key step to prove our main results is to identify the tame class group and the
abelian tame fundamental group with finite coefficients with the motivic cohomology
({\`a} la Suslin-Voevodsky) and {\'e}tale cohomology with compact support ({\`a} la
Grothendieck), respectively.

Although the identification of the abelian tame fundamental group with the
{\'e}tale cohomology with compact support is fairly
straightforward, the proof of the isomorphism between the tame class group and
the motivic cohomology with compact support takes a significant amount of work
(cf. \propref{prop:Tame-MCCS}). Two technical albeit key steps are the existence of the
pull-back map between tame class groups (cf. \propref{prop:TCG-PB}),
and a continuity property of the tame class group (cf. \lemref{lem:Continuity}).
We also establish the existence of the push-forward map  between tame class groups
(cf. \propref{prop:TCG-PF}).
Apart from being vital for our proofs, these results are of independent interest
even if their proofs are long.

The second key step is to show that the above identification
is compatible with the reciprocity map between the class group and fundamental
group on one hand, and the realization map between the motivic cohomology and
{\'e}tale cohomology on the other hand (cf. \propref{prop:Rec-Real}).
This is especially delicate in positive characteristic. We achieve this by means
of purity theorems for motivic cohomology groups 
({\`a} la Cisinski-D{\'e}glise \cite{CD-Doc}) and the theory of Gysin morphisms and 
fundamental classes of regular closed immersions over imperfect fields
({\`a} la D{\'e}glise-Jin-Khan \cite{DJK} and Navarro \cite{Navarro}).
We use these results and a descent argument to eventually reduce the identification
problem to perfect base fields, where we use a result of Yamazaki \cite{Yamazaki}.

The third key step is to uniformly bound the motivic and {\'e}tale
cohomology (with and without compact support) as one varies the finite coefficients
(but fixes the twist).
We have to do this for singular schemes as well because we apply this to the
complement of the underlying smooth scheme in a chosen smooth compactification,
and this complement is generally singular. In order to achieve this, we prove
a general finiteness result for motivic and {\'e}tale cohomology over local fields,
which is of independent interest. This was proven  earlier over finite fields by
Colliot-Th{\'e}l{\`e}ne-Sansuc-Soul{\'e} \cite{CSS} and Kahn \cite{Kahn-03}
(who also proved some finiteness results over $p$-adic fields).
We prove the finiteness over arbitrary
local fields by reducing to the case of finite fields.
This boundedness theorem allows us to reduce the proof of \thmref{thm:Main-1} to the
known case of smooth projective schemes by means of localization sequences.

The above argument only proves parts (1) and (2) of
\thmref{thm:Main-1}. To prove the stronger claims for surfaces, we use our
theorem on the boundedness of {\'e}tale cohomology to reduce the problem to the
case of a projective surface which is a result of Jannsen-Saito \cite{JS-Doc}.
Another key step in our proof of the full divisibility of the kernel of the reciprocity
map for surfaces in characteristic zero
is a result of Tate \cite{Tate-Kyoto} saying that $K^M_2(k)$ is a
direct sum of a divisible and a finite group if $k$ is a local field.

We prove the $p$-divisibility of the kernel of the reciprocity
map for curves in characteristic $p > 0$ by reducing to the case of projective curves
(which was considered by Kato-Saito \cite{Kato-Saito-1})
using \corref{cor:Rec-tame-unr}. 
The proof of \thmref{thm:Main-2} is reduced to the earlier results of
Grothendieck and Yoshida \cite{Yoshida03} for projective schemes
using the finiteness of $J^\tm(X)$. We prove the latter result using the
boundedness theorem for {\'e}tale cohomology discussed in the previous paragraph.

The fourth key step (which is especially essential for Theorems~\ref{thm:Main-5} and
~\ref{thm:Main-6}) is to show that $E^\tm(X)$ and $J^\tm(X)$ 
can be written as quotients of similar groups for smooth curves. To prove this claim
for $E^\tm(X)$, we first replace $C(\ov{X})$ by a (a priori different) group
which we denote by $C^{\nr}(X)$. It is easy to show that
the kernel of the map $C^\tm(X) \to C^{\nr}(X)$ has the desired property. We then
use a moving lemma for Bloch's cycle complex to show that
$C(\ov{X})$ is isomorphic to $C^{\nr}(X)$.  To prove \thmref{thm:Main-5} for curves,
we bound $E^\tm(X)$ and $J^\tm(X)$ by the subquotients of $K^M_2(K)$ and $G_K$,
where $K$ runs through 2-local fields. We then use Kato's reciprocity theorem for
such fields (cf. \cite{Kato-cft-1} and \cite{Kato80})
plus the divisibility theorem of Tate to complete the proof.

To prove \thmref{thm:Main-4}, we use the Lefschetz arguments and Bertini theorems of
Jannsen-Saito \cite{JS-JAG} and Ghosh-Krishna \cite{Ghosh-Krishna-Bertini}
to reduce to the case of surfaces.
To prove the latter case, the key step is to prove a surjectivity statement for the
{\'e}tale realization map for certain motivic cohomology of a smooth projective
surface over a local field. Using a spectral sequence argument, the proof of
this surjectivity is reduced to showing that in certain indices, the
specialization map between the Kato homology of the generic and special fibers of
a smooth model of the projective surface is an isomorphism, part of a set of
conjectures of Kato. But this special case of Kato's conjectures was already solved
in \cite{JS-Doc}. This is the step where we use the good reduction hypothesis.

\subsection{Notations}\label{sec:Notn}
For a Noetherian scheme $X$, we shall denote the category of
separated and finite type schemes over $X$ by $\Sch_X$.
We shall let $\Sm_X$ denote the category of smooth schemes over $X$.
If $X$ is reduced, $X_\sing$ will denote the singular locus of $X$ with reduced closed
subscheme structure and $X_\reg$ the complementary open subscheme.

Throughout the paper, we shall use characteristic exponent for fields. However,
the notation $\Char(k)$ will be used for the actual characteristic of a
field $k$.
For a field extension ${k'}/k$ and $X \in \Sch_k$, we let
$X_{k'} = X \times_{\Spec(k)} \Spec(k')$.
For a reduced Noetherian scheme $X$, we shall let $X_n$ denote the normalization of $X$
and $k(X)$ the total ring of fractions for $X$. We shall let $X_\red$ denote the
largest reduced closed subscheme of $X$. We shall let $X^{(i)}$ (resp. $X_{(i)}$) denote
the set of codimension (resp. dimension) $i$ points on $X$. 
For any $X \in \Sch_k$, we let $\sC(X)$ be the set of closed integral curves in $X$.

For a scheme $X$ endowed with a topology $\tau \in \{\zar, \nis, \et\}$,
a closed subset $Y \subset X$, and a
$\tau$-sheaf $\sF$, we let $H^*_{\tau, Y}(X, \sF)$ denote the $\tau$-cohomology groups
of $\sF$ with support in $Y$. If $\tau$ is fixed in a subsection, we shall drop it
from the notations of cohomology groups. We shall write $H^*_{\tau, Y}(X, \sF)$ as
$H^*_{\tau}(X, \sF)$ whenever $Y = X_\red$.
If $X = \Spec(A)$ is affine, we may write $H^*_{\tau}(X, \sF)$ as
$H^*_{\tau}(A, \sF)$. We shall write $H^q_\et(A, {\Q}/{\Z}(q-1))$ (resp.
$H^q_\et(X, {\Q}/{\Z}(q-1))$) as $H^q(A)$ (resp. $H^q(X)$) for
any $q \ge 0$ (cf. \cite[\S~3.2, Defn.~1]{Kato80}).

We shall let $\Ab$ denote the category of abelian groups.
For $A, B \in \Ab$, the notation $A \otimes B$ will indicate
tensor product of $A$ and $B$ over $\Z$.
For an abelian group $A$ and $m\in \N$, we let 
${}_m A$ (resp. $A/m$) denote the kernel (resp. cokernel)
of the multiplication map $A \xrightarrow{m} A$. For a prime $p$,
we shall let $A\{p\}$ (resp. $A\{p'\}$) denote the $p$-primary (resp. prime-to-$p$
primary) component of $A$. 
%For a commutative ring $\Lambda$, we shall write $M \otimes \Lambda$ as $M_\Lambda$.
For a set $\mathbb{M}$ of prime numbers, the $\M$-completion of $A$
is the inverse limit $A_{\mathbb{M}} = \varprojlim_m A/{m}$, where the limit runs
through all natural numbers $m$ whose prime divisors lie in $\mathbb{M}$.  
%%the definition of Pontryagin dual is recalled in section 2.4. 
%For a topological group $M$, we let $M^{\vee}$ denote the Pontryagin dual
%$\Hom_{\cont}(M, \Q/\Z)$ of $M$. 

For a Noetherian scheme $X$, we shall let $\pi^{\ab}_1(X)$
denote the abelian {\'e}tale fundamental group of $X$.
We shall write $\pi^{\ab}_1(\Spec(F))$
interchangeably as $G_F$, where the latter is the abelianized absolute Galois group
of a field $F$.
%We shall let $I(p) \subset \N$ denote the set of natural numbers not divisible by the
%prime $p$.
We shall let $\Tab$ denote the category of topological abelian groups with
continuous homomorphisms.

\section{The tame fundamental group}\label{sec:TFG}
In this section, we shall recall the definitions of the tame {\'e}tale fundamental
group and tame idele class groups and prove some initial results.
We fix some notations which will be used in this paper.
For a local ring $A$, we let $A^h$ (resp. $\wh{A}$)
denote the Henselization (resp. completion) of $A$ with respect to its maximal ideal.
For a discrete valuation field $K$, we shall let
$\sO_K, \fm_K, \ff$ and $v_K$ denote the ring of integers of $K$, the maximal ideal of
$\sO_K$, the residue field of $\sO_K$ and the normalized valuation of 
$K$, respectively.

\subsection{Recollection of ramification filtration}\label{RF**}
Recall from \cite[\S~7.4]{Gupta-Krishna-REC} that for a topological abelian group $G$,
the Pontryagin dual of $G$ is $G^\vee := \Hom_\Tab(G, \T)$ is a 
topological abelian group with the compact-open topology, where $\T$ is the circle
group. 
When $G$ is either profinite or discrete torsion, then 
$G^\vee \cong \Hom_\Tab(G, {\Q}/{\Z})$, where ${\Q}/{\Z} = \T_\tor$ has the
discrete topology. In this case, the Pontryagin duality says that the evaluation map
${\rm ev}_G \colon G \to (G^\vee)^\vee$ is a topological isomorphism.

Let $K$ be a Henselian discrete valuation field with a separable closure $\ov{K}$.
Let $K^{\ab}$ be the maximal abelian extension of $K$ inside $\ov{K}$ and
$K^{\nr}$ the maximal unramified subextension of $K^{\ab}$.
Let $p \ge 1$ denote the characteristic exponent of $K$.
Recall from \cite{Abbes-Saito} (see also
 \cite{Saito20} and \cite[\S~6.1]{Gupta-Krishna-REC}) that $G_K$ has an upper numbering
 non-logarithmic decreasing filtration $\{G^{(r)}_K\}_{r \in \Q_{\ge -1}}$ by closed
 subgroups
 such that  $G^{(-1)}_K = G_K$ and $G^{(0)}_K$ (resp. $G^{(1)}_K$) is the inertia group
  $\Gal({K^{\ab}}/{K^{\nr}})$ (resp. the wild inertia group) of $G_K$.
  We shall call $G^{(\bullet)}_K$ the ramification filtration of $G_K$.

For any integer $n \ge -1$, we let $\Fil_n H^1(K) =
\{\chi \in H^1(K)| \chi(a) = 0 \ \textnormal{for all } a\in  G^{(n)}_K\}$.
The filtration
$\Fil_{\bullet} H^1(K)$ of $H^1(K)$ is exhaustive such that
(cf. \cite[Lem.~7.11]{Gupta-Krishna-REC})
\begin{equation} \label{eqn:fil_n-dual}
({G_K}/{G^{(n)}_K})^\vee \cong \Fil_n H^1(K) \inj H^1(K) \cong G^\vee_K.
\end{equation}

For any prime $\ell$, we let $\Fil_n H^1_\et(K)\{\ell\} =
H^1(K)\{\ell\} \cap \Fil_n H^1(K) = H^1_\et(K, {\Q_\ell}/{\Z_\ell}) \cap
\Fil_n H^1(K)$.
It follows from \cite[Thm.~6.1]{Gupta-Krishna-REC} that for $n \geq 1$, we have 
\begin{equation}\label{eqn:MS-AS}
\Fil_n H^1(K) =  ({\underset{\ell \neq p}\bigoplus}
    H^1_\et(K, {{\Q}_{\ell}}/{\Z_\ell})) \bigoplus
  \left(\Fil_n H^1_\et(K)\{p\}\right).
\end{equation}

\subsection{Definition of $\pi^{\ab, \tm}_1(X)$}\label{sec:RFil}
For the rest of \S~\ref{sec:TFG},
we fix a field $k$ (of characteristic exponent $p$)
and an integral regular $k$-scheme $X$ of dimension
$d \ge 1$.
We fix a normal compactification (cf. \cite{Nagata}) $\ov{X}$ of $X$ such
that $\ov{X} \setminus X$ is the support of a Weil divisor
$D_0 := \stackrel{r}{\underset{i =1}\sum}D_i$, where $\{D_1, \ldots , D_r\}$ is
the set of irreducible components of $\ov{X} \setminus X$, all of dimension $d-1$.
We allow $D_0$ to be empty. We let $j \colon X \inj \ov{X}$ be the inclusion.
Let $K$ denote the function field of $X$.
For $x \in \ov{X}^{(1)}$, let $A_x = \wh{\sO_{X,x}}$ and let $K_x$ denote the fraction
field of $A_x$.

For an integral normal curve $C$ over $k$ with the unique
normal compactification $\ov{C}$ and function field $k(C)$,
we let $\Delta(C) = \ov{C} \setminus C$
with the reduced induced closed subscheme structure.
For an effective divisor $E =\stackrel{r}{\underset{i =1}\sum} n_i E_i$ having
support on $\ov{C} \setminus C$, we let 
$\Fil_E H^1(k(C))$ be the subgroup of characters $\chi$ of $G_{k(C)}$ satisfying the
property that $\chi \in H^1_\et(C, {\Q}/{\Z}) \subset H^1(k(C))$, and
the image of $\chi$ under the canonical map $H^1(k(C)) \to H^1(k(C)_i)$ lies in
$\Fil_{n_i} H^1(k(C)_i)$ for every
$1 \le i \le r$, where $k(C)_i$ is the completion of $k(C)$ along the discrete
valuation defined by $E_i$.
% (see \cite[Defn.~7.13]{Gupta-Krishna-REC}).
 It follows from \eqref{eqn:MS-AS}
 that 
 \begin{equation}\label{eqn:MS-AS-2}
\Fil_E H^1(k(C)) =   H^1(C) \{p'\} \bigoplus
  \Fil_E H^1(k(C))\{p\}.
\end{equation}

%let $\Fil_E H^1(k(C))$ be the subgroup of characters
%of $\pi^{\ab}_1(C)$ whose ramifications are bounded
%by $E$ 

For any $C \in \sC(X)$, we let
% $\ov{C}$ be
%the scheme theoretic closure of $C$ in $\ov{X}$. We let
$H^1_\tm(C_n) := \frac{H^1(C_n)}{\Fil_{\Delta(C_n)} H^1(k(C_n))}$ and
$H^1_{\nr}(C_n) := \frac{H^1(C_n)}{\Fil_{0} H^1(k(C_n))}$ and consider them
as discrete topological abelian groups (recall that $C_n$ denotes the normalization
of $C$).
We let 
$H^1_\tm(X) := {\underset{C \in \sC(X)}\prod} \ H^1_\tm(C_n)$ and
$H^1_{\nr}(X) = {\underset{C \in \sC(X)}\prod} \ H^1_{\nr}(C_n)$, considered as
topological abelian groups with the product topology.

\begin{defn}\label{defn:Tame-fil}
 We let $\Fil^{\tm} H^1(K) = \Ker(H^1(X) \xrightarrow{\Phi^\tm_X}
 H^1_\tm(X))$ and define the abelianized tame fundamental group of $X$ to be the
 topological abelian group $\pi^{\ab, \tm}_1(X) = \left(\Fil^{\tm} H^1(K)\right)^\vee$
 with the profinite topology.
   \end{defn}

   \begin{remk}\label{remk:Tame-pi*}
     The proof of \cite[Thm.~7.17]{Gupta-Krishna-REC} shows mutatis mutandis that
   $\pi^{\ab, \tm}_1(X)$ describes the Galois category of finite abelian covers of
   $X$ whose pull-back to $C$ is tamely ramified along $\ov{C} \setminus C$
   in the sense of \cite[Defn.~5.7.15]{Szamuely} for every regular curve
   $C$ mapping to $X$ (such covers are called curve-tame).
   By \cite[Thm.~4.4]{Kerz-Schmidt-JNT},
   we get that $\pi^{\ab, \tm}_1(X)$ describes the category of finite
   abelian covers of $X$ which are tamely ramified along $D_0$ in the sense
   of \cite[Defn.~5.7.15]{Szamuely}. 
%   We refer to \cite[\S~2.3]{GKR-arxiv} for proofs of these descriptions.
 \end{remk}

 \begin{prop}\label{prop:Tame-cov}
   Given a morphism of integral regular $k$-schemes $f \colon X' \to X$, the
   push-forward map
   $f_* \colon \pi^{\ab}_1(X') \to \pi^{\ab}_1(X)$ descends to a continuous homomorphism
  \[
    f_* \colon \pi^{\ab,\tm}_1(X') \to \pi^{\ab,\tm}_1(X).
  \]
  \end{prop}
  \begin{proof}
    This is easily deduced from Definition~\ref{defn:Tame-fil}. Alternatively,
    this also follows from the above description of $\pi^{\ab,\tm}_1(X)$ using
    the machinery of Galois category of curve-tame covers as in
    Remark~\ref{remk:Tame-pi*}.
    \end{proof}

 It follows from ~\eqref{eqn:MS-AS-2} that
\begin{equation}\label{eqn:Tame-fil-1}
  \Fil^{\tm} H^1(K) = \Fil^{\tm} H^1(K)\{p\} \bigoplus H^1(X)\{p'\}.
\end{equation}
 In particular, $\Fil^{\tm} H^1(K) \cong H^1(X)$ if $p =1$. We let
\begin{equation}\label{eqn:Tame-fil-0}
  \Fil^{\nr} H^1(K) = \Ker(H^1(X) \xrightarrow{\Phi^{\nr}_X} H^1_{\nr}(X)).
\end{equation}
It follows from \cite[Lem.~8.7]{Gupta-Krishna-Duality} that
$\pi^{\ab}_1(\ov{X}) \cong \left(\Fil^{\nr} H^1(K)\right)^\vee$ if $\ov{X}$ is regular.

It follows from the definition that there are canonical surjections
 of topological abelian groups $\pi^{\ab}_1(X) \surj \pi^{\ab, \tm}_1(X) \surj
 \pi^{\ab}_1(\ov{X})$. We let $W^\tm({X})$ be the kernel of the first surjection.
The structure map $\pi_X \colon X \to \Spec(k)$ induces a continuous homomorphism
$\pi_X \colon \pi^{\ab, \tm}_1(X) \to G_k$. 
We let $\pi^{\ab, \tm}_1(X)_0$ denote the kernel of this map.
An immediate consequence of Definition~\ref{defn:Tame-fil} is the following.

 \begin{lem}\label{lem:Tame-fil-curve}
   Assume that $\dim(X) = 1$. Then there is a canonical isomorphism of
    topological abelian groups $W^\tm(X) \xrightarrow{\cong} 
    \left(\frac{H^1_{\et}(X, {\Q}/{\Z})}{\Fil_{\Delta(X)} H^1(K)}\right)^\vee$.
    In particular, there is an exact sequence of topological abelian groups
    \[
      {\bigoplus}_{x \in \Delta(X)} G^{(1)}_{K_x} \to \pi^{\ab}_1(X) \to
      \pi^{\ab, \tm}_1(X) \to 0.
\]
\end{lem}

\subsection{Relation with $\pi^{\ab}_1(X)$ and
  $\pi^{\ab}_1(\ov{X})$}\label{sec:Reln}
In this paper, we shall endow a direct sum
$G = {\underset{\lambda \in I}\bigoplus} G_\lambda$ of topological abelian groups
$\{G_\lambda\}_{\lambda \in I}$ with the direct sum topology
(cf. \cite[\S~2.6]{Gupta-Krishna-REC}). The following is an easy exercise whose proof is
omitted.

\begin{lem}\label{lem:direct-sum}
  With the direct sum topology, $G$ is the coproduct of
  $\{G_\lambda\}_{\lambda \in I}$ in the category of topological abelian groups.
  The canonical homomorphism of abelian groups
   \[
     \theta_I \colon \Hom_{\Tab}({\underset{\lambda \in I}\bigoplus} G_\lambda, A) \to
     {\underset{\lambda \in I}\prod} \Hom_\Tab(G_\lambda, A)
   \]
   is an isomorphism for any topological abelian group $A$.
\end{lem}

For an integral regular curve $C$ over $k$ and integer $i \ge 0$, we let
$G^{(i)}_{\Delta(C)} = {\underset{x \in \Delta(C)}\bigoplus} G^{(i)}_{x}$,
where $G_x$ denotes the abelian absolute Galois group of
the completion $k(C)_x$ of the function field $k(C)$ at the point $x$.
Recall that a sequence of topological abelian groups 
$A' \xrightarrow{\alpha} A \xrightarrow{\beta} A''$ is {\sl topologically} exact if $A''$ is Hausdorff and
$\Ker(\beta)$ coincides with the closure
of ${\rm Image}(\alpha)$.

\begin{prop}\label{prop:TFG-3}
  There is a topologically exact sequence of topological abelian groups
  \begin{equation}\label{eqn:TFG-3-00}
    {\underset{C \in \sC(X)}\bigoplus} G^{(1)}_{\Delta(C_n)} \to \pi^{\ab}_1(X)
    \xrightarrow{\tau_X}
    \pi^{\ab, \tm}_1(X) \to 0.
  \end{equation}
\end{prop}
\begin{proof}
  We can assume $\dim(X) \ge 2$ else the lemma follows from
  \lemref{lem:Tame-fil-curve}.
  Now, it follows from Definition~\ref{defn:Tame-fil}  and the
    dimension one case that we have a complex
\begin{equation}\label{eqn:TFG-3-2}
  {\underset{C \in \sC(X)}\bigoplus} W^\tm(C_n) \xrightarrow{\tau'_X} \pi^{\ab}_1(X)
  \xrightarrow{\tau_X} \pi^{\ab, \tm}_1(X) \to 0,
\end{equation}
in which $\tau'_X$ and $\tau_X$ are continuous and there is a surjection
${\underset{C \in \sC(X)}\bigoplus}  G^{(1)}_{\Delta(C_n)} \surj 
{\underset{C \in \sC(X)}\bigoplus} W^\tm(C_n)$.
Hence, it suffices to show that ~\eqref{eqn:TFG-3-2} is topologically exact.

We let $H$ be the closure of the  image of $\tau'_X$.
We let $A$ be the image of $\Phi^\tm_X$ (cf. Definition~\ref{defn:Tame-fil})
with the quotient topology (which is discrete) so that
\cite[Lem.~7.11]{Gupta-Krishna-REC} yields a short exact sequence of continuous
  homomorphisms
  \begin{equation}\label{eqn:TFG-3-01}
0 \to A^\vee \to \pi^{\ab}_1(X) \xrightarrow{\tau_X} \pi^{\ab, \tm}_1(X) \to 0.
\end{equation}
It remains to show that the inclusion $H\inj A^\vee$ is actually an isomorphism. 

We endow $H$ with the
subspace topology induced from the profinite topology of $A^{\vee}$ so that
$H^\vee \cong \Hom_\Tab(H, {\Q}/{\Z})$.
By \cite[Lem.~7.12]{Gupta-Krishna-REC}, it suffices to show that the dual
map $A \to H^\vee$ is injective. 
Since the maps ${\underset{C \in \sC(X)}\bigoplus} W^\tm(C_n) \to H \to A^\vee$
are continuous, we have the dual homomorphisms (not necessarily continuous)
$A \to  H^\vee \to
{\underset{C \in \sC(X)}\prod} W^\tm(C_n)^\vee$ by \lemref{lem:direct-sum}.
It suffices therefore to show that the composite map is injective.
Equivalently, we need to show that the right vertical arrow in the diagram
\begin{equation}\label{eqn:TFG-3-02}
  \xymatrix@C.8pc{
    0 \ar[r] & \Fil^\tm H^1(K) \ar[r] & H^1_\et(X, {\Q}/{\Z}) \ar[r]
    \ar[d]^-{\Phi^\tm_X} &
    A \ar[d]^-{(\tau'_X)^\vee} \ar[r] & 0 \\
    & & {\underset{C \in \sC(X)}\prod}
    \frac{H^1_\et(C_n, {\Q}/{\Z})}{\Fil_{\Delta(C_n)} H^1(k(C_n))} \ar[r]^-{\cong} &
    {\underset{C \in \sC(X)}\prod} W^\tm(C_n)^\vee}
\end{equation}
is injective.
But this follows from Definition~\ref{defn:Tame-fil} because one easily
checks (by projecting to each factor of the products on the bottom) that this diagram  
is commutative.
\end{proof}

 We now let $\ov{X}$ be a regular compactification of $X$ and let
  $j \colon X \inj \ov{X}$ be the inclusion. 
  Let $J^\tm(X)$ denote the kernel of the canonical surjection
  $\tau_{\ov{X}} \colon \pi^{\ab, \tm}_1(X) \surj \pi^{\ab}_1(\ov{X})$.
  Then $J^\tm(X)$ is a closed subgroup of $\pi^{\ab, \tm}_1(X)$ and hence is a profinite
  group. The structure of $J^\tm(X)$ will play a crucial role in our
  exposition. The following result provides a partial understanding of
  this structure. We shall obtain a refined statement in a later section.
  We let $\P$ the set of prime numbers different from $p$.

\begin{prop}\label{prop:TFG-5}
There is a topologically exact sequence of topological
  abelian groups
    \begin{equation}\label{eqn:TFG-5-00}
    {\underset{C \in \sC(X)}\bigoplus} \frac{G^{(0)}_{\Delta(C_n)}}{G^{(1)}_{\Delta(C_n)}} \to
    \pi^{\ab, \tm}_1(X) \xrightarrow{\tau_{\ov{X}}}
    \pi^{\ab}_1(\ov{X}) \to 0.
  \end{equation}
\end{prop}
\begin{proof}
Using Definition~\ref{defn:Tame-fil} and repeating the proof of
  \propref{prop:TFG-3}, we get a topologically  exact sequence
  \begin{equation}\label{eqn:TFG-5-0}
    {\underset{C \in \sC(X)}\bigoplus} G^{(0)}_{\Delta(C_n)} \to \pi^{\ab}_1(X) \to
    \pi^{\ab}_1(\ov{X}) \to 0.
  \end{equation}
  We combine this with ~\eqref{eqn:TFG-3-00} to
  conclude the proof.
\end{proof}

\begin{remk}\label{remk:Top-ex}
  The proofs of Propositions~\ref{prop:TFG-3} and~\ref{prop:TFG-5} show that
  ~\eqref{eqn:TFG-3-00} and ~\eqref{eqn:TFG-5-00} are algebraically
  exact if $\dim(X) =1$. We do not expect this to be true in general although we do
  not know of a counterexample. We shall in fact show in
  \thmref{thm:Tame-nr-main} that ~\eqref{eqn:TFG-5-00} is algebraically exact
  in all dimensions if $k$ is a local field.
\end{remk}

For an abelian group $G$ and a prime $\ell$, we let ${G}_\ell =
{\varprojlim}_n \ {G}/{\ell^n}$.
Suppose now that $G$ is a profinite abelian group. Then ${G}/{m}$ is a profinite
abelian group for every integer $m$, and hence so is ${G}_\ell$ for every prime
$\ell$ under the inverse limit topology. Moreover, the completion map
$G \to {G}_\ell$ is continuous. The following result is an elementary exercise
(cf. \cite[Prop.~2.3.8]{Pro-fin}).

\begin{lem}\label{lem:Prod-dec}
  Let $G$ be a profinite abelian group. Then the canonical map
  $G \to {\underset{\ell}\prod} {G}_\ell$ is a topological isomorphism,
  where the product runs through all primes. The canonical
  map ${G}/{\ell^n} \to {G_{\ell}}/{\ell^n}$ is a topological isomorphism for every
  prime $\ell$ and integer $n \ge 1$.
\end{lem}

\begin{lem}\label{lem:Fil-p}
    For $p \ge 2$,
    the inclusion $H^1_\et(\ov{X}, {\Q}/{\Z}) \inj H^1_\et(X, {\Q}/{\Z})$ induces
    a bijection
    \[
      H^1_\et(\ov{X}, {\Q_p}/{\Z_p}) \xrightarrow{\cong} \Fil^\tm H^1(K)\{p\}.
      \]
\end{lem} 
\begin{proof}
  By ~\eqref{eqn:Tame-fil-0}, ~\eqref{eqn:Tame-fil-1} and
  Definition~\ref{defn:Tame-fil}, it
  suffices to show that, if $\dim(X) = 1$, then
  $H^1_{\et}(X, {\Q_p}/{\Z_p}) \cap \Fil_{\Delta(X)} H^1(K) =
  H^1_{\et}(\ov{X}, {\Q_p}/{\Z_p})$. 
  Equivalently,
  \begin{equation}\label{eqn:Fil-p-0}
    \Ker\left({\pi^{\ab}_1(X)}_p \surj {\pi^{\ab, \tm}_1(X)}_p\right) =
    \Ker\left({\pi^{\ab}_1(X)}_p \surj {\pi^{\ab}_1(\ov{X})}_p\right).
    \end{equation}
    This in turn is equivalent to showing that 
    ${\pi^{\ab, \tm}_1(X)}_p \to {\pi^{\ab}_1(\ov{X})}_p$ is bijective.
    To show this, it suffices to show that
    the map ${\pi^{\ab, \tm}_1(X)}/{p^n} \surj {\pi^{\ab}_1(\ov{X})}/{p^n}$ is
  bijective for all integers $n \ge 1$.
  Using \propref{prop:TFG-3} and Remark~\ref{remk:Top-ex},
  it suffices to show that $\frac{G^{(0)}_{\Delta(X)}}{G^{(1)}_{\Delta(X)}}$ is
  $p$-divisible. But the latter group is in fact uniquely $p$-divisible
  by \cite[\S~3, p.~650]{Kerz-Schmidt-MA}. 
\end{proof}

\begin{prop}\label{prop:Pi-decom}
  The canonical maps
  ${\pi^{\ab}_1(X)}_{\P} \to  {\pi^{\ab, \tm}_1(X)}_{\P}$ and
  ${\pi^{\ab, \tm}_1(X)}_{p} \to  {\pi^{\ab}_1(\ov{X})}_p$
  are topological isomorphisms.
  In particular, there is a canonical topological isomorphism between profinite groups
  \[
    \pi^{\ab, \tm}_1(X) \xrightarrow{\cong}   {\pi^{\ab}_1(X)}_{\P} \bigoplus
  {\pi^{\ab}_1(\ov{X})}_p.
\]
\end{prop}
\begin{proof}
  The $p = 1$ case of the proposition is ~\eqref{eqn:Tame-fil-1}.
  The $p \ge 2$ case follows by 
  combining ~\eqref{eqn:Tame-fil-1}, Definition~\ref{defn:Tame-fil},
  \lemref{lem:Prod-dec} and
  \lemref{lem:Fil-p}.
\end{proof}

\begin{cor}\label{cor:Pi-decom-0}
  Let $m, n \ge 1$ be integers such that $(p,m) =1$.
  Then the canonical maps
  \[
    {\pi^{\ab}_1(X)}/m \to  {\pi^{\ab, \tm}_1(X)}/m; \ \ \
{\pi^{\ab, \tm}_1(X)}/{p^n} \to  {\pi^{\ab}_1(\ov{X})}/{p^n}
  \]
  are isomorphisms of topological abelian groups.
\end{cor}
\begin{proof}
  Combine \lemref{lem:Prod-dec} and \propref{prop:Pi-decom}.
\end{proof}

\subsection{Structure of $J^\tm(X)$ when $\dim(X) = 1$}\label{sec:J-1}
In order to understand $J^\tm(X)$ via \propref{prop:TFG-5}, we need to
have some information about ${G^{(0)}_F}/{G^{(1)}_F}$ for a two-dimensional
local field $F$. \lemref{lem:Local-case} below serves this purpose.

For any commutative ring $A$, we let $K^M_*(A)$ denote the direct sum
${\underset{n \ge 0} \oplus} K^M_n(A)$ of Milnor $K$-groups of $A$
(cf. \cite[Defn.~1.3]{Kato-Saito-2}). For an ideal $I \subset A$, we let
$K^M_n(A,I) = \Ker(K^M_n(A) \to K^M_n(A/I))$. Recall that if
$F$ is a discrete valuation field, then $F^\times$ is a topological abelian group with
its `canonical topology' for which $\{1 + \fm^m_F\}_{m \ge 0}$ form a base of
open neighborhoods of 0 (where $1 + \fm^0_F := \sO^\times_F$).
In this paper, $F^\times$ will always be considered a topological group with
its canonical topology.
For any integers $m \ge 0$ and $n \ge 1$, let
$U'_mK^M_n(F)$ be the subgroup $\{1 + \fm^m_F, \sO^\times_F, \cdots , \sO^\times_F\}
\subset K^M_n(F)$. Note that $U'_0 K^M_n(F) = K^M_n(\sO_F)$. We let
$U'_{-1} K^M_n(F) = K^M_n(F)$.

\begin{lem}\label{lem:Local-case}
  Let $F$ be a complete discrete valuation field whose residue field is a local
  field $\ff$. Then the following hold.
  \begin{enumerate}
  \item
    If ${\Char}(\ff) > 0$, then ${G^{(0)}_F}/{G^{(1)}_F}$ is a finite group of order
    invertible in $\ff$.
  \item
    If ${\Char}(\ff) = 0$, then $G^{(1)}_F = 0$ and $G^{(0)}_F$ is a finite group.
    \end{enumerate}
\end{lem}
\begin{proof}
We first assume that ${\Char}(\ff)  > 0$.
  Using the Pontryagin duality, it suffices to show that
  $\left({G^{(0)}_F}/{G^{(1)}_F}\right)^\vee \cong \frac{\Fil_1 H^1(F)}{\Fil_0 H^1(F)}$
  is finite of order invertible in $\ff$.

By \cite[\S~3.1, 3.5]{Kato80}, the cup product in Galois cohomology and the
  Norm-residue map from Milnor $K$-theory to the Galois cohomology give rise
  to a pairing
  \begin{equation}\label{eqn:Local-case-0}
    K^M_2(F) \times H^1(F) \to H^3_\et(F, {\Q}/{\Z}(2)) \xrightarrow{\cong} {\Q}/{\Z}.
    \end{equation}
    %with respect to certain `Kato's topology' on $K^M_2(F)$.
    %By definition, it is the finest topology for which the product map $F^\times \times F^\times \to K^M_2(F)$ is continuous.

By the main result and Remark~4 of \cite[\S~3.5]{Kato80}, 
    ~\eqref{eqn:Local-case-0} induces homomorphisms 
    \begin{equation}\label{eqn:Local-case-1}
      H^1(F) \xrightarrow{\cong}  \Hom_\Tab(K^M_2(F), {\Q}/{\Z})
      \inj \Hom_\Ab(K^M_2(F), {\Q}/{\Z}).
    \end{equation}
    %where the first map is an isomorphism of topological abelian groups.
    Let $\psi_F$ denote the composite homomorphism.
    By \cite[Thm.~6.1, 6.3]{Gupta-Krishna-REC} and \cite[Chap.~I, Thm.~3.1]{Saito-JNT},
    it follows that $\psi_F$ induces
    a compatible system of homomorphisms of abelian groups
  \begin{equation}\label{eqn:Local-case-2}
      \psi_F \colon \Fil_n H^1(F) \to    
      \Hom_\Ab(\frac{K^M_2(F)}{U'_n K^M_2(F)}, {\Q}/{\Z}); \ \ n \ge 0.
    \end{equation}
    Taking the graded quotients, we get a homomorphism
    \begin{equation}\label{eqn:Local-case-3}
      \ov{\psi}_F \colon \frac{\Fil_{n+1} H^1(F)}{\Fil_n H^1(F)} \to    
      \Hom_\Ab(\frac{U'_n K^M_2(F)}{U'_{n+1} K^M_2(F)}, {\Q}/{\Z}); \ \ n \ge 0.
    \end{equation}
    Moreover, \cite[Thm.~6.3]{Gupta-Krishna-REC} and  \cite[Chap.~I, Thm.~3.1]{Saito-JNT} imply
    that this map is injective.
    Since $\frac{U'_0 K^M_2(F)}{U'_{1} K^M_2(F)} \xrightarrow{\cong} K^M_2(\ff)$, we
  conclude that $\frac{\Fil_{1} H^1(F)}{\Fil_0 H^1(F)}$ is a subgroup of
  $\Hom_\Ab(K^M_2(\ff), {\Q}/{\Z})$.

Next, the main result of
    \cite{Merkurjev} implies that  $K_2^M(\ff) = H_1 \oplus H_2 $, where $H_1$ is a
  uniquely divisible group and $H_2$ is the group of roots of unity in $\ff$.
  Furthermore, it is well known that $H_2 \cong {\Z}/(q-1)$, where $q$ is the
  cardinality of the residue field of $\ff$
  (cf. \cite[Prop.~5.7 (ii)]{Neukirch}). We thus get an inclusion
  \[
  \frac{\Fil_{1} H^1(F)}{\Fil_0 H^1(F)} \inj   
\Hom_\Ab(H_1, {\Q}/{\Z}) \oplus \Hom_\Ab(H_2, {\Q}/{\Z})
\cong \Hom_\Ab(H_1, {\Q}/{\Z}) \oplus {\Z}/(q-1).
\]
Since $\frac{\Fil_{1} H^1(F)}{\Fil_0 H^1(F)}$ is a torsion group and
  $\Hom_\Ab(H_1, {\Q}/{\Z})$ is uniquely divisible, it follows that
  $\frac{\Fil_{1} H^1(F)}{\Fil_0 H^1(F)} \inj {\Z}/(q-1)$.
  This finishes the proof of part (1) of the lemma.

We now assume that ${\Char}(\ff) = 0$. In this case, one knows that
  $G^{(1)}_F = 0$ (cf. \cite[Chap.~IV, Cor.~2]{Serre-LF}). 
  To understand the structure of $G^{(0)}_F$, we use the short exact sequence
  \begin{equation}\label{eqn:Local-case-4}
    0 \to H^1(\ff) \to H^1(F) \to H^0(\ff) \to 0
    \end{equation}
    of discrete abelian groups by \cite[\S~3, Lem.~2]{Kato80}.
    This sequence implies that $G^{(0)}_F = \Ker(G_F \surj G_\ff)
    = (H^0(\ff))^\vee$.
On the other hand, we have $H^0(\ff) = H^0_\et(\ff, {\Q}/{\Z}(-1)) \cong
    \Hom_\Ab(\mu_\infty(\ff), {\Q}/{\Z})$, where $\mu_\infty(\ff)$ is the group of roots of
    unity in $\ff$. Since $\mu_\infty(\ff)$ is a finite
    group (cf. \cite[Prop.~5.7(i)]{Neukirch}),
    it follows that $G^{(0)}_F$ is finite. 
\end{proof}

\begin{cor}\label{cor:Tame-closed}
   Let $X$ be as in \propref{prop:TFG-5} with $\dim(X) =1$. Then $J^\tm(X)$ is
  a finite group of order prime to the 
  characteristic exponent of $k$.
\end{cor}
\begin{proof}
 Combine Remark ~\ref{remk:Top-ex} and 
 \lemref{lem:Local-case}.
 \end{proof}

\section{The tame class group}\label{sec:TCG}
 In this section, we shall introduce the tame class group 
 and prove its functorial properties. These properties will play key roles in our
 proofs. We fix an arbitrary infinite field $k$ throughout this section.

\subsection{Lower Milnor $K$-theory of semilocal rings}\label{sec:SLK}
  For a Noetherian scheme $X$, we let $K_j(X)$ denote the
Quillen $K$-groups of the exact category of locally free sheaves on $X$.
We shall usually write $K_j(X) = K_j(A)$ when $X = \Spec(A)$ is an affine scheme.
We let $\sK^M_{j,X}$ denote the sheaf of $j$-th Milnor $K$-theory on $X$ with
respect to Zariski (or Nisnevich) topology (cf. \cite[Defn.~1]{Kerz-JAG}).

\begin{lem}\label{lem:Semilocal-K}
  Let $A$ be a Noetherian semilocal ring containing $k$. Then the canonical map
  $K^M_j(A) \to K_j(A)$ is an isomorphism for $j \le 2$.
\end{lem}
\begin{proof}
  The lemma is trivial for $j =0$ and follows from \cite[Chap.~III, Lem.~1.4]{K-book}
  for $j =1$. To prove it for $j =2$, we consider the maps
  \begin{equation}\label{eqn:Semilocal-K-0}
    K^M_2(A) \to K_2(A) \to \frac{H_2(GL(A), \Z)}{H_2(GL_1(A), \Z)},
  \end{equation}
  where the right arrow is the 2nd Hurewicz map for $BGL^+(A)$. A theorem of
  Suslin (which was shown for semilocal rings by Levine in
  \cite[\S~3, Thm.~1]{Levine-K})
  says that the composite map is bijective. On the other hand,
  the Hurewicz map is an isomorphism by \cite[Chap.~III, Lem.~1.4]{K-book} and
  \cite[Lem.~4.2]{Nest-Suslin}. The lemma follows.
  \end{proof}

\begin{lem}\label{lem:Semilocal-K-1}
Let $A$ be the semilocal ring at the closed points
  $\{x_1, \ldots , x_r\}$ on a regular integral curve over $k$. Let $F$ denote the
  quotient field of $A$. Then the Gersten sequence
  \[
    0 \to  K^M_j(A) \to K^M_j(F) \xrightarrow{\partial_F} \
    \stackrel{r}{\underset{i=1}\bigoplus} K^M_{j-1}(k(x_i)) \to 0
  \]
  is exact for every $j \le 2$.
  \end{lem}
  \begin{proof}
   Combine \lemref{lem:Semilocal-K} and
    \cite[Thm.~5.11]{Quillen} and \cite[Lem.~A1]{Gupta-Krishna-AKT-1}.
\end{proof}

\begin{cor}\label{cor:Semilocal-K-2}
  Let $A$ be as in \lemref{lem:Semilocal-K-1} and let $I \subset A$ be a nonzero
  radical ideal.
  Let $A_i$ be the localization of $A$ at $x_i$ and let
  $\fm_i$ be the maximal ideal of $A_i$ for $1 \le i \le r$.
  Then we have the following relations between the subgroups of $K^M_j(F)$ for
  $j \le 2$.
  \[
    K^M_j(A) = \ \stackrel{r}{\underset{i =1}\bigcap} K^M_j(A_i) \ \mbox{and} \
    K^M_j(A,I) = \ \stackrel{r}{\underset{i =1}\bigcap} K^M_j(A_i, IA_i).
  \]
\end{cor}
\begin{proof}
  The first equality follows directly by comparing the exact sequence of
  \lemref{lem:Semilocal-K-1} for $A$ with similar exact sequences for $A_i$'s.
  To prove the other equality, note that it already follows from
  \lemref{lem:Semilocal-K-1} that $K^M_j(A,I)
  \subseteq \ \stackrel{r}{\underset{i =1}\bigcap} K^M_j(A_i, IA_i)$.
  It is also clear that $\stackrel{r}{\underset{i =1}\bigcap} K^M_j(A_i, IA_i)
  \subset  \ \stackrel{r}{\underset{i =1}\bigcap} K^M_j(A_i) = K^M_j(A)$.
  To finish the proof, we only need to observe that the canonical map
  $A/I \to \stackrel{r}{\underset{i =1}\prod} {A_i}/{IA_i}$ is a ring
  isomorphism.
\end{proof}

\begin{lem}\label{lem:K2-dvr}
  Let $A$ be a discrete valuation ring containing an infinite field
  whose quotient field is $F$ and the maximal ideal is $\fm$.
  Let $\wh{A}$ denote the $\fm$-adic completion of $A$ and $\wh{F}$ the fraction field
  of $\wh{A}$.  Let $I \subset \fm$
  be a nonzero ideal. Then the commutative squares
  \begin{equation}\label{eqn:K2-dvr-0}
    \xymatrix@C.8pc{
      K^M_j(A) \ar[r] \ar[d] & K^M_j(F) \ar[d] & & K^M_j(A, I) \ar[r] \ar[d] & K^M_j(A)
      \ar[d] \\
K^M_j(\wh{A}) \ar[r] & K^M_j(\wh{F}) & & K^M_j(\wh{A}, I\wh{A}) \ar[r] & K^M_j(\wh{A})}
\end{equation}
are Cartesian for every $j \le 2$.
\end{lem}
\begin{proof}
Follows from the Gersten resolution for Quillen $K$-theory (cf. \cite{Quillen}).
\end{proof}

Let $C$ be an integral regular curve over $k$. 
For an effective Cartier divisor $D \subset {C}$ and a closed point
$x \in  {C}$, we let $\wh{I_{D,x}} \subset \wh{\sO_{{C},x}}$ be the 
ideal of the scheme theoretic inverse image of $D$ under the
canonical composite map $\Spec(\wh{\sO_{C,x}}) \to \Spec({\sO_{{C},x}}) \to C$.
We let $I_{D,x}$ be the extension of the ideal of $D$ in $\sO_{C,x}$.
Consider the map $K^M_j(k(C)) \to 
        {\underset{x \in D}\bigoplus} \left(K^M_{j-1}(k(x)) \oplus K^M_j(k(x))\right)$,
        whose $x$-component is defined by $a \mapsto (\partial_x(a),
        \partial_x(\{\pi_x\}  \cup a))$ for a fixed choice of uniformizing parameter
        $\pi_x$ at $x$. Let $A_D := \sO_{C,D}$.
     
\begin{lem}\label{lem:Fil-Milnor-comp}
Let $I_D \subset A_D$ be the ideal of $D$ and let $j \le 2$.
    Then the canonical maps
    \[
      \begin{array}{lll}
        K^M_j(A_D, I_D) 
        & \to &
        \Ker \left(K^M_j(k(C)) \to
    {\underset{x \in D}\bigoplus}
   \frac{K^M_j(k(C))}{K^M_j({\sO_{{C},x}}, I_{D,x})}\right) \\
   & \to &      
\Ker \left(K^M_j(k(C)) \to
    {\underset{x \in D}\bigoplus}
           \frac{K^M_j(k(C)_x)}{K^M_j(\wh{\sO_{{C},x}}, \wh{I_{D.x}})}\right)
      \end{array}
    \]
    are isomorphisms. If $D$ is reduced, then we also have
    \[
K^M_j(A_D, I_D) \xrightarrow{\cong} \Ker \left[K^M_j(k(C)) \to
  {\underset{x \in D}\bigoplus} \left(K^M_{j-1}(k(x)) \oplus K^M_j(k(x))\right)\right].
\]
\end{lem}
\begin{proof}
  The first part follows directly by combining
  \corref{cor:Semilocal-K-2} and \lemref{lem:K2-dvr}.
  For the second part, we can use
  $\Ker \left(K^M_j(k(C)) \to {\underset{x \in D}\bigoplus}
    \frac{K^M_j(k(C))}{K^M_j({\sO_{{C},x}}, I_{D,x})}\right)$ instead of
  $ K^M_j(A_D, I_D)$. The desired isomorphism is then
   an easy exercise using the Gersten sequence. 
\end{proof}

For an effective Cartier divisor $D \subset {C}$ and integer $j \ge 0$, we let
\begin{equation}\label{eqn:Fil-Miln}
  \Fil_D K^M_j(k(C)) = \Ker \left(K^M_j(k(C)) \to
    {\underset{x \in D}\bigoplus}
    \frac{K^M_j(k(C)_x)}{K^M_j(\wh{\sO_{{C},x}}, \wh{I_{D,x}})}\right).
  \end{equation}

  It follows from \lemref{lem:Fil-Milnor-comp} that
  $\Fil_D K^M_j(k(C)) = K^M_j(A_D, I_D)$ if $j \le 2$.

\subsection{Tame class group}
 \label{sec:IG}
Let $X \in \Sch_k$ be an integral normal scheme.
For $C \in \sC(X), \ x \in C_{(0)}$ and $j \ge 0$, there is a boundary homomorphism
$\partial_{C,x} \colon K^M_{j+1}(k(C)) \to K^M_j(k(x))$ (cf. \cite{Kato86})
which is defined to be the composite map
$K^M_{j+1}(k(C)) \cong K^M_{j+1}(k(C_n)) \xrightarrow{\partial} {\underset{y \in
    \pi^{-1}(x)}
  \oplus} K^M_j(k(y)) \xrightarrow{N}  K^M_j(k(x))$, where $\pi \colon C_n \to C$
is the normalization map, $\partial$ is the
classical tame symbol and $N$ is the sum of norm maps.
Since $\partial_{C,x}(\alpha) = 0$ for almost all $x$ (depending on $\alpha$), we get
a map $\partial_C \colon K^M_{j+1}(k(C)) \to {\underset{x \in X_{(0)}}\oplus} 
K^M_j(k(x))$.
Taking the sum over all $C \in \sC(X)$, we get a boundary homomorphism
$\partial^j_X \colon {\underset{C \in \sC(X)}\oplus} K^M_{j+1}(k(C)) \to
{\underset{x \in X_{(0)}}\oplus} K^M_j(k(x))$. 
The following definition is due to Yamazaki \cite[Defn.~1.1]{Yamazaki}.

\begin{defn}\label{defn:TCG-0}
  The tame class group of $X$ is defined to be the cokernel of the boundary homomorphism
  \[
    \partial^1_X \colon {\underset{C \in \sC(X)}\bigoplus} \Fil_{\Delta(C_n)} K^M_{2}(k(C))
      \longrightarrow
 {\underset{x \in X_{(0)}}\bigoplus} K^M_1(k(x))
\]
and it is denoted by $C^\tm(X)$. 
\end{defn}

\begin{remk}\label{remk:Yamazaki*}
  Yamazaki uses a group $UK^M_2(k(C))$ instead of $ \Fil_{\Delta(C_n)} K^M_{2}(k(C))$.
  But the two groups coincide by \lemref{lem:Fil-Milnor-comp}.
  \end{remk}

In the sequel, we shall use the following notations:
\begin{equation}\label{eqn:Gen-Reln}
F^0_1(X) = {\underset{x \in X_{(0)}}\bigoplus} K^M_1(k(x)), \ F^1_2(X) = 
{\underset{C \in \sC(X)}\bigoplus} \Fil_{\Delta(C_n)} K^M_{2}(k(C_n)).
\end{equation}

\subsection{Contravariance of tame class group}\label{sec:Tame-functorial*}
Let $k$ be an infinite field of characteristic exponent $p \ge 2$
and $\ov{X}$ a complete integral normal $k$-scheme. Let
 ${k'}/k$ be a purely inseparable algebraic extension of fields. Suppose that
there is an increasing sequence of finite field extensions ${k_i}/k$ inside $k'$ such
that $k = k_0$ and $k' = \cup_i k_i$. For any $i \ge 1$, we let
$\ov{X}_i$ be the normalization of $(\ov{X}_{k_i})_\red$ and  let 
%$g_i \colon \ov{X}_{k_i}\to \ov{X}$ and 
$f_i \colon \ov{X}_i \to \ov{X}$ denote the projection morphism.
%We shall denote the restriction of $g_i$ to $(\ov{X}_{k_i})_\red$ also by $g_i$.
We let $\ov{X}'$ denote the normalization of $(\ov{X}_{k'})_\red$
%Let $g \colon \ov{X}_{k'} \to \ov{X}$ and
and let $f \colon \ov{X}' \to \ov{X}$ be the
projection morphism. We fix a nonempty open immersion $u \colon X \inj \ov{X}$.
We define $X_i$ and $X'$ analogous to the definitions of $\ov{X}_i$ and $\ov{X}'$.
Let $f'_i \colon \ov{X}' \to \ov{X}_i$ be the projection morphism.
The following lemma is an elementary exercise. We refer to \cite[Tag~01YV]{SP}
for a proof.

\begin{lem}\label{lem:Inv-lim}
  The inverse system of schemes $\{\ov{X}_{k_i}\}_{i \ge 0}$ satisfies the following.
  \begin{enumerate}
  \item
    $\{\ov{X}_{k_i}\}$ is an inverse system of complete irreducible 
    $k$-schemes with finite transition maps whose inverse limit is 
    the irreducible $k'$-scheme $\ov{X}_{k'}$.
  \item
    $\{(\ov{X}_{k_i})_\red\}$ is an inverse system of complete
    integral $k$-schemes with finite
    transition maps whose inverse limit is the integral $k'$-scheme
    $(\ov{X}_{k'})_\red$.
  \item
    $\{\ov{X}_i\}$ is an inverse system of complete integral normal $k$-schemes with
    finite transition maps whose inverse limit is the $k'$-scheme $\ov{X}'$.
  \item
    The canonical map $X' \to \ov{X}' \times_{\ov{X}} X$ is an isomorphism.
   \end{enumerate}
\end{lem}

Let $f \colon \ov{X}' \to \ov{X}$ be as above. Assume that the transition map
$f_{ij} \colon X_j \to X_i$ is flat for each $j \ge i \ge 0$. 
For $x \in X_{(0)}$, let $S_x$ be the scheme theoretic
inverse image of $x$ under $f$ and let $S'_x$ be the set theoretic inverse image
(with the reduced subscheme structure). Let $\iota_x \colon S'_x \inj S_x$ be the
inclusion and $f_x \colon S_x \to \Spec(k(x))$ be the projection.
Our assumption implies that $S_x$ is a 0-dimensional $k'$-scheme. In particular,
$S'_x = \{x_1, \ldots , x_s\}$ is a finite set. We let
$S_x = \stackrel{s}{\underset{i = 1}\coprod} S^i_x$, where $(S^i_x)_\red =
\Spec(k'(x_i))$ for $1 \le i \le s$.

We have a pull-back map
$f^*_x \colon K_n(k(x)) \to G_n(S_x)$ and the push-forward map
$\iota_{x *} \colon K_n(S'_x) \xrightarrow{\cong} G_n(S_x)$.
We let $f^* \colon K_n(k(x)) \to K_n(S'_x) \cong \
\stackrel{s}{\underset{i =1}\bigoplus} K_n(k'(x_i))$ be the map $\iota^{-1}_{x *} \circ
f^*_x$.  For $n = 0$, a straightforward application  of the Devissage theorem
shows that $f^*_x(1) = f^*_x([k(x)]) = [\sO_{S_x}] =
\stackrel{s}{\underset{i =1}\sum} \ell({S^i_x})[k'(x_i)]=
 \stackrel{s}{\underset{i =1}\sum} \ell({S^i_x}) \iota_{x*} ([k'(x_i)]) = 
 \iota_{x*}( \stackrel{s}{\underset{i =1}\sum} \ell({S^i_x})  f_{x, i}^*(1))
 $, where $f_{x,i} \colon \Spec(k'(x_i)) \to \Spec(k(x))$ is the projection.
 Here, for an Artinian local ring $R$, we let $\ell(\Spec(R))$ denote the length of
 $R$. 

As $f^*_x, f^*_{x,i}$ and $\iota_{x *}$ are $K_*(k(x))$-linear, it follows
that for $n \ge 0$ and $\alpha \in K_n(k(x))$, one has
$f^*_x(\alpha) =  f^*_x(\alpha \cdot 1) = \alpha \cdot f^*_x(1) 
=  \iota_{x*}(\stackrel{s}{\underset{i =1}\sum} \ell({S^i_x}) f^*_{x,i}(\alpha))$.
We conclude that
\begin{equation}\label{eqn:Defn-PB}
  f^* = (\ell({S^1_x}) f^*_{x,1}, \cdots ,  \ell({S^s_x}) f^*_{x,s}).
  \end{equation}
Taking $n = 1$ and summing over $x \in X_{(0)}$, we get a pull-back map
$f^* \colon F^0_1(X) \to F^0_1(X')$.
%$f^* \colon {\underset{x \in X_{(0)}}\bigoplus} K^M_1(k(x)) \to {\underset{x \in X'_{(0)}}\bigoplus} K^M_1(k(x))$.

\begin{prop}\label{prop:TCG-PB}
  $f^*$ descends to a homomorphism $f^* \colon C^\tm(X) \to C^\tm(X')$
 such that $f^* = f'^*_i \circ f^*_i$ for every $i \ge 0$.
 %Moreover, the map $f^*_i \colon C^\tm(X) \to C^\tm(X_i)$ is continuous when $k$ is a local field. 
 \end{prop}
\begin{proof}
We let $C \in \sC(X)$ and let $\ov{C}$ be the closure of $C$ in $\ov{X}$.
We let $\nu \colon \ov{C}_n \to \ov{X}$ be the canonical map from the normalization
of $\ov{C}$. We let $\ov{C}' = \ov{C} \times_{\ov{X}} \ov{X}'$ and let 
  $\ov{C}'_1, \ldots , \ov{C}'_r$ be the irreducible components of $\ov{C}'$.
Since $\ov{X}'$ is an inverse limit of integral schemes with finite transition maps,
  it is an easy exercise to check that each $C'_i \subset \ov{C}'_i$ is dense, if
  we let $C'_i = \ov{C}'_i \cap X'$. We set
  $\ov{Y} = \ov{C}_n \times_{\ov{X}} \ov{X}' \cong \ov{C}_n \times_{\ov{C}} \ov{C}'$ and
  let $\wt{g} \colon \ov{Y} \to \ov{C}_n$ and $\nu' \colon \ov{Y} \to \ov{X}'$ be the
  projections. Let $\nu'_i \colon \ov{Y}_i := (\ov{C}'_i)_n \to \ov{X}'$ and
  $\wt{g}_i \colon \ov{Y}_i \to \ov{C}_n$ be the canonical maps. It is clear that
  each $\wt{g}_i$ is flat. We let
  $Z = \coprod_i \ov{Y}_i =  (\ov{C}'_\red)_n \xleftarrow{\cong} (\ov{Y}_\red)_n$.
The last isomorphism holds because $\ov{Y} \to \ov{C}'$ is a finite birational
morphism of curves over $k'$.

We let $W$ (resp. $W_i$) be the spectrum of the semilocal ring
$\sO_{\ov{C}_n, T}$ (resp. $\sO_{\ov{Y}_i, T_i}$), where $T = \nu^{-1}(\ov{X} \setminus X)$ and
$T_i = \nu'^{-1}_i(\ov{X}' \setminus X') = \wt{g}^{-1}_i(T)$.
We let $W' = \Spec(k(C)) \times_{\ov{C}_n} \ov{Y}$ so
that $\Spec(k(Z)) \xrightarrow{\cong} W'_\red$ via the map $Z \to \ov{Y}$.
Let $I \subset \sO_{\ov{C}_n, T}$ be the ideal of $T$ and $I_i \subset
\sO_{\ov{Y}_i, T_i}$ the ideal of $T_i$. Note that $W'_\red = W' \times_{\ov{Y}} \ov{Y}_\red$
because the inclusion $W' \inj \ov{Y}$ is a limit of open immersions.
We write $W' = \stackrel{r}{\underset{i =1}\coprod} W'_i$, where $(W'_i)_\red =
\Spec(k'(\ov{C}'_i))$. We let $\pi_i \colon \Spec(k'(\ov{C}'_i)) \inj W'$ be the
inclusion. Let $g'$ (resp. $g'_i$) be the restriction of $\wt{g}$ to
$W'_\red$ (resp. $\Spec(k'(\ov{C}'_i))$).

We fix $\alpha \in K^M_2(\sO_{\ov{C}_n, T}, I)$ and let $\Sigma \subset C_n$ be the
smallest 0-dimensional reduced closed subscheme such that $\partial(\alpha) \in
K_1(\Sigma) \subset F^0_1(C_n)$. Let $S = \nu(\Sigma)$ and $S' = S \times_{X} X'$. We let
$\Sigma' = \Sigma \times_{C_n} \nu'^{-1}(X')$. We let 
$\pi \colon W'_\red \inj W', \ \tau' \colon \Sigma'_\red \inj \Sigma'$ and
$\tau \colon S'_\red \inj S'$ be the inclusion maps.
To show that $f^* \colon F^0_1(X) \to F^0_1(X')$ takes $\partial^1_X(F^1_2(X))$
into $\partial^1_{X'}(F^1_2(X'))$, we need to show that
$\tau^{-1}_* \circ f^*_S \circ \nu_*(\partial(\alpha))$ lies in the image of
$\partial^1_{X'} \colon F^1_2(X') \to F^0_1(X')$, where $f_S \colon S' \to S$ is
the restriction of $f$ to $S'$.

To show the above, we consider the diagram
\begin{equation}\label{eqn:TCG-PB-1}
  \xymatrix@C.6pc{
    \stackrel{r}{\underset{i =1}\bigoplus} K_2(W_i, T_i) \ar@{->>}[r]
    \ar@{}[dr] | {\refone} &
 \stackrel{r}{\underset{i =1}\bigoplus} K^M_2(\sO_{\ov{Y}_i, T_i}, I_i) \ar@{^{(}->}[r]
 \ar@{}[dr] | {\reftwo} &
    \stackrel{r}{\underset{i =1}\bigoplus} K^M_2(k(\ov{C}'_i))
    \ar[r]^-{\cong} \ar@{}[dr] | {\refthree} &
    K_2(W'_\red)  \ar[d]^-{\sum_i \ell({W'_i}) \pi_{i *}} & \\
    K_2(W, T) \ar@{->>}[r] \ar[u]^-{g'^*} &
    K^M_2(\sO_{\ov{C}_n, T}, I) \ar@{^{(}->}[r]
    \ar[u]^-{\oplus_i g'^*_i} \ar[dr]_-{\partial}
    & K^M_2(k(C)) \ar[r]^-{\wt{g}^*} \ar[d]^-{\partial} \ar[u]^-{\oplus_i g'^*_i}
    \ar@{}[dr] | {\reffour}
& G_2(W') \ar[d]^-{\partial} \ar@{}[dr] | {\reffive} &
K_2(W'_\red) \ar[l]_-{\pi_*}^-{\cong} \ar[d]^-{\partial} \\
& & K_1(\Sigma) \ar[r]^-{\wt{g}^*} \ar[d]^-{\nu_*} \ar@{}[dr] | {\refsix} & G_1(\Sigma')
\ar[d]^-{\nu'_*} \ar@{}[dr] | {\refseven} &
K_1(\Sigma'_\red) \ar[l]_-{\tau'_*}^-{\cong} \ar[d]^-{\nu'_*} \\
& & K_1(S) \ar[r]^-{f^*_S} & G_1(S') & K_1(S'_\red) \ar[l]_-{\tau_*}^-{\cong}.}
\end{equation}

\enlargethispage{2pt}

The squares (1) and (2) in this diagram commute by the contravariant functoriality of
$K$-theory. The computations of ~\eqref{eqn:Defn-PB} show that (3) commutes.
The squares (4)  and (5) commute because the localization sequence in $G$-theory
commutes
with the flat pull-back and proper push-forward. Note that all horizontal arrows
in (3), (4) and (6) are flat pull-back maps (because $\Sigma$ and $S$ are spectra of
products of fields). The square (6) commutes because the flat pull-back
commutes with the proper push-forward in $G$-theory induced by
a Cartesian square of schemes (cf. \cite[Prop.~2.11]{Quillen}). The square (7) clearly
commutes because all its arrows are the proper push-forward maps in $G$-theory.
The lone triangle clearly commutes. The horizontal arrows in (1) are surjective by
\lemref{lem:Semilocal-K}.

It follows from the commutativity of (6) that
$f^*_S \circ \nu_*(\partial(\alpha)) = \nu'_* \circ \wt{g}^*(\partial(\alpha))$.
On the other hand, the commutativity of other squares in ~\eqref{eqn:TCG-PB-1} shows
that $\nu'_* \circ \wt{g}^*(\partial(\alpha)) = \nu'_* \circ \partial \circ
\wt{g}^*(\alpha) =
   \nu'_* \circ \partial (\sum_i \pi_{i *}(\ell({W'_i}) g'^*_i(\alpha)))$.
   Letting $\beta_i = \ell({W'_i}) g'^*_i(\alpha) \in  K^M_2(\sO_{\ov{Y}_i, T_i}, I_i) \subset
   K^M_2(k(\ov{C}'_i))$, we conclude that
   %\begin{equation}\label{eqn:TCG-PB-2}
     $f^*_S \circ \nu_*(\partial(\alpha)) = \sum_i  \nu'_* \circ \partial \circ \pi_{i *}
     (\beta_i)$.
     %\end{equation}

     Hence, we get
\[
  \begin{array}{lllll}
\tau^{-1}_* \circ  f^*_S \circ \nu_*(\partial(\alpha)) & = &
    \sum_i   \tau^{-1}_* \circ \nu'_* \circ \partial \circ \pi_{i *}(\beta_i)
    & = & \sum_i \nu'_* \circ \tau'^{-1}_*  \circ \partial \circ \pi_{i *}(\beta_i)  \\
     & = & \sum_i \nu'_* \circ \partial \circ \pi^{-1}_* \circ \pi_{i *}(\beta_i)  
    & = & \sum_i \nu'_* \circ \partial (\beta_i) \\      
& = &   \sum_i (\nu'_i)_* \circ \partial (\beta_i). 
\end{array}
  \]   
Since $\beta_i \in \Fil_{\Delta((C'_i)_n)} K^M_2(k(C'_i)) \subset F^1_2(X')$ and
  $(\nu'_i)_* \circ \partial (\beta_i)$ lies in the image of the map
  $\partial^1_{X'} \colon F^1_2(X') \to F^0_1(X')$ for every $1 \le i \le r$, we are done.

To show that $f^* = f'^*_i \circ f^*_i$ for each $i$, one easily reduces to checking it
on $K^M_1(k(x))$ for a closed point $x \in X$. In the latter case, we have to only
check that for a fixed $i \ge 1$, a closed point $x_i \in f^{-1}_i(x)$ and
a closed point $x' \in f'^{-1}_i(x_i)$, one has $\ell(S_x(x')) =
\ell(W_x(x_i)) \ell(T_{x_i}(x'))$, where $S_x(x')$ is the factor of the scheme theoretic
inverse image $S_x$ of $\Spec(k(x))$ in $X'$ whose support is $\{x'\}$, 
$W_x(x_i)$ is the factor of the scheme theoretic
inverse image $W_x$ of $\Spec(k(x))$ in $X_i$ whose support is $\{x_i\}$ and
$T_{x_i}(x')$ is the factor of the scheme theoretic
inverse image $T_{x_i}$ of $\Spec(k_i(x_i))$ in $X'$ whose support is $\{x'\}$.
But this follows from \lemref{lem:length}.
\end{proof}

\begin{lem}\label{lem:length}
    Let $A_1 \to A_2$ be a flat homomorphism of
    Artinian local algebras over a field such that $J_1A_2 \subset J_2$, where $J_i$ 
is the maximal ideal of $A_i$. Let $A'_2 = {A_2}/{J_1A_2}$. Then
    $\ell(A_2) = \ell(A_1) \ell(A'_2)$.
  \end{lem}
  \begin{proof}
The lemma follows from  \cite[Ex.~A.1.1]{Fulton}. We omit the details. 
\end{proof}

  We now let $k$ be an infinite field of characteristic exponent $p \ge 2$ and let
  $X$ be a smooth and integral $k$-scheme. Let ${k'}/k$ be a purely inseparable
  algebraic extension of fields. If we let $u \colon X \inj \ov{X}$ be
  the inclusion of $X$ into any of its integral normal compactifications
  (which always exist), then the map
  $f^* \colon F^0_1(X) \to F^0_1(X')$ as defined in ~\eqref{eqn:Defn-PB}
  does not depend on the choice of $\ov{X}$. This is because of the fact that
  under the smoothness condition on $X$, one has that $X_i = X_{k_i}$ for
  every $i \ge 0$ and $X' = X_{k'}$ in the notations used above. Furthermore,
   $f_{ij} \colon X_j \to X_i$ is finite and flat.
\propref{prop:TCG-PB} therefore implies the following.

 \begin{cor}\label{cor:Sm-PB-insep}
    For $X$ and ${k'}/k$ as above, the map $f^* \colon F^0_1(X) \to F^0_1(X_{k'})$ is
    well-defined and induces a pull-back map $f^* \colon C^\tm(X) \to C^\tm(X_{k'})$.
    \end{cor}

 We next note that if $k$ is any infinite field
 and $f \colon X'\to X$ is a finite and flat morphism of
 integral normal $k$-schemes, then the discussion before \propref{prop:TCG-PB} holds
 for $f$, and ~\eqref{eqn:Defn-PB} defines a map $f^* \colon F^0_1(X) \to F^0_1(X')$
 without the need for compactifications.
 Indeed, on $x \in X_{(0)}$-component, the map $f^* $ is defined as follows. 
\[
 f^* = (\ell({S^1_x}) f^*_{x,1}, \cdots ,  \ell({S^s_x}) f^*_{x,s}),
\]
where the scheme-theoretic inverse image is
$f^{-1}(x) = \stackrel{s}{\underset{i = 1}\coprod} S^i_x$ with
$\Spec(k(x_i)) = (S^i_x)_{{\rm red}} $ and $f_{x,i} \colon \Spec(k(x_i)) \to \Spec(k(x))$
the projection. 
In the proof of \propref{prop:TCG-PB}, we have in fact shown the following.

\begin{prop}\label{prop:TCG-PF-fin}
  Let $f \colon X' \to X$ be a finite and flat map between integral normal $k$-schemes.
  Then the above map $f^* \colon F^0_1(X) \to F^0_1(X')$ descends to a homomorphism
  \[
    f^* \colon C^\tm(X) \to C^\tm(X').
  \]
  \end{prop}

\subsection{Covariance of tame class group}\label{sec:Tame-functorial}
Let $k$ be any infinite field.
To prove the covariance property of the tame class group, we shall need the
  following local result.

\begin{lem}\label{lem:Norm-rel}
  Let $A$ be the local ring at a closed point of an integral regular curve over $k$
  and let $f \colon A \inj A'$ be a finite morphism of semilocal Dedekind domains.
  Let $J$ (resp. $J'$) denote the Jacobson radical of $A$ (resp. $A'$).
  Let $F$ (resp. $F'$) denote the quotient field of $A$ (resp. $A'$). Then the norm
  map $f_* \colon K^M_2(F') \to K^M_2(F)$ restricts to
  \[
    f_* \colon K^M_2(A') \to K^M_2(A)  \ \ \mbox{and} \ \
    f_* \colon K^M_2(A', J') \to K^M_2(A,J).
  \]
\end{lem}
\begin{proof}
 We can replace $K^M_2(-)$ by $K_2(-)$ using
  \lemref{lem:Semilocal-K}. 
 It follows from \cite[Lem.~3.4]{Gupta-Krishna-REC} that $f_*$ induces the norm map
 $f_* \colon K_2(A') \to K_2(A)$. Using Lemmas~\ref{lem:Semilocal-K}, 
 ~\ref{lem:Semilocal-K-1} and \cite[Lem.~11.3]{Gupta-Krishna-AKT-1},
 we see that these norm maps coincide with the push-forward maps between
 Quillen $K$-groups associated to the finite morphisms of rings and fields.

 To show that $f_*$ preserves the relative $K$-groups,
 we let $\wh{A}$ and $\wh{A'}$ be the completions with respect to the Jacobson
 ideals. Let $\ff$ be the residue field of $A$ and $\ff_1, \ldots , \ff_r$ the
 residue fields of $A'$. Let $A'_i$ be the localization of $A'$ at the maximal
 ideal corresponding to the residue field $\ff_i$. Let $e_i$ denote the ramification
 index of the extension $A \to A'_i$. 

 Since $J' = \sqrt{JA'}$, it follows that $\wh{A'}$ is the $J$-adic completion
 of $A'$ when treated as an $A$-module. Since $A'$ is finite over $A$, it
 follows from \cite[Thm.~8.7, 8.15]{Matsumura} that the maps
 $\wh{A} \otimes_A A' \xrightarrow{\lambda_1} \wh{A'} \xrightarrow{\lambda_2}
 \stackrel{r}{\underset{i =1}\prod} \wh{A'_i}$
 are isomorphisms, where $\wh{A'_i}$ is the completion of $A'_i$ with respect to its
 maximal ideal. Let $\wh{f}_i \colon \wh{A} \inj \wh{A'_i}$ and
 $f'_i \colon \ff \inj \ff_i$ be the inclusions.

 We now consider the diagram
 \begin{equation}\label{eqn:Norm-rel-1}
   \xymatrix@C1pc{
     K_j(A') \ar[d]_-{f_*} \ar[r] & \stackrel{r}{\underset{i =1}\oplus}
     K_j(\wh{A'_i}) \ar[d]^-{\wh{f}_*} \ar[r] & \stackrel{r}{\underset{i =1}\oplus}
     K_j(\ff_i) \ar[d]^-{\sum_i e_i (f'_i)_*} \\
     K_j(A) \ar[r] & K_j(\wh{A}) \ar[r] & K_j(\ff)}
 \end{equation}
 of $K_*(A)$-modules for $0 \le j \le 2$, in which the horizontal arrows are the
 canonical pull-back maps in $K$-theory. To finish the proof of the lemma,
 it suffices to show that the outer rectangle in ~\eqref{eqn:Norm-rel-1} is
 commutative. Using the isomorphisms $\lambda_1$ and $\lambda_2$ 
 between the complete rings, it
 follows from \cite[\S~7, Prop.~2.11]{Quillen} that the left square is commutative.
 It remains to show that the right square in ~\eqref{eqn:Norm-rel-1} is commutative.
 But this follows from \cite[\S~3, Lem.~13]{Kato80} 
since the middle vertical arrow is the same as the map ${\sum_i (\wh{f}_i)_*}$.
\end{proof}

Let $f \colon X' \to X$ be a morphism of integral normal $k$-schemes.
For every $x \in X'_{(0)}$, we have the norm map $f_* \colon K^M_1(k(x)) \to
K^M_1(k(f(x)))$. This induces a homomorphism $f_* \colon F^0_1(X')  \to F^0_1(X)$.

\begin{prop}\label{prop:TCG-PF}
  $f_*$ descends to a push-forward homomorphism
  \[
    f_* \colon C^\tm(X') \to C^\tm(X).
  \]
\end{prop}
\begin{proof}
  The proof reduces to showing that if $f \colon X' \to X$ is a finite morphism
  of integral normal curves, and  $\alpha \in \Fil_{\Delta(X')} K^M_2(k(X'))$
  (where $f_* \colon K^M_2(k(X')) \to K^M_2(k(X))$ is the norm map), then
$f_*(\alpha) \in \Fil_{\Delta(X)} K^M_2(k(X))$. But this follows from
Lemmas~\ref{lem:Fil-Milnor-comp} and ~\ref{lem:Norm-rel}.
\end{proof}

\begin{cor}\label{cor:Sm-PB}
  Let $X$ be an integral smooth $k$-scheme and let $f \colon X' := X_{k'} \to X$ be
  the induced morphism where ${k'}/k$ is a finite field
  extension. Then there exists a pull-back map $f^* \colon C^\tm(X) \to C^\tm(X')$
  such that $f_* \circ f^*$ is multiplication by $[k':k]$ on $C^\tm(X)$.
  \end{cor}
  \begin{proof}
    The existence of $f^*$ follows from \propref{prop:TCG-PF-fin} because $f$ is
    necessarily flat.
    %by \cite[Thm.~23.1]{Matsumura}. 
    To prove the other assertion of
    the corollary, we let $x \in X_{(0)}$ and let $S_x = \Spec(k(x)) \times_X X'$.
    Then one knows that the composite $K_1(k(x)) \xrightarrow{f^*} G_1(S_x)
    \xrightarrow{f_*} K_1(k(x))$ is multiplication by $[f_*(\sO_{S_x})] \in
    K_0(k(x)) \cong \Z$. But this is multiplication by
    $\dim_{k(x)}(S_x) = [k':k]$. 
\end{proof}

  \subsection{Continuity of tame class group}\label{sec:Cont}
  Recall from
  \cite[Defn.~5]{Kerz-JAG} that a functor $F$ from the category of (unital) commutative
  rings to the category of abelian groups is said to be continuous if
  for every filtering direct limit of rings $A = \varinjlim A_i$, the natural
  map $\varinjlim F(A_i) \to F(A)$ is an isomorphism. It is shown in
  Proposition~6 of op. cit. that the Milnor $K$-theory of rings is a continuous
  functor. This implies the same for the relative Milnor $K$-theory.
  
Let $k$ now be an imperfect field of characteristic exponent $p \ge 2$.
Assume further that $[k: k^p] < \infty$.
Let $k^{\heartsuit}$ denote a fixed
perfect closure of $k$. For $X \in \Sch_k$, we let
$X^{\heartsuit} = X_{k^\hs}$. Note that
$\pi \colon X^{\heartsuit} \xrightarrow{\cong} \lim X_\bullet$,
where the pro-scheme $X_\bullet = \{X_i\}_{i \ge 0}$ is defined by taking 
$k_i = k^{1/{p^i}}$ and $X_i = X \times_{\Spec(k)} \Spec(k_i)$.
For every $j \ge i$, the projection map $f_{ij} \colon X_j \to X_i$
is a finite flat radicial morphism whose degree is a power of $p$.
We let $g_i \colon X_i \to X, \ f_i \colon X^{\heartsuit} \to X_i$ and
$f \colon X^{\heartsuit} \to X$ be the projection maps. If $X \in \Sm_k$ is integral,
then
$X_i$ is an integral smooth $k_i$-scheme for each $i$ and $X^{\heartsuit}$ is an
integral smooth $k^{\heartsuit}$-scheme. In particular, for $i \geq 0$, we have homomorphisms
$f_i^* \colon C^t(X_i) \to C^t(X^{\hs})$ by \corref{cor:Sm-PB-insep}. Taking the limit, we get a homomorphism 
$\pi^* \colon {\varinjlim}_{i\ge 0} \ C^\tm(X_i) \to C^\tm(X^{\heartsuit})$. The following will be a key
lemma for us.

\begin{lem}\label{lem:Continuity}
  Assume that $X$ is a smooth and integral $k$-scheme. 
  Then the pull-back 
  \[
    \pi^* \colon \ {\varinjlim}_{i \ge 0} \ C^\tm(X_i) \to C^\tm(X^{\heartsuit})
  \]
  is an isomorphism.
\end{lem}
\begin{proof}
Suppose that $y \in X^{\heartsuit}_{(0)}$ and $y_j = f_j(y)$ for any $j \ge 0$.
Then there will exist an integer $j \gg 0$ and a closed subscheme $Y_j \subset X_j$
such that $Y_j \times_{X_j} X^{\heartsuit} = \Spec(k(y))$. The support of such $Y_j$ must
be the closed point $y_j$ of $X_j$. Furthermore, the fppf descent
implies that $Y_j$ must be reduced (e.g., use
\cite[Tag~033F]{SP} and note that $f_j$ is an affine morphism).
  In other words, $Y_j = \Spec(k(y_j))$. The same argument implies that
  $\Spec(k(y_j)) \times_{X_j} X_l = \Spec(k(y_l))$ and
  $\Spec(k(y_j)) \times_{X_j} X^{\heartsuit}
  = \Spec(k(y))$ for every $l \ge j$. It follows from this and the definition of
  $f^*_l$ 
  that $f^*_l(a_l) = a_l$ for every $l \ge j$ and $a_l \in K^M_1(k(y_l))$.
By the continuity of Milnor $K$-theory, it follows that
  every $a \in K^M_1(k(y))$ has the property that there exists $l \ge j$ and
  $a_l \in K^M_1(k(y_l))$ such that $f^*_l(a_l) = a$. This shows that $\pi^*$ is
  surjective.

To prove that $\pi^*$ is injective, suppose there exists $\alpha
  \in \varinjlim_j \ C^\tm(X_j)$ such that $\pi^*(\alpha) = 0$.
  We can then find an integer $j \ge 0$, a 0-dimensional reduced closed subscheme
  $S_j \subset X_j$ such that $\alpha$ is represented by an element of $K_1(S_j)$.
  We let $S'$ be the set theoretic inverse image of $S_j$ in $X^{\heartsuit}$ 
(with the reduced closed subscheme structure) so that
  $f^*_j(\alpha) \in K_1(S')$. As explained in the proof of surjectivity of
  $\pi^*$, we can find $l \gg j$ and a reduced closed subscheme $S_l \subset X_l$ such
  that $S_l \times_{X_l} X^{\heartsuit} = S'$. We can replace $\alpha$ by
  $f^*_{jl}(\alpha) \in K_1(S_l)$ (note that $S' \to S_l \to S_j$ are bijections). We
  can also replace $k$ by any $k_l$ and $\alpha$ by $f^*_{jl}(\alpha)$
  with $l \gg j$. We can thus assume that there
  exists a 0-dimensional reduced closed subscheme
  $S_0 \subset X$ such that $\alpha$ is represented by an element of $K_1(S_0)$.
  Furthermore, $S' = S_0 \times_X X^{\heartsuit}$. In particular, if we let
  $S_j = f_j(S')$, then $S_j = S_0 \times_X X_j$ and
  $S' = S_j \times_{X_j} X^{\heartsuit}$ for every $j \ge 0$.

  We fix an integral compactification $u \colon X \inj \ov{X}$ and
  let $u' \colon X^\hs \inj \ov{X}^\hs$ be the induced open immersion.
  Since $X^\hs \in \Sm_{k^\hs}$, we see that $u'$ has factorization
  $X^\hs \inj (\ov{X}^\hs)_\red \inj \ov{X}^\hs$. We let $\ov{X}_j := \ov{X}_{k_j}$.
Since $f^*(\alpha)$ lies in the image of $\partial^1_{X^{\heartsuit}} \colon
F^1_2(X^{\heartsuit}) \to F^0_1(X^{\heartsuit})$, we can find integral curves
$C_1, \ldots , C_r \subset X^{\heartsuit}$ satisfying the following.
Let $\ov{C}_i$ be the closure of $C_i$ in $\ov{X}^\hs$ with the reduced (hence integral)
closed subscheme structure and let
$\nu'_i \colon \ov{Y}_i := (\ov{C}_i)_n \to \ov{X}^\hs$ be the canonical map.
Let $T'_i = (\nu'_i)^{-1}(\ov{X}^\hs \setminus X^{\heartsuit})$ and
$A'_i = \sO_{\ov{Y}_i, T'_i}$. Let $J'_i$ be the Jacobson radical of
$A'_i$. Then, there exist elements $\beta'_i \in K^M_2(A'_i, J'_i)
\subset K^M_2(k^\hs(C_i))$ (in view of \lemref{lem:Fil-Milnor-comp}) such that
$f^*(\alpha) = \stackrel{r}{\underset{i = 1}\sum} \partial(\beta'_i) \in K_1(S')$.

%\enlargethispage{15pt}

We can now find an integer $j \gg 0$ and closed subschemes $\ov{Z}^j_i \subset \ov{X}_j$
such that
\begin{equation}\label{eqn:Continuity-0}
  \ov{C}_i \cong \ov{Z}^j_i \times_{\ov{X}_j} \ov{X}^\hs \ \textnormal {for all } 1 \le i \le r.
\end{equation}
In particular, $\ov{C}_i \to \ov{Z}^j_i$ is faithfully flat.
By the fppf descent, each $\ov{Z}^j_i$ must be an integral
curve in $\ov{X}_j$. By the same argument, it follows that
$\ov{Z}^l_i := \ov{Z}^j_i \times_{\ov{X}_j} \ov{X}_l$ is an integral curve in $\ov{X}_l$
for every $l \ge j$ and $1 \le i \le r$. Moreover,
$\ov{C}_i \xrightarrow{\cong} {\underset{l \ge j}\varprojlim} \ \ov{Z}^l_i$.
We let $Z^l_i = \ov{Z}^l_i \bigcap X_l$. Since $\ov{Z}^l_i \subset (\ov{X}_l)_\red
\subset ((\ov{X}_j)_\red)_l$
  for each $l$ and $i$, and similarly $\ov{C}_i \subset (\ov{X}^\hs)_\red
  \subset  ((\ov{X}_j)_\red)^\hs$,
we can pass from $\ov{X}$ to $(\ov{X}_j)_\red$ which allows us to assume that
$j = 0$. Note that $(\ov{X}_j)_\red$ is an integral compactification of $X_j$
for every $j \ge 0$.

Let $\ov{Y}^j_i$ (resp.  $Y^j_i$) be the normalization of $\ov{Z}^j_i$ (resp. $Z^j_i$)
and let $\nu^j_i \colon \ov{Y}^j_i \to \ov{Z}^j_i$ be the normalization maps.
Since the normalization is a functor in the category of integral schemes with dominant
morphisms which commutes with projective limits of integral schemes with
affine transition morphisms, it follows that there is a commutative diagram
\begin{equation}\label{eqn:Continuity-1}
  \xymatrix@C1pc{
    \ov{Y}_i \ar[r]^-{\phi_l} \ar[d]_-{\nu'_i} & \ov{Y}^l_i \ar[r]^-{\phi_{jl}}
    \ar[d]^-{\nu^l_i} & \ov{Y}^j_i \ar[d]^-{\nu^j_i} \\
    \ov{C}_i \ar[r]^-{\phi_l} & \ov{Z}^l_i \ar[r]^-{\phi_{jl}} & \ov{Z}^j_i}
\end{equation}
for every $l \ge j \ge 0$ and $\ov{Y}_i \xrightarrow{\cong}
{\varprojlim}_j \ \ov{Y}^j_i$.
If we let $T^j_i = (\nu^j_i)^{-1}(\ov{Z}^j_i \setminus Z^j_i) = 
(\nu^j_i)^{-1}(\ov{X}_j \setminus X_j)$ and $A^j_i = \sO_{\ov{Y}^j_i, T^j_i}$, then we
moreover have
$T^l_j = (\phi_{jl})^{-1}(T^j_i)$ and $T'_i = (\phi_j)^{-1}(T^j_i)$ for every
$l \ge j \ge 0$. This implies that ${\varinjlim}_j \ A^j_i
\xrightarrow{\cong} A'_i$ for each $1 \le i \le r$.

Since the transition maps of $\{A^j_i\}_{j \ge 0}$ are faithfully flat and $J'_i$ is
finitely generated, we can find
$j \gg 0$ such that $J'_i = J^j_i A'_i$ for each $i$. In particular,
$J'_i = \varinjlim_j J^j_i$.
By the continuity of relative Milnor $K$-theory, it follows
that $K^M_2(A'_i, J'_i) = \varinjlim_j K^M_2({A}^j_i, {J}^j_i)$ for every $i$.
By passing from $X$ to $X_j$ with $j \gg 0$, we can therefore assume that
there exist 
\begin{equation}\label{eqn:Continuity-2}
\beta^0_i \in K^M_2(A^0_i, J^0_i) \subset K^M_2(k(Z^0_i)) \ \mbox{such \ that} \
\beta'_i = f^*(\beta^0_i) \ \textnormal{for all } i.
\end{equation}

We now observe that the map $\phi_0 \colon C_i \to Z^0_i$ is flat by
~\eqref{eqn:Continuity-0}. Hence, for every open subscheme $U \subset Z^0_i$ and
$U' = (\phi_0)^{-1}(U)$, there is a commutative diagram of Quillen $K$-theory
localization sequences
\[
\xymatrix@C1pc{
  K_2(Z^0_i) \ar[r] \ar[d]_-{\phi^*_0} & K_2(U) \ar[r]^-{\partial} \ar[d]^-{\phi^*_0} &
  K_1(Z^0_i \setminus U) \ar[d]^-{\phi^*_0} \\
  K_2(C_i) \ar[r] & K_2(U') \ar[r]^-{\partial} & K_1(C_i \setminus U'),}
\]
where the vertical arrows are flat pull-back maps.
Taking the limit over $U$ and noting that $\Spec(k^\hs(C_i)) = C_i \times_{Z^0_i}
\Spec(k(Z^0_i))$, we get a commutative diagram
\begin{equation}\label{eqn:Continuity-3}
  \xymatrix@C1pc{
    K^M_2(k(Z^0_i)) \ar[r]^-{\partial} \ar[d]_-{\phi^*_0} & F^0_1(Z^0_i)
    \ar[d]^-{\phi^*_0} \ar@{^{(}->}[r] & F^0_1(X) \ar[d]^-{f^*} \\
    K^M_2(k^\hs(C_i))  \ar[r]^-{\partial} & F^0_1(C_i)
    \ar@{^{(}->}[r]  & F^0_1(X^\hs),}
  \end{equation}
  where we have identified the Milnor $K_i$ and their boundary maps for fields with
  Quillen $K_i$ and their boundary maps for $i \le 2$.
We conclude from this that 
\[
  f^*\left(\alpha - \stackrel{r}{\underset{i = 1}\sum}
(\partial(\beta^0_i))\right) = \phi^*_0\left(\alpha - \stackrel{r}{\underset{i = 1}\sum}
(\partial(\beta^0_i))\right) = 0 \ \mbox{in} \ K_1(S').
\]

Using the continuity property of $K_1(-)$ again, it follows that
$g^*_j(\alpha - \stackrel{r}{\underset{i = 1}\sum}
(\partial(\beta^0_i))) = \phi^*_{0j}((\alpha - \stackrel{r}{\underset{i = 1}\sum}
(\partial(\beta^0_i)))) = 0$ in $K_1(S_j)$ for all $j \gg 0$.
Equivalently, $g^*_j(\alpha) =  \stackrel{r}{\underset{i = 1}\sum}
g^*_j(\partial(\beta^0_i))$ in $F^0_1(X_j)$ for all $j \gg 0$.
Since we have shown in ~\eqref{eqn:Continuity-2} that
$\beta^0_i \in K^M_2(A^0_i, J^0_i) \subset K^M_2(k(Z^0_i))$ for every $i$, it follows from
\propref{prop:TCG-PB} that $g^*_j(\alpha)$ dies in $C^\tm(X_j)$.
We have thus shown that $\pi^*$ is injective.
This concludes the proof.
\end{proof}

\begin{cor}\label{cor:Continuity-*}
  Keep the assumptions of \lemref{lem:Continuity}. Then we have an isomorphism
  \[
      \pi^* \colon C^\tm(X)[\tfrac{1}{p}] \to  C^\tm(X^\hs)[\tfrac{1}{p}].
    \]
    \end{cor}
  \begin{proof}
    In view of \lemref{lem:Continuity}, it suffices to show that if ${k'}/k$ is a finite
    purely inseparable extension and $X' = X \times_{\Spec(k)} \Spec(k')$, then
    the pull-back map $\pi^* \colon C^\tm(X)[\tfrac{1}{p}] \to  C^\tm(X')[\tfrac{1}{p}]$
    is an isomorphism if $\pi \colon X' \to X$ is the projection.
    In this case, the surjectivity follows from \lemref{lem:Milnor-K-perf} and the
    injectivity from \corref{cor:Sm-PB}.
    \end{proof}

\begin{lem}\label{lem:Milnor-K-perf}
  Let ${k'}/k$ be a purely inseparable field extension. Let $m , n \ge 0$ be integers
  such that $(p,m) =1$. Then we have isomorphisms
  \[
    \pi^* \colon {K^M_n(k)}[\tfrac{1}{p}] \to {K^M_n(k')}[\tfrac{1}{p}] \ \mbox{and} \
    \pi^* \colon {K^M_n(k)}/m \to {K^M_n(k')}/m.
  \]
\end{lem}
\begin{proof}
  By a limit argument, we can easily reduce to the case when ${k'}/k$ is a
  finite extension.
  In this case, one knows that the composite map
${K^M_n(k)} \to {K^M_n(k')} \xrightarrow{N_{{k'}/k}} {K^M_n(k)}$ is
  multiplication by $p^r$ for some $r \ge 0$ by \cite[\S~5.4]{Bass-Tate}.
  This implies that $\pi^*$ are injective. For surjectivity,
  we can use the surjection $K^M_1(k') \otimes \cdots \otimes K^M_1(k') \surj K^M_n(k')$
  (where the tensor product is taken $n$ times) to reduce to showing the surjectivity
  for $n = 1$. On the other hand, it is easy to see that ${K^M_1(k')}/{K^M_1(k)}$ is
  annihilated by $p^r$ for some $r \ge 0$. In particular, it is uniquely $m$-divisible.
\end{proof}

\subsection{Relation with Wiesend class groups}\label{sec:Wiesend*}
We let $k$ be a local field and $X \in \Sch_k$ an integral normal $k$-scheme.
For $C \in \sC(X)$,
we let $K^M_2(k(C_n)_{\infty}) = {\underset{x \in \Delta(C_n)}\bigoplus} K^M_2(k(C_n)_x)$.
We let $U'_1K^M_2(k(C_n)_{\infty})  =
{\underset{x \in \Delta(C_n)}\bigoplus} U'_1K^M_2(k(C_n)_x)$ (cf. \S~\ref{sec:J-1}).
\begin{defn}\label{defn:Idel-X-0}
 % \begin{enumerate}
 % \item
  We let  $I(X) = \left({\underset{C \in \sC(X)}\bigoplus} K^M_2(k(C_n)_{\infty})\right)
  \bigoplus \left({\underset{x \in X_{(0)}}\bigoplus} K^M_1(k(x))\right)$ and call it the idele group of $X$.
%\item
%$I^{\tm,W}(X) = \left({\underset{C \in \sC(X)}\bigoplus} \frac{K^M_2(k(C_n)_{\infty})}
%    {U'_1 K^M_2(k(C_n)_{\infty})}\right) \bigoplus 
%  \left({\underset{x \in X_{(0)}}\bigoplus} K^M_1(k(x))\right)$.
% \end{enumerate}
\end{defn}

For $C \in \sC(X)$, the product of inclusions into the completions along the points of
$\Delta(C_n)$ induces $\iota_C \colon K^M_2(k(C_n)) \to K^M_2(k(C_n)_{\infty})$.
We also have the boundary map 
$\partial_C \colon  K^M_2(k(C_n)) \to F^0_1(C)$ (cf. \S~\ref{sec:IG}).
We thus get a homomorphism $K^M_2(k(C_n)) \xrightarrow{(\iota_C, \partial_C)}
K^M_2(k(C_n)_{\infty}) \bigoplus F^0_1(C) \inj K^M_2(k(C_n)_{\infty}) \bigoplus  F^0_1(X)$.
We let $F^{1,W}_2(X) = {\underset{C \in \sC(X)}\bigoplus} K^M_2(k(C))$.

\begin{defn}\label{defn:Idele-X-1}
  The idele class group $C(X)$ is defined to be the cokernel of the homomorphism
  \[
    \sum_{C \in \sC(X)} (\iota_C, \partial_C) \colon F^{1,W}_2(X) \to I(X).
  \]
 \end{defn}

 Let $\vartheta_X \colon {\underset{C \in \sC(X)}\bigoplus} U'_1 K^M_2(k(C_n)_{\infty})
 \inj {\underset{C \in \sC(X)}\bigoplus} K^M_2(k(C_n)_{\infty}) \to C(X)$ denote the
 composite map.
 
  \begin{lem}\label{lem:Approximation}
    There is a canonical exact sequence
    \[
      {\underset{C \in \sC(X)}\bigoplus} U'_1 K^M_2(k(C_n)_{\infty})
        \xrightarrow{\vartheta_X} C(X) \xrightarrow{\kappa_X} C^\tm(X) \to 0.
      \]
    \end{lem}
    \begin{proof}
It is easy to check that the identity map of $F^0_1(X)$ induces canonical a 
homomorphism $\delta_X \colon C^\tm(X) \to \coker(\vartheta_X)$.
The surjectivity of $\delta_X$ is an easy consequence of the weak approximation
theorem for Milnor $K$-theory (cf. \cite[\S~2, Prop.~3]{Kato86}).
To show that $\delta_X$ is injective, let $C_1, \ldots , C_r \in \sC(X)$ be distinct
curves and $\alpha_i \in K^M_2(k(C_i))$ such that
  $\sum_i \iota_{C_i}(\alpha_i) + \sum_i \partial_{C_i}(\alpha_i) =  \beta \in F^0_1(X)$.
  Then we must have $\iota_{C_i}(\alpha_i)  \in U'_1 K^M_2(k((C_i)_n)_{\infty}) \subset
  K^M_2(k(C_i)_\infty)$ for every $i$.
  Using \lemref{lem:K2-dvr} and \corref{cor:Semilocal-K-2}, this implies that
  $\alpha_i \in \Fil_{\Delta((C_i)_n)} K^M_2(k(C_i))$ for every $i$.
  It follows from this that $\alpha := \sum_i \alpha_i \in F^1_2(X)$ and
  $\partial^1_X(\alpha) = \beta$.  We have thus shown that $\delta_X$
  is injective.
\end{proof}

Henceforth, we shall make no distinction between $C^\tm(X)$ and $\coker(\vartheta_X)$.
The norm map $N_{{k(x)}/k} \colon k(x)^\times \to k^\times$ for $x \in \ov{C}_n$
and $C \in \sC(X)$
define a homomorphism $N_X \colon I(X) \to k^\times$. It clearly annihilates the
image of ${\underset{C \in \sC(X)}\bigoplus} U'_1 K^M_2(k(C_n)_{\infty})$.
Suslin's reciprocity law for Milnor $K$-theory \cite{Suslin-Izv} implies that
$N_X$ factors through $C^\tm(X)$. We let $C^\tm(X)_0$ be the kernel of
$N_X \colon C^\tm(X) \to k^\times$. We define $C(X)_0$ similarly.

\subsection{Relation between $C^\tm(X)$ and $C(X)$}
\label{sec:Comp-CG}
Let $k$ be a local field.
By \lemref{lem:Approximation}, we have canonical homomorphisms
$C(X) \xrightarrow{\kappa_X} C^\tm(X) \xrightarrow{\varpi_X} C(\ov{X})$.
The following result provides a comparison between the first two groups.
This is the analogue of the first isomorphism of
\corref{cor:Pi-decom-0} for class groups.

  \begin{lem}\label{lem:Idele-Tame-prime-to-p}
    Let $X \in \Sch_k$ be an integral normal scheme and $m \in k^\times$ an integer.
    Then we have an isomorphism 
    \begin{equation}\label{eqn:Idele-Tame-prime-to-p-0}
      \kappa_X  \colon {C(X)}/m \xrightarrow{\cong} {C^{\tm}(X)}/m.
    \end{equation}
    \end{lem}
    \begin{proof}
      Combine Lemmas~\ref{lem:Approximation} and ~\ref{lem:Rigidity-Milnor}.
     \end{proof}

\begin{lem}\label{lem:Rigidity-Milnor}
    Let $A$ be a Henselian local Noetherian ring containing an infinite field having
    the maximal ideal $\fm$. Let $I \subset \fm$ be an $\fm$-primary ideal
    and let $m$ be an integer invertible in $A$. Then
    ${K^M_n(A, I)}/m = 0$ for all integers $n \ge 0$.
  \end{lem}
  \begin{proof}
    We can assume $n \ge 1$ because the lemma is trivial otherwise.
    We have a short exact sequence
    \begin{equation}\label{eqn:Rigidity-Milnor-0}
      0 \to K^M_n(A, I) \to K^M_n(A, \fm) \to K^M_n(A/I, \fm/I) \to 0.
      \end{equation}
      Since $\fm = \sqrt{I}$, it follows from \cite[Lem.~2.6]{Gupta-Krishna-BF} that
      $K^M_n(A/I, \fm/I)$ is uniquely $m$-divisible if the residue characteristic of
      $A$ is positive. The same also holds if the residue characteristic of
      $A$ is zero as a consequence of \cite[Theorem~2.12]{GT}. Indeed,
      since $A/I$ is $\fm$-adic complete local ring 
       with residue       
       field $A/\fm$, the quotient map $A/I \to A/\fm$ admits a splitting by Cohen
       structure theorem. Since a local ring
       containing an infinite field is weakly $n$-stable for all $n \geq1$, we can now
       apply \cite[Theorem~2.12]{GT} to conclude that 
      $K^M_n(A/I, \fm/I)$ is uniquely $m$-divisible.
      We can therefore assume that
      $I = \fm$.

      Now, there is a surjection
      $K^M_1(A,\fm) \otimes (A^\times \otimes \cdots \otimes A^\times) \surj K^M_n(A, \fm)$
    by \cite[Lem.~1.3.1]{Kato-Saito-2}, where the tensor product of $A^\times$ is taken
    $(n-1)$ times. It suffices therefore to show that ${K^M_1(A,\fm)}/m = 0$.
    Since $K^M_1(A,\fm) \cong (1 + \fm) \subset A^\times$, the desired statement follows
    easily from Hensel's lemma, which is applicable since $m \in A^\times$. We skip its
    proof.
\end{proof}

  \subsection{Tame vs. unramified class groups}
  \label{sec:CFT-tame-unr}
  %The goal of this section is prove an analogue of \thmref{thm:Tame-nr-main} for
 % the class groups. 
 The goal of this section is to prove \thmref{thm:Main-6}.
  We let $k$ be an infinite field.
  Let $\ov{X}$ be a smooth projective connected $k$-scheme of dimension $d \ge 1$. Let
  $j \colon X \inj \ov{X}$ be the inclusion of a nonempty open subscheme with the
  complement $Z$ (having the reduced closed subscheme structure).
Since the proof is based on a moving lemma for the higher Chow groups, we recall the
  necessary definitions.

We let $\square^n_k = (\P^1_k \setminus \{1\})^n$ and $\ov{\square}^n_k = (\P^1_k)^n$.
  For any closed subscheme $Y \subset \ov{X} \times \ov{\square}^n_k$, we let
  $g^i_Y \colon Y \to \ov{\square}^1_k$ be the projection to the $i$-th factor of
  $\ov{\square}^n_k$ and let $f_Y \colon Y \to \ov{X}$ be the projection to $\ov{X}$.
We recall the following from \cite[\S~1]{Krishna-Levine}.
Let $z^{d+1}(\ov{X}, \bullet)$ denote the cubical cycle complex of $\ov{X}$
{\`a} la Bloch of codimension $(d+1)$-cycles.
The differential of this complex at level $n$ is given by
$\partial = \stackrel{n}{\underset{i = 1}\sum} (-1)^i (\partial^i_\infty - \partial^i_0)$.
The homology of this complex (modulo the subcomplex of degenerate cycles)
at level $n$ is the higher Chow group $\CH^{d+1}(\ov{X}, n)$.

Let $(z^{d+1}_Z(\ov{X}, \bullet), \partial_Z)$ be the subcomplex of
$(z^{d+1}(\ov{X}, \bullet), \partial)$
  such that $z^{d+1}_Z(\ov{X}, n)$ is the free abelian group generated by
  admissible integral cycles which intersect $Z \times F$ properly for each face
  $F$ of $\square^n_k$.
  Since there can be no degenerate 0-cycles and degenerate cycles are preserved under
  the boundary maps, one checks that the homology of $z^{d+1}_Z(\ov{X}, \bullet)$
  (modulo the subcomplex of degenerate cycles) at level one is the cokernel of
  the map $\partial_Z \colon z^{d+1}_Z(\ov{X}, 2) \to z^{d+1}_Z(\ov{X}, 1)$, where
  $z^{d+1}_Z(\ov{X}, 1)$ is the free abelian group generated by closed points
  of $\ov{X} \times \square^1_k$ which neither lie in $Z \times \square^1_k$ nor in
  any proper face of $\ov{X} \times \square^1_k$
  and $z^{d+1}_Z(\ov{X}, 2)$ is the free abelian group generated by integral
  curves in $\ov{X} \times \square^2_k$ which intersect all faces of
  $X \times \square^2_k$ and $Z \times \square^2_k$ properly. In particular, they
  do not meet any proper face of $Z \times \square^2_k$.
  By Bloch's  moving lemma \cite{Bloch-JAG} (cf. \cite[Thm.~1.10]{Krishna-Levine}
  for cubical version), the canonical map
\begin{equation}\label{eqn:ML-0}
  \frac{z^{d+1}_Z(\ov{X}, 1)}{\partial_Z(z^{d+1}_Z(\ov{X}, 2))} \to \CH^{d+1}(\ov{X},1)
\end{equation}
is an isomorphism.

Recall from \S~\ref{sec:IG} that $C(\ov{X})$ is the cokernel of the map
$\partial^1_{\ov{X}} \colon F^{1}_2(\ov{X}) = {\underset{C \in \sC(\ov{X})}\bigoplus}
K^M_2(k(C)) \to F^0_1(\ov{X})$. 
It follows from the Gersten complex for the Milnor $K$-theory of regular local rings
(cf. \lemref{lem:Semilocal-K-1}) that for any $C \in \sC({X})$, one has
$\partial^1_{\ov{C}_n}(K^M_2(\sO_{\ov{C}_n, \Delta(C_n)})) \subset F^0_1(C_n)$. In particular,
$\partial^1_{\ov{X}}(K^M_2(\sO_{\ov{C}_n, \Delta(C_n)})) \subset F^0_1(X)$.
We write $K^M_2(\sO_{\ov{C}_n, \Delta(C_n)})$ as $\Fil_{\nr} K^M_2(k(C))$ for
$C \in \sC({X})$. 
We let $F^{1, \nr}_2({X}) =  {\underset{C \in \sC({X})}\bigoplus} 
\Fil_{\nr} K^M_2(k(C))$ and let $C^{\nr}({X})$ be the cokernel of the map
$\partial^1_{X} \colon F^{1, \nr}_2({X}) \to  F^0_1(X)$.
Our key result is the following.

\begin{lem}\label{lem:ML-main}
  The canonical map $C^{\nr}({X}) \to C(\ov{X})$ is an isomorphism.
\end{lem}
\begin{proof}
  For any closed point $x \in z^{d+1}(\ov{X}, 1)$, we let $y$ be the image of
  $x$ in $\ov{X}$ and let $\Psi_1([x]) = N_{{k(x)}/{k(y)}}(g^1_x) \in K^M_1(k(y)) \subset
  F^0_1(\ov{X})$. For an integral curve $Y \in z^{d+1}(\ov{X}, 2)$, we let
  $\Psi_2([Y]) = 0$ if $f_Y \colon Y \to \ov{X}$ is constant.
  If the scheme-theoretic image $f_Y(Y)$ is an integral curve $\ov{Y'}$, we let $Y'= \ov{Y'} \cap X$. In this case, we let 
  $\Psi_2([Y]) = N_{{k(Y)}/{k(Y')}}(\{g^1_Y, g^2_Y\}) \in
  K^M_2(k(Y')) \subset F^1_2(\ov{X})$.
  By \cite[Thm.~2.5]{Landsburg}, $\Psi_1$ and $\Psi_2$ define a morphism of complexes
\begin{equation}\label{eqn:ML-main-0}
  \Psi \colon \left(z^{d+1}(\ov{X}, 2) \xrightarrow{\partial} z^{d+1}(\ov{X}, 1)\right)
  \to \left(F^1_2(\ov{X}) \xrightarrow{\partial^1_{\ov{X}}} F^0_1(\ov{X})\right).
\end{equation}
such that the induced map $\Psi \colon
\CH^{d+1}(\ov{X},1) = {\rm Coker}(\partial) \to {\rm Coker}(\partial^1_{\ov{X}}) =
C(\ov{X})$ is an isomorphism.

Since a 0-cycle in $z^{d+1}_Z(\ov{X}, 1)$ is supported in $X \times \square^1_k$,
we see that $\Psi_1(z^{d+1}_Z(\ov{X}, 1)) \subset F^0_1(X)$.
If $Y$ is an integral 1-cycle in $z^{d+1}_Z(\ov{X}, 2)$, then $Y$ does not lie in
$Z \times \square^2_k$ and does not meet any proper face of $X \times \square^2_k$
inside $Z \times \square^2_k$.
This implies that the image of $\{g^1_Y, g^2_Y\}$ dies in $K^M_1(k(y))$ for every
$y \in \nu^{-1}(Z)$, where $\ov{Y}$ is the closure of $Y$ in
$\ov{X} \times \ov{\square}^2_k$ and $\nu \colon \ov{Y}_n \to
\ov{X} \times \ov{\square}^2_k \to \ov{X}$ is the composite (finite) map.
It follows from \cite[Lem.~3.4]{Gupta-Krishna-REC} that the image of 
$\Psi_2([Y]) =  N_{{k(Y)}/{k(Y')}}(\{g^1_Y, g^2_Y\})$ dies in
$K^M_1(k(y'))$ for every $y' \in \Delta(Y'_n)$.
An application of \lemref{lem:Semilocal-K-1} shows  
that $\Psi_2([Y]) \in \Fil_{\nr} K^M_2(k(Y'))$.
We conclude that $\Psi_2(z^{d+1}_Z(\ov{X}, 2)) \subset F^{1,\nr}_2({X})$.
It follows that 
~\eqref{eqn:ML-main-0} induces a commutative diagram 
\begin{equation}\label{eqn:ML-main-1}
  \xymatrix@C1pc{
   \left(z^{d+1}_Z(\ov{X}, 2) \xrightarrow{\partial_Z} z^{d+1}_Z(\ov{X}, 1)\right)
   \ar[r]^-{\Psi} \ar@{^{(}->}[d] &
   \left(F^{1,\nr}_2({X}) \xrightarrow{\partial^1_{X}} F^0_1({X})\right) \ar@{^{(}->}[d] \\
    \left(z^{d+1}(\ov{X}, 2) \xrightarrow{\partial} z^{d+1}(\ov{X}, 1)\right)
    \ar[r]^-{\Psi} & \left(F^1_2(\ov{X}) \xrightarrow{\partial^1_{\ov{X}}}
      F^0_1(\ov{X})\right)}
  \end{equation}
of complexes, where the vertical arrows are the canonical inclusions.

Taking the homology, we get a commutative diagram
\begin{equation}\label{eqn:ML-main-2}
  \xymatrix@C1pc{
    {\rm Coker}(\partial_Z) \ar[r]^-{\Psi} \ar[d] & {\rm Coker}(\partial^1_X) \ar[d] \\
    {\rm Coker}(\partial) \ar[r]^-{\Psi} & {\rm Coker}(\partial^1_{\ov{X}}).}
\end{equation}

The left vertical arrow in ~\eqref{eqn:ML-main-2}
is an isomorphism by ~\eqref{eqn:ML-0} and the bottom
horizontal arrow is an isomorphism by ~\eqref{eqn:ML-main-0}. It follows that
the top horizontal arrow is injective.
To prove that this arrow is surjective, we fix $x \in X_{(0)}$.
Taking the graph of a rational function defines a homomorphism
$\Gamma_x \colon K^M_1(k(x)) \to \CH^1(k(x), 1)$ (cf. \cite[\S~2]{Totaro}).
Bloch's push-forward map on the cycle complexes
corresponding to inclusion $\Spec(k(x)) \inj \ov{X}$ induces a morphism
$z^1(k(x), \bullet) \to z^{d+1}(\ov{X}, \bullet)$ which clearly factors through
$z^{d+1}_Z(\ov{X}, \bullet)$. We thus have homomorphisms
$K^M_1(k(x)) \xrightarrow{\cong} \CH^1(k(x),1) \to {\rm Coker}(\partial_Z)
\xrightarrow{\Psi} {\rm Coker}(\partial^1_X)$ whose composition is identity on
$K^M_1(k(x))$.
This shows that the top horizontal arrow in ~\eqref{eqn:ML-main-2} is surjective.
We have thus shown that all arrows in ~\eqref{eqn:ML-main-2} are isomorphisms.
This concludes the proof.
\end{proof}

\begin{prop}\label{prop:ML-main-3}
  There is an exact sequence
  \[
   {\underset{C \in \sC(X), x \in \Delta(C_n)}\bigoplus} {K^M_2(k(x))}
   \to {C^\tm(X)} \to {C(\ov{X})} \to 0.
 \]
\end{prop}
\begin{proof}
  This is a direct consequence of Definition~\ref{defn:TCG-0} and
  \lemref{lem:ML-main}.
\end{proof}

The following result proves \thmref{thm:Main-6} (cf. \S~\ref{sec:Comp-CG} for
the definition of $\varpi_X$). Note that the second isomorphism of \thmref{thm:Main-6}
is already shown in \corref{cor:Pi-decom-0}.

\begin{thm}\label{thm:TCG-UTCG}
  Let $k$ be a local field of characteristic exponent $p \ge 2$. 
  Then the map $\varpi_{X} \colon C^\tm(X) \to C(\ov{X})$ is surjective. This map is
  bijective modulo $p^n$, for every integer $n \ge 1$.
\end{thm}
\begin{proof}
 In view of \propref{prop:ML-main-3}, it suffices to show that ${K^M_2(k(x))}/{p^n} = 0$ if $x$ is a closed point
of a curve over $k$. But this follows from 
\cite[Chap.~9, Prop.~4.2]{Fesenko-Vostokov}.
 %   To that end, we use the main result of  \cite{Merkurjev}, which says that
%$K^M_2(k(x))$ is the direct sum of a (uniquely) divisible group and the group %of roots of unity in $k(x)$. Since the latter group is cyclic of order prime to
%$p$ (cf. \cite[Prop.~II.5.7]{Neukirch}), we are done.
\end{proof}

\begin{cor}\label{cor:Rec-tame-unr}
  Under the notations and assumptions of \thmref{thm:TCG-UTCG},
  one has a commutative diagram 
  \begin{equation}\label{eqn:Rec-tame-unr-0}
    \xymatrix@C1pc{
      {C^\tm(X)}/{p^n} \ar[r]^-{\rho^\tm_{X}} \ar[d]_-{j_*} & {\pi^{\ab,\tm}_1(X)}/{p^n}
      \ar[d]^-{j_*} \\
      {C(\ov{X})}/{p^n} \ar[r]^-{\rho_{\ov{X}}} & {\pi^{\ab}_1(\ov{X})}/{p^n}}
  \end{equation}
  for every integer $n \ge 1$ such that the vertical arrows are isomorphisms.
\end{cor}
\begin{proof}
Combine \corref{cor:Pi-decom-0}, \propref{prop:TCG-PF} and \thmref{thm:TCG-UTCG}.
\end{proof}

\section{The reciprocity maps}\label{sec:REC}
In this section, we shall define the reciprocity homomorphisms from various class groups
of integral normal schemes over a local field to the corresponding {\'e}tale fundamental
groups. For topological abelian groups $G$ and $H$, we let
$\Hom_{\cf}(G, H) = \Hom_{\Tab}(G, H)_\tor$.

Let $K$ be a complete discrete valuation field whose residue field $\ff$ is a local
field. Recall from \cite[\S~7]{Kato-cft-1} that $K^M_1(K)$ and $K^M_2(K)$ are
endowed with the structure of topological abelian groups. We shall refer to their
underlying topology as the `Kato topology'.
The cup product pairing $K^M_2(K) \otimes H^1(K) \to H^3(K) \cong {\Q}/{\Z}$
induces a continuous homomorphism $\rho_K \colon K^M_2(K) \to G_K$.
We shall use the following result.

\begin{prop}\label{prop:Kato-REC}
We have the following.
\begin{enumerate}
\item
  The dual of $\rho_K$ induces isomorphisms
      \[
        \rho^\vee_K \colon H^1(K) \xrightarrow{\cong}
    \left\{\begin{array}{ll}\Hom_\cf(K^M_2(K), {\Q}/{\Z}) & \mbox{if $\Char(K) = 0$} \\
             \Hom_\Tab(K^M_2(K), {\Q}/{\Z}) & \mbox{if $\Char(K) > 0$}.
           \end{array}
         \right.
         \]
       \item
         These isomorphisms induce isomorphisms
         \[
         \rho^\vee_K \colon \Fil_0 H^1(K) \xrightarrow{\cong}    
         \Hom_\cf\left(\frac{K^M_2(K)}{U'_0 K^M_2(K)}, {\Q}/{\Z} \right) \  \mbox{if} \
           \Char(K) = 0.
       \]
       \[
       \rho^\vee_K \colon \Fil_m H^1(K) \xrightarrow{\cong}    
       \Hom_\Tab\left(\frac{K^M_2(K)}{U'_m K^M_2(K)}, {\Q}/{\Z} \right) \  \textnormal{for all } m
       \ge 0 \  \mbox{if} \  \Char(K) > 0.
     \]
\end{enumerate}
\end{prop}
\begin{proof}
Combine \cite[\S~3.2, \S~3.5]{Kato80}, \cite[Thm.~2, p.~304]{Kato-cft-1} and 
  \cite[Thm.~6.3]{Gupta-Krishna-REC}.
\end{proof}

\subsection{Reciprocity for idele class group}\label{sec:REC-0}
We now let $k$ be a local field and 
$X \in \Sch_k$ an integral normal scheme with function field $K$.
For $x \in X$, let $\iota^x \colon \Spec(k(x)) \to X$ be the inclusion.
Since the desired reciprocity
homomorphism is classical when $\dim(X)=0$, we assume that $\dim(X) \ge 1$.

By taking sum of the composite maps $K^M_1(k(x)) \xrightarrow{\rho_{k(x)}} G_{k(x)} \to
\pi^{\ab}_1(X)$ over $x \in X_{(0)}$, we get a homomorphism
$\rho^1_X \colon F^0_1(X) \to \pi^{\ab}_1(X)$. 
For $C \in \sC(X)$, \propref{prop:Kato-REC} yields 
a continuous reciprocity homomorphism
$\rho^\eta_\infty \colon K^M_2(k(\eta)_\infty) \to \pi^{\ab}_1(\Spec(k(\eta)_\infty))$,
where $\eta \in C$ is the generic point. Composing this with 
$\pi^{\ab}_1(\Spec(k(\eta)_\infty)) \to
\pi^{\ab}_1(k(\eta)) \to \pi^{\ab}_1(X)$ and summing over $C \in \sC(X)$, we get a
homomorphism
$\rho^2_X \colon {\underset{C \in \sC(X)}\bigoplus} K^M_2(k(C_n)_\infty) \to 
\pi^{\ab}_1(X)$. We let $\rho_X = (\rho^2_X + \rho^1_X) \colon I(X) \to \pi^{\ab}_1(X)$.

\begin{lem}\label{lem:REC-Idele}
  $\rho_X$ factors through a group homomorphism
  \[
    \rho_X \colon C(X) \to \pi^{\ab}_1(X).
  \]
\end{lem}
\begin{proof}
  By definitions of $\rho_X$ above and $C(X)$,
  it suffices to prove the lemma when $\dim(X) =1$. 
  We assume this to be the case and
  let $\ov{X}$ be the unique normal compactification of $X$. We let
$I'(\ov{X}) = {\underset{x \in \ov{X}_{(0)}}{\prod'}} K^M_2(K_x)$
be the restricted product with
respect to the subgroups $U'_0 K^M_2(K_x)$.  We let $I'(X) =
{\underset{x \in {X}_{(0)}}{\prod'}} K^M_2(K_x)$ and consider the diagram
\begin{equation}\label{eqn:REC-2}
  \xymatrix@C1pc{
    K^M_2(K) \ar[dr]_-{(\iota_X, \partial_X)} \ar[r]^-{\alpha} & I'(\ov{X})
  \ar[r]^-\gamma \ar[d]^-{\beta} & G_K \ar@{->>}[d] \\
    & I(X) \ar[r]^-{\rho_X} & \pi^{\ab}_1(X),}
  \end{equation}
  where the right vertical arrow is the canonical surjection and $\alpha$
  is induced by taking product of completions.
  It is easy to check that this arrow is well-defined. Our task is to show that
  $\rho_X \circ (\iota_X, \partial_X) = 0$.

To show the above, we first define $\beta$ and $\gamma$ such that
  ~\eqref{eqn:REC-2} commutes. To define $\beta$,
  note that we can write $I'(\ov{X}) =
  K^M_2(K_\infty) \bigoplus I'(X)$.
  We set $\beta$ to be the identity map on the first summand and 
  $\beta(a_x) = \partial_{X,x}(a_x)$ if $a_x \in K^M_2(K_x)$ with $x \in X_{(0)}$.
  This is well-defined because $\Ker(\partial_{X,x}) = U'_0 K^M_2(K_x)$.

To define $\gamma$, let $\chi \in H^1(K)$ and $(a_x) \in I'(\ov{X})$.
  We can find a dense open subscheme $U \subset X$
  such that $\chi$ is unramified along $U$. That is, $\chi \in H^1_\et(U, {\Q}/{\Z})
  \subset H^1(K)$. Equivalently, $\chi$ factors through $\pi^{\ab}_1(U)$.
It follows therefore from \propref{prop:Kato-REC}(3) that
$\sum_{x \in X_{(0)}} \< \chi_x, a_x\>$ is a finite sum. In other words, the cup product
pairing which defines the reciprocity for 2-local fields induces a pairing
\begin{equation}\label{eqn:REC-5}
  \<,\> \colon   H^1(K) \otimes I'(\ov{X}) \to {\Q}/{\Z}.
\end{equation}

Equivalently, we get a homomorphism $\gamma \colon  I'(\ov{X})\to G_K$ which is
compatible with the reciprocity homomorphisms of the completions of $K$.
 Furthermore,
\cite[\S~3.2, Lem.~6(2)]{Kato80} implies that ~\eqref{eqn:REC-2} is commutative. 
Since the composite horizontal arrow on the top in ~\eqref{eqn:REC-2} is zero by
\cite[\S~5, p.~120]{Kato-Saito-83} (see also \cite[Chap~II, Prop.~1.2]{Saito-JNT}), it
follows that $\rho_X \circ (\iota_X, \partial_X)$
is zero. This implies that $\rho_X$ factors through $C(X)$, as desired.
\end{proof}
  
\subsection{Reciprocity for tame class group}\label{sec:REC-Tame}
Let $k$ be a local field and $X \in \Sch_k$ an integral normal scheme with function
field $K$. 
\begin{prop}\label{prop:REC-T}
$\rho_X$ induces a homomorphism
  $\rho^\tm_X \colon C^{\tm}(X) \to \pi^{\ab, \tm}_1(X)$ such that we have a
  commutative diagram
  \begin{equation}\label{eqn:REC-T-0}
    \xymatrix@C1pc{
      C(X) \ar[r]^-{\rho_X} \ar[d]_-{\kappa_X} & \pi^{\ab}_1(X) \ar[d]^-{\tau_X} \\
      C^{\tm}(X) \ar[r]^-{\rho^{\tm}_X} & \pi^{\ab, \tm}_1(X).}
  \end{equation}
\end{prop}
\begin{proof}
 Using \propref{prop:Tame-cov} and \lemref{lem:Approximation}, it suffices to prove
  the proposition when $\dim(X) = 1$. We assume this to be the case and let $\ov{X}$
  be the unique normal compactification of $X$. Let $\eta$ denote the generic point of
  $X$. Using \propref{prop:TFG-3}, Definition~\ref{defn:Idele-X-1} and
  \lemref{lem:REC-Idele}, it suffices to show that if $x \in \Delta(X)$, then the
  image of $U'_1K^M_2(K_x)$ under the reciprocity map $\rho_{K_x} \colon K^M_2(K_x) \to
  G_{K_x}$ lies in the subgroup $G^{(1)}_{K_x}$. If $\Char(k) = 0$, then
  $U'_1K^M_2(K_x)$ is a divisible group by \lemref{lem:Rigidity-Milnor}. Hence,
  its image under $\rho_{K_x}$ must be zero as $G_{K_x}$ is profinite. If
  $\Char(k) > 0$, then using the Pontryagin duality and isomorphism 
~\eqref{eqn:fil_n-dual}, the
  assertion is equivalent to showing that $\chi \circ \rho_{K_x}(a) = 0$ for every
  $\chi \in \Fil_1 H^1(K_x)$ and $a \in U'_1K^M_2(K_x)$.
But this follows from \propref{prop:Kato-REC}(2).
\end{proof}

\begin{cor}\label{cor:REC-T-PF}
  Let $f \colon X' \to X$ be a  morphism of regular $k$-schemes.
  Then we have a commutative diagram
  \begin{equation}\label{eqn:REC-T-PF-0}
    \xymatrix@C1pc{
      C^\tm(X') \ar[r]^-{\rho^\tm_{X'}} \ar[d]_-{f_*} & \pi^{\ab,\tm}_1(X') \ar[d]^-{f_*} \\
      C^\tm(X) \ar[r]^-{\rho^\tm_{X}} & \pi^{\ab, \tm}_1(X).}
  \end{equation}
 \end{cor}
\begin{proof}
  All arrows in the diagram exist by Propositions~\ref{prop:Tame-cov},
  ~\ref{prop:TCG-PF} and ~\ref{prop:REC-T}.
  Its commutativity follows from the
explicit construction of these maps using \cite[\S~3.2, Cor.~1]{Kato80}.
\end{proof}

\section{Tame class group and motivic cohomology}\label{sec:TCG-MC}
The motivic cohomology of schemes with compact support ({\`a} la Voevodsky)
will be a key player in our proofs of the main results of this paper.
%However, this is not well behaved over imperfect base fields.
The purpose of this section is to prove some necessary results related to
motivic cohomology over possibly imperfect fields.
We shall use these results to establish an isomorphism
between the tame class group with finite coefficients (which are invertible in
the ground field) and certain motivic cohomology group with compact support over
local fields. 
Throughout this section, we fix a prime field $\kappa$  of characteristic exponent
$p \ge 1$. We let $\Nsch_{\kappa}$ denote the category of Noetherian $\kappa$-schemes of
finite Krull dimension. We let $\Reg_\kappa \subset \Nsch_{\kappa}$ denote the full
subcategory of consisting of regular $\kappa$-schemes.

\subsection{Review of Voevodsky's motives}\label{sec:DM}
In \cite[Thm.~11.4.5]{CD-Springer}, Cisinski-D{\'e}glise showed that
given any $X \in \Nsch_{\kappa}$ and any (unital) commutative ring $R$,
there is a monoidal triangulated category of
integral mixed motives $\dm(X, R)$ which coincides with the original construction of
Voevodsky \cite{Voe00} when $X$ is the spectrum of a perfect field and of Suslin
\cite{Suslin-AKT} when $X$ is the spectrum of an arbitrary field. The assignment
$X \mapsto \dm(X, R)$ satisfies many (but not necessarily all) of the axioms which
together constitute the so-called `Grothendieck six functor formalism'
(cf. \cite[Thm.~1, p.~xviii]{CD-Springer}).

We now let $\Lambda$ be a commutative ring which we take to be $\Z$ if $p = 1$ or
any $\Z[\tfrac{1}{p}]$-algebra if $p \ge 2$.
For $X$ as above, one defines $\un{\dm}_\cdh(X, \Lambda)$ in the same
manner as $\dm(X, \Lambda)$ by replacing the Nisnevich topology on $\Sm_X$ by the cdh
topology on $\Sch_X$. We let $\dm_\cdh(X, \Lambda)$ be the full
localizing (i.e., thick) triangulated subcategory of $\un{\dm}_\cdh(X, \Lambda)$
generated by motives of the form $M_X(Y)(n)$ for $Y \in \Sm_X$ and $n \in \Z$.
If $X = \Spec(A)$ is affine, we shall use $\dm_\cdh(A, \Lambda)$ as another notation
for $\dm_\cdh(X, \Lambda)$. Similar convention will be followed for other functors on
the category of affine $\kappa$-schemes. The constant  Nisnevich (resp. cdh) sheaf with
transfer associated to the ring $\Lambda$ is the identity
object for the monoidal structure of $\dm(X, \Lambda)$
(resp. $ \dm_\cdh(X, \Lambda)$). We shall denote both of these by $\Lambda_X$.

In \cite{CD-Doc}, 
Cisinski-D{\'e}glise showed that ${\dm}_\cdh(X, \Lambda)$ has better behavior for
singular schemes than
$\dm(X, \Lambda)$ in that it satisfies all axioms of the six functor formalism.
Their main result is the following.

\begin{thm}$($Cisinski-D{\'e}glise, \cite{CD-Doc}$)$\label{thm:CD-cdh}
  $(1)$ \ The assignment $X \mapsto \dm_\cdh(X, \Lambda)$ defines a presheaf of monoidal
  triangulated categories on $\Nsch_{\kappa}$ which satisfies all conditions listed in
  \cite[\S~A.5]{CD-Springer}. In particular, it satisfies the six functor formalism and
  the localization axiom (cf. \cite[Defn.~2.3.2]{CD-Springer}).
  Additionally, given any separated and finite type morphism $f \colon X \to Y$ in
  $\Nsch_{\kappa}$, the functor $f^* \colon \dm_\cdh(Y, \Lambda) \to \dm_\cdh(X, \Lambda)$
  admits a left adjoint $f_{\sharp} \colon \dm_\cdh(X, \Lambda) \to \dm_\cdh(Y, \Lambda)$
  if either $f$ is smooth or $Y$ is the spectrum of a field. \\
  $(2)$ \ If $X \in \Reg_\kappa$, then the change of topology functor
  $a^*_\cdh \colon \dm(X, \Lambda) \to \dm_\cdh(X, \Lambda)$
  is an equivalence of monoidal triangulated categories. In particular,
  the presheaf of monoidal triangulated categories $X \mapsto \dm(X, \Lambda)$
  on $\Reg_{\kappa}$ satisfies all conditions listed in \cite[\S~A.5]{CD-Springer}.
\end{thm}
\begin{proof}
  The fact that the functor $X \mapsto \dm_\cdh(X, \Lambda)$ on $\Nsch_{\kappa}$ admits
  the six functors having the properties listed as conditions (1) to (6) in
  \cite[\S~A.5]{CD-Springer} are stated in \cite[\S~1.6, 6.1]{CD-Doc} most of which
  were already shown in \cite[Thm.~11.4.5]{CD-Springer}. The absolute purity
  is shown in Proposition~6.2 and the duality is shown in Theorem~7.3 of
  \cite{CD-Doc}. The continuity and localization axioms are shown in
  Theorem~5.11 of op. cit.. 

  To prove the existence of $f_\sharp$, we let $f \colon X \to Y$ be a
  separated and finite type morphism. It is then well known that the functor
  $f^* \colon \un{\dm}_\cdh(Y, \Lambda) \to \un{\dm}_\cdh(X, \Lambda)$ admits
  a left adjoint $f_{\sharp} \colon \un{\dm}_\cdh(X, \Lambda) \to
  \un{\dm}_\cdh(Y, \Lambda)$
  which is characterized by the property that for any $Z \in \Sch_X$ with the
  structure map $g \colon Z \to X$, one has $f_{\sharp}(M_X(Z)) = M_Y(Z)$.
  Here, $M_X(Z) \in \un{\dm}_\cdh(X, \Lambda)$ is the motive of $Z$ when the
  latter is considered as an object of $\Sch_X$ via $g$ and
  $M_Y(Z) \in \un{\dm}_\cdh(Y, \Lambda)$ is the motive of $Z$ when the
  latter is considered as an object of $\Sch_Y$ via $f \circ g$. Note that this
  is where one needs $X$ to be separated and finite type over $Y$.

  If $f$ is smooth, then $f_\sharp$ as defined above carries $\dm_\cdh(X, \Lambda)$ to
  $\dm_\cdh(Y, \Lambda)$ because $\dm_\cdh(X, \Lambda)$ is generated by the
  (twists and shifts of) motives of the objects of $\Sm_X$, and the compositions of
  such objects with $f$ define objects of $\Sm_Y$ whose motives lie in
  $\dm_\cdh(Y, \Lambda)$.
  If $Y = \Spec(k)$ with $k$ a field, then we have the functors
  \[
  \dm_\cdh(X, \Lambda) \inj \un{\dm}_\cdh(X, \Lambda) \xrightarrow{f_\sharp}
  \un{\dm}_\cdh(Y, \Lambda) \hookleftarrow {\dm}_\cdh(Y, \Lambda).
  \]
  Since the last inclusion on the right is an equivalence by
  \cite[Prop.~8.1(c)]{CD-Doc}, we see that $f_\sharp$ takes $\dm_\cdh(X, \Lambda)$
  to ${\dm}_\cdh(Y, \Lambda)$ in this case too.  The first claim of part (2) of the
  theorem is shown in Corollary~5.9 of op. cit. and the second claim follows by
  combining it with part (1). This concludes the proof.
\end{proof}

\begin{comment}
Using \cite[Lem.~3.17, Prop.~7.1]{CD-Doc} (see also \cite[Prop.~9.1.14]{CD-Springer})
and the continuity axiom of \thmref{thm:CD-cdh},
we have the following result in which we follow the notations of
\S~\ref{sec:Cont}. 

\begin{prop}\label{prop:CD}
  Let $k$ be as in \S~\ref{sec:Cont} and $X \in \Sch_k$. For an integer $i \ge 0$, let
  $f_i \colon X^\hs \to X_i, \ f \colon X^\hs \to X$ and $g_i \colon X_i \to X$ be the
  projection maps. Then we  have the following.
  \begin{enumerate}
  \item
Each of the compositions
    \[
      \dm_\cdh(X, \Lambda) \xrightarrow{g^*_i} \dm_\cdh(X_i, \Lambda)
      \xrightarrow{(g_i)_*}  \dm_\cdh(X, \Lambda) \ \mbox{and} 
\]
\[
  \dm_\cdh(X_i, \Lambda) \xrightarrow{(g_i)_*}     
\dm_\cdh(X, \Lambda) \xrightarrow{g^*_i} \dm_\cdh(X_i, \Lambda)
\]
is multiplication by a power of $p$.
\item
  There is an equivalence of monoidal triangulated categories
  \[
    2-{\varinjlim}_i \  \dm_\cdh(X_i, \Lambda) \xrightarrow{\simeq}  
    \dm_\cdh(X^{\heartsuit}, \Lambda),
  \]
  where the left hand side is a colimit of 2-categories.
In particular, the functors
\[
  \dm_\cdh(X, \Lambda) \xrightarrow{g^*_i} \dm_\cdh(X_i, \Lambda) \xrightarrow{f^*_i}
  \dm_\cdh(X^{\heartsuit}, \Lambda)
 \]
 are equivalences of monoidal triangulated categories.
\end{enumerate}
\end{prop}
\end{comment}

We now let $k$ be a field containing $\kappa$. 
For a closed immersion $\iota \colon W \inj X$ with structure maps
$f \colon X \to \Spec(k)$ and $h = f \circ \iota$ in $\Sch_k$, consider the following
objects in $\dm_\cdh(k, \Lambda) \simeq \dm(k, \Lambda)$.
\begin{equation}\label{eqn:Motive-def}
  M_k(X) = f_\sharp \Lambda_X, \
  M^c_k(X) = f_* f^! \Lambda_k \ \mbox{and} \
  M^W_k(X) = h_! \iota^* f^! \Lambda_k,
\end{equation}
where all functors are considered on the premotivic category
$X \mapsto \dm_\cdh(X, \Lambda)$. Recall that $\Z_k(1) =
{\Z_{\rm tr}(\P^1_k)}/{\Z_{\rm tr}(\Spec(k))}$ as a presheaf with transfer
via the identification of $\Spec(k)$ with the residue field of $\infty \in \P^1_k$.
In particular, $\Lambda_k(1) \cong M_k({\P^1_k}/{\Spec(k)}) \in  \dm(k, \Lambda)$. 
We shall drop the field $k$ from the notations of motives of objects in $\Sch_k$ once
it is fixed in a context.

\begin{lem}$($Cisinski-D{\'e}glise, \cite{CD-Doc}$)$\label{lem:Motive-0}
  There is a canonical isomorphism $M(X) \cong f_{!} f^{!} \Lambda_k$ in
  $\dm_\cdh(k, \Lambda)$. In particular, $M^c(X) \cong M(X)$ if $X$ is complete.
\end{lem}
\begin{proof}
  This is proven in \cite[(8.7.1)]{CD-Doc} but we reproduce the proof here to show
  that it does not require perfectness assumption on $k$.
  We first note that for any object $M \in \dm_\cdh(k, \Lambda)$, we have
  the functorial isomorphisms (of internal homs) 
  \begin{equation}\label{eqn:Motive-0-0}
    \begin{array}{lll}
      \un{\Hom}_{\dm_\cdh(k, \Lambda)}(M(X), M) & \cong & 
      \un{\Hom}_{\dm_\cdh(k, \Lambda)}(f_{\sharp} \Lambda_X, M) \\
      & \cong & f_* \un{\Hom}_{\dm_\cdh(X, \Lambda)}(\Lambda_X, f^* M) \\
 & \cong & f_* \un{\Hom}_{\dm_\cdh(X, \Lambda)}(f^* \Lambda_k, f^* M) \\
      & \cong & \un{\Hom}_{\dm_\cdh(k, \Lambda)}(\Lambda_k, f_*f^* M) \\
      & \cong &  f_*f^* M, \\
    \end{array}
  \end{equation}
  where all isomorphisms follow using the adjointness of pairs $(f_\sharp, f^*)$
  and $(f^*, f_*)$ as well as the isomorphism $f^* \Lambda_k \cong \Lambda_X$.

  By the adjointness of $(f_{!}, f^{!})$, we have the isomorphism
  $\un{\Hom}_{\dm_\cdh(k, \Lambda)}(f_{!} f^{!} \Lambda_k, \Lambda_k) \cong 
  f_* \un{\Hom}_{\dm_\cdh(X, \Lambda)}(f^{!} \Lambda_k, f^{!} \Lambda_k)$.
  Combining this with ~\eqref{eqn:Motive-0-0} and the duality isomorphism
  $\Lambda_X \cong \un{\Hom}_{\dm_\cdh(X, \Lambda)}(f^{!} \Lambda_k, f^{!} \Lambda_k)$ (cf.
  \cite[Thm.~7.3]{CD-Doc}), we get
  \[
  \un{\Hom}_{\dm_\cdh(k, \Lambda)}(f_{!} f^{!} \Lambda_k, \Lambda_k) \cong f_* \Lambda_X
  \cong f_* f^* \Lambda_k \cong \un{\Hom}_{\dm_\cdh(k, \Lambda)}(M(X), \Lambda_k).
  \]
  Since the natural map $M \to \un{\Hom}_{\dm_\cdh(k, \Lambda)}(\un{\Hom}_{\dm_\cdh(k, \Lambda)}
  (M, \Lambda_k), \Lambda_k)$ is an isomorphism by \cite[Thm.~7.3]{CD-Doc},
  we get $\un{\Hom}_{\dm_\cdh(k, \Lambda)}(f_{!} f^{!} \Lambda_k, M)
  \cong \un{\Hom}_{\dm_\cdh(k, \Lambda)}(M(X), M)$ for every object $M \in
  \dm_\cdh(k, \Lambda)$. But this clearly implies 
  $f_{!} f^{!} \Lambda_k \cong M(X)$.
\end{proof}

\begin{defn}\label{defn:MC-defn}
Given a closed immersion $\iota \colon W \inj X$ in $\Sch_k$, we let 
\begin{equation}\label{eqn:Coh-defn**}
  H_i(X, \Lambda(j)) = \Hom_{\dm_\cdh(k, \Lambda)}(\Lambda_k(j)[i], M(X));
\end{equation}
\[
H^i(X, \Lambda(j)) = \Hom_{\dm_\cdh(k, \Lambda)}(M(X), \Lambda_{k}(j)[i]);
\]
\[
  H^i_c(X, \Lambda(j)) = \Hom_{\dm_\cdh(k, \Lambda)}(M^c(X),  \Lambda_{k}(j)[i]);
\]
\[
  H^i_W(X, \Lambda(j)) =   \Hom_{\dm_\cdh(k, \Lambda)}(M^W(X),  \Lambda_{k}(j)[i]).
\]
\end{defn}
The above groups are respectively called the (motivic) homology, cohomology,
cohomology with compact support and cohomology with support in $W$  of $X$.
If $X = \Spec(A)$ is affine, we shall often use
the notation $A$ instead of $X$ while using these groups.

\begin{lem}\label{lem:Ex-seq}
  Given a closed immersion $\iota \colon W \inj X$ in $\Sch_k$ with the inclusion
  $u \colon U \inj X$ of the complement, there are long exact sequences
  \[
    (1) \hspace*{1cm}     \cdots \to H^i_W(X, \Lambda(j)) \xrightarrow{\iota_*}
    H^i(X, \Lambda(j)) 
    \xrightarrow{u^*}  H^i(U, \Lambda(j)) \xrightarrow{\partial}
    H^{i+1}_W(X, \Lambda(j)) \to \cdots .
  \]
  \[
    (2) \hspace*{1cm}   \cdots \to H^i_c(U, \Lambda(j)) \xrightarrow{u_*}
    H^i_c(X, \Lambda(j)) 
    \xrightarrow{\iota^*}  H^i_c(W, \Lambda(j)) \xrightarrow{\partial}
    H^{i+1}_c(U, \Lambda(j)) \to \cdots .   
  \]
 \end{lem}
\begin{proof}
  We let $f \colon X \to \Spec(k)$ denote the structure map of $X$.
  We let $h = f \circ \iota$ and $g = f \circ u$.
We have a distinguished triangle
$\iota_* \iota^! f^! \Lambda_k \to f^! \Lambda_k \to u_* u^* f^! \Lambda_k$
   in $\dm_\cdh(X, \Lambda)$ by a combination of \cite[Prop.~1.4.9]{Ayoub-1}
   and \cite[Thm.~5.1]{CD-Doc}. Using the isomorphism $u^* \cong u^!$ and
   subsequently applying $f_*$, we get a distinguished triangle
   \begin{equation}\label{eqn:Ex-seq-0}
     M^c(W) \to M^c(X) \to M^c(U)
   \end{equation}
   in $\dm_\cdh(k, \Lambda)$.
   We also have a distinguished triangle
   $u_! u^* f^! \Lambda_k \to f^! \Lambda_k \to \iota_* \iota^* f^! \Lambda_k$
   in $\dm_\cdh(X, \Lambda)$ by \cite[Thm.~5.11]{CD-Doc} (see also p.~147 of op. cit.).
   %and \cite[Lem.~2.3.5]{CD-Springer}.
We again proceed as in the previous case to get a distinguished triangle
   \begin{equation}\label{eqn:Ex-seq-1}
     M(U) \to M(X) \to M^W(X)
   \end{equation}
   in $\dm_\cdh(k, \Lambda)$. The distinguished triangles of motives
   ~\eqref{eqn:Ex-seq-0} and ~\eqref{eqn:Ex-seq-1} yield the long exact sequences
   of cohomology in the standard way.
   \end{proof}

We shall use the following description of motivic cohomology groups.  

\begin{lem}\label{lem:MC-et}
  In the notations of \lemref{lem:Ex-seq},
  there are canonical isomorphisms
  \[
    H^i(X, \Lambda(j)) \cong \Hom_{\dm_\cdh(k, \Lambda)}(\Lambda_k, f_* \Lambda_{X}(j)[i]);
\]
\[
  H^i_c(X, \Lambda(j)) \cong  \Hom_{\dm_\cdh(k, \Lambda)}(\Lambda_k,  f_! \Lambda_{X}(j)[i]);
\]
\[
H^i_W(X, \Lambda(j)) \cong
\Hom_{\dm_\cdh(X, \Lambda)}(\iota_* \Lambda_W,  \Lambda_{X}(j)[i]).
\] 
\end{lem}
\begin{proof}
  This is well known to experts and we give a sketch. We shall use the notations
  and distinguished triangles which appeared in the proof of \lemref{lem:Ex-seq}.
  For brevity, we shall use the notation $\sC_k$ (resp. $\Lambda^i_j$) for
  $\dm_\cdh(k, \Lambda)$ (resp. $\Lambda_k(j)[i])$ throughout the proof.
  We choose a compactification $v \colon X \inj \ov{X}$ of $X$ (cf. \cite{Nagata}) and
  let $\iota' \colon Z \inj \ov{X}$ be the complementary reduced closed subscheme.
  We let ${f'} \colon \ov{X} \to \Spec(k)$ and
  ${h'} = {f'} \circ \iota' \colon Z \to \Spec(k)$
  denote the structure maps.
 The first isomorphism follows directly from ~\eqref{eqn:Motive-0-0} if we take
  $M = \Lambda^i_j$ and use external homs instead of internal homs. 

  For the second isomorphism, one first applies the functorial isomorphism of
  ~\eqref{eqn:Motive-0-0} to $X$ and $Z$ to get a commutative diagram of motives
  \begin{equation}\label{eqn:MC-et-0}
    \xymatrix@C1pc{
      \un{\Hom}_{\sC_k}(M(\ov{X}), \Lambda^i_j) \ar[r]^-{\cong} \ar[d] &
      \un{\Hom}_{\sC_k}(\Lambda_k, f'_* \Lambda_{\ov{X}}(j)[i]) \ar[d] \\
\un{\Hom}_{\sC_k}(M(Z), \Lambda^i_j) \ar[r]^-{\cong} &
\un{\Hom}_{\sC_k}(\Lambda_k, h'_* \Lambda_Z(j)[i]),}
  \end{equation}
  in which the horizontal arrows are isomorphisms and
  the vertical arrows are obtained by the canonical maps
  $M(Z) \to M(\ov{X})$ and $f'_* \Lambda_{\ov{X}} \to h'_* \Lambda_Z$ in $\sC_k$.

  Considering the homotopy fibers of the two columns in ~\eqref{eqn:MC-et-0}
  and using the distinguished triangles 
  \[
    M(Z) \to M(\ov{X}) \to M^c(X); \ \
    f_! \Lambda_X \to f'_* \Lambda_{\ov{X}} \to h'_* \Lambda_Z,
  \]
  one gets a canonical map $\theta_X \colon \un{\Hom}_{\sC_k}(M^c(X), \Lambda^i_j) \to
  \un{\Hom}_{\sC_k}(\Lambda_k, f_{!} \Lambda_X(j)[i])$ which fits into the
  commutative diagram of distinguished triangles
  \begin{equation}\label{eqn:MC-et-1}
    \xymatrix@C.8pc{
      \un{\Hom}_{\sC_k}(M^c(X), \Lambda^i_j) \ar[r] \ar[d]_-{\theta_X} &
     \un{\Hom}_{\sC_k}(M(\ov{X}), \Lambda^i_j) \ar[r] \ar[d]^-{\cong}_-{\theta_{\ov{X}}} & 
     \un{\Hom}_{\sC_k}(M(Z), \Lambda^i_j) \ar[d]^-{\cong}_-{\theta_Z} \\
     \un{\Hom}_{\sC_k}(\Lambda_k, f_{!} \Lambda_X(j)[i]) \ar[r] &
     \un{\Hom}_{\sC_k}(\Lambda_k, f'_* \Lambda_{\ov{X}}(j)[i]) \ar[r] &
     \un{\Hom}_{\sC_k}(\Lambda_k, h'_* \Lambda_Z(j)[i]).}
  \end{equation}
It follows that $\theta_X$ is an isomorphism.

Applying the external hom functor $\Hom_{\sC_k}(\Lambda_k, -)$ to $\theta_X$
  and using the isomorphisms
  $\Hom_{\sC_k}(\Lambda_k, \un{\Hom}_{\sC_k}(M^c(X), \Lambda^i_j))
  \cong \Hom_{\sC_k}(M^c(X), \Lambda^i_j)$ and
  $\Hom_{\sC_k}(\Lambda_k, \un{\Hom}_{\sC_k}(\Lambda_k, A))
  \cong \Hom_{\sC_k}(\Lambda_k, A)$ (where we write $A$ as shorthand for
  $f_{!} \Lambda_X(j)[i]$),
  obtained via the adjointness of $(\otimes_{\Lambda_k}, \un{\Hom}_{\sC_k})$
  (cf. \cite[\S~A.5.1 (1)]{CD-Springer}), we get an isomorphism of abelian groups
  $\theta'_X \colon \Hom_{\sC_k}(M^c(X), \Lambda^i_j) \xrightarrow{\cong}
  \Hom_{\sC_k}(\Lambda_k, A)$ and this proves the second
  isomorphism.

To prove the third isomorphism, we look at the commutative diagram
  of motives
\begin{equation}\label{eqn:MC-et-2}
    \xymatrix@C1pc{
\un{\Hom}_{\sC_k}(M({X}), \Lambda^i_j) \ar[r]^-{\cong} \ar[d] &
      \un{\Hom}_{\sC_k}(\Lambda_k, f_* \Lambda_{X}(j)[i]) \ar[d] \\
\un{\Hom}_{\sC_k}(M(U), \Lambda^i_j) \ar[r]^-{\cong} &
\un{\Hom}_{\sC_k}(\Lambda_k, g_* \Lambda_U(j)[i]),}
  \end{equation}
in which the horizontal arrows are isomorphisms (cf. ~\eqref{eqn:Motive-0-0}) and
  the vertical arrows are obtained by the canonical maps
  $M(U) \to M(X)$ and $f_* \Lambda_{X} \to g_* \Lambda_U$ in $\sC_k$.

 Considering the homotopy fibers of the two columns in ~\eqref{eqn:MC-et-2}
  and using the distinguished triangles 
  \[
    M(U) \to M({X}) \to M^W(X); \ \
    h_* \iota^{!}\Lambda_X \to f_* \Lambda_{X} \to g_* \Lambda_U,
  \]
one gets a canonical map $\theta^W_X \colon \un{\Hom}_{\sC_k}(M^W(X), \Lambda^i_j) \to
  \un{\Hom}_{\sC_k}(\Lambda_k, h_* \iota^{!}\Lambda_X(j)[i])$ which fits into the
  commutative diagram of distinguished triangles
  \begin{equation}\label{eqn:MC-et-3}
    \xymatrix@C.8pc{
      \un{\Hom}_{\sC_k}(M^W(X), \Lambda^i_j) \ar[r] \ar[d]_-{\theta^W_X} &
     \un{\Hom}_{\sC_k}(M(X), \Lambda^i_j) \ar[r] \ar[d]^-{\cong}_-{\theta_X} & 
     \un{\Hom}_{\sC_k}(M(U), \Lambda^i_j) \ar[d]^-{\cong}_-{\theta_U} \\
     \un{\Hom}_{\sC_k}(\Lambda_k, h_* \iota^{!} \Lambda_X(j)[i]) \ar[r] &
     \un{\Hom}_{\sC_k}(\Lambda_k, f_* \Lambda_{{X}}(j)[i]) \ar[r] &
     \un{\Hom}_{\sC_k}(\Lambda_k, g_* \Lambda_U(j)[i]).}
  \end{equation}
  It follows that $\theta^W_X$ is an isomorphism. We now argue exactly as in the
  proof of the second isomorphism to conclude that $\theta^W_X$ induces an
  isomorphism of abelian groups $(\theta^W_X)' \colon {\Hom}_{\sC_k}(M^W(X), \Lambda^i_j)
  \xrightarrow{\cong} {\Hom}_{\sC_k}(\Lambda_k, h_* \iota^{!} \Lambda_X(j)[i])$.

  To finish the proof, we now use the isomorphisms
\begin{equation}\label{eqn:MC-et-4}
  \begin{array}{lll}
    {\Hom}_{\sC_k}(\Lambda_k, h_* \iota^{!} \Lambda_X(j)[i]) & \cong &
    {\Hom}_{\dm_\cdh(W, \Lambda)}(h^*\Lambda_k, \iota^{!} \Lambda_X(j)[i]) \\
    & \cong &  {\Hom}_{\dm_\cdh(W, \Lambda)}(\Lambda_W, \iota^{!} \Lambda_X(j)[i]) \\
    & \cong &  {\Hom}_{\dm_\cdh(X, \Lambda)}(\iota_{*} \Lambda_W, \Lambda_X(j)[i]), \\
\end{array}
\end{equation}
where the last isomorphism uses the adjoint pair $(\iota_{!}, \iota^{!})$
and the identity $\iota_{!} = \iota_*$ (as $\iota$ is proper).
The proof of the lemma is now complete.
\end{proof}

\begin{cor}\label{cor:MC-et-5}
  Let $X, W \in \Sch_k$ be as in \lemref{lem:Ex-seq}. If $X$ and $W$ are regular, then
  \[
  H^i(X, \Lambda(j)) \cong \Hom_{\dm(k, \Lambda)}(\Lambda_k, f_* \Lambda_{X}(j)[i]);
\]
\[
H^i_c(X, \Lambda(j)) \cong {\Hom}_{\dm(k, \Lambda)}(\Lambda_k,
  f_{!}\Lambda_X(j)[i]);
  \]
  \[
H^i_W(X, \Lambda(j)) \cong
\Hom_{\dm(X, \Lambda)}(\iota_* \Lambda_W,  \Lambda_{X}(j)[i]).
\]
\end{cor}
\begin{proof}
 Combine \thmref{thm:CD-cdh} and \lemref{lem:MC-et}. 
%  Under our assumption, we can replace $\dm_\cdh(k, \Lambda)$ in \lemref{lem:MC-et}
  %by $\dm(k, \Lambda)$ using \thmref{thm:CD-cdh}.
%  and this explains the first and third isomorphisms.
%  To prove the second, note that our assumption implies by
%  \cite[Cor.~3.2(2)]{CD-Doc} that $f_{!}$ satisfies the
%  projection formula, i.e.,
% $f_{!} (\Lambda_X(j)[i]) = f_{!}(\Lambda_X \otimes  f^*(\Lambda_k(j)[i]))
%  \cong (f_{!} \Lambda_X)(j)[i]$ so that we have
%  \[
%  {\Hom}_{\dm(k, \Lambda)}(\Lambda_k, f_{!}\Lambda_X(j)[i]) \cong
%  {\Hom}_{\dm(k, \Lambda)}(\Lambda_k(-j)[-i], f_{!}\Lambda_X).
%  \]
%  This concludes the proof.
  \end{proof}

\subsection{Motivic cohomology via stable homotopy category}\label{sec:MSH}
For our purpose, we need to represent the motivic cohomology groups (and their
variants) for regular $\kappa$-schemes in their motivic stable homotopy categories
{\`a} la Voevodsky (cf. \cite{Ayoub-1},
\cite{Ayoub}). We recall from op. cit. that for a scheme $X \in \Nsch_\kappa$,
$\sh(X)$ is the monoidal triangulated category of $T$-spectra over $X$
obtained by the $T$-stabilization of the $\A^1$-localization of the category of pointed
Nisnevich simplicial sheaves of sets on $\Sm_X$, where $\Sigma_s = (S^1, 1)$
(resp. $\Sigma_t = (\G_m, 1)$) is the simplicial (resp. Tate) circle and
$T = \Sigma_s \wedge \Sigma_t$. In this paper, we shall work with the localization
of the monoidal triangulated category $\sh(X)$ in which all hom-groups are endowed
with the structure of  $\Z[\tfrac{1}{p}]$-modules. From here onward, $\sh(X)$ will
always mean this localized category.

By \cite{Ayoub-1} and \cite{Ayoub},
$X \mapsto \sh(X)$ defines a presheaf of monoidal triangulated categories on
$\Nsch_\kappa$ which satisfies the six functor formalism.
Another property that we shall use is that $\sh(X)$ is a spectral category
(cf. \cite[\S~2.2]{DJK}), i.e., it is a category enriched in spectra. In particular,
there is an $S^1$-spectrum
${\rm Maps}_{\sh(X)}(A, B)$ for any $A, B \in \sh(X)$ such that
one has $\pi_n({\rm Maps}_{\sh(X)}(A, B)) \cong \Hom_{\sh(X)}(\Sigma^{n}_sA,B)$.
We let $\1_X = \Sigma^\infty_T X_{+}$ denote
the identity object for the monoidal structure of $\sh(X)$ and let $\Sigma^{i,j} = \Sigma^{i-j}_s \circ \Sigma^{j}_t$ as an automorphism of $\sh(X)$ for
$i, j \in \Z$.
%
%so that
%there is a unique monoidal morphism $\1_X \to \h\Lambda_X$ and
%$\Hom_{\sh(X)}(\1_X, E)$ is canonically isomorphic to $\Hom_{\sh(X)}(\h\Lambda_X, E)$
%for every $E \in \h\Lambda_{X}$-Mod, where the latter is the triangulated full
%subcategory of $\h\Lambda_{X}$-module spectra in $\sh(X)$. 
%We let

By \cite[\S~2.b]{CD-Doc}, forgetting the transfer structure and then taking the
$T$-suspension defines a functor $\phi_X \colon {\dm}(X, \Lambda) \to
{\sh}(X)$ (in op. cit., this functor is denoted by $\phi_*$) which has a strictly monoidal (cf. \cite[\S~11.2.16]{CD-Springer})
left adjoint $\psi_X \colon {\sh}(X) \to  {\dm}(X, \Lambda)$ (denoted by $\phi^*$ in op. cit.). 
%The latter is uniquely determined via the left Kan extension by the property that
%$\psi_X(\Sigma^\infty_T Y_+) = g_\sharp (\Lambda_Y)$
%for any $g \colon Y \to X$ in $\Sm_X$. 
In particular, $\psi_X(\1_X) \xrightarrow{\cong} \Lambda_X$. 
We let $\h\Lambda_{X} = \phi_X(\Lambda_X)$ and let
$\h\Lambda_{X}$-Mod be the monoidal triangulated subcategory of
 $\sh(X)$ consisting of module spectra over ring spectrum $\h\Lambda_{X}$.
%Since $\psi_X$ is a strictly monoidal, it follows that for $i,j \in \Z$, we have 
%\begin{equation}\label{eqn:prop-phi}
%\psi_X( \Sigma^{i, j}\1_X) = \Lambda_X(j)[i] \textnormal{ and } \phi_X(\Lambda_X(j)[i])  = \Sigma^{i, j} \h\Lambda_{X}. 
%\end{equation}
If $v \colon X \to Y$ is a morphism
in $\Nsch_\kappa$, then the exchange transformation $v^* \phi_Y \to \phi_X v^*$ induces
a canonical map $v^* \h\Lambda_{Y} \to \h\Lambda_{X}$. We recall the following result from 
\cite[Cor.~3.6]{CD-Doc}. 

\begin{thm}\label{thm:CD-5.11}
%  The functors $\phi_{(-)}$ and $\psi_{(-)}$ commute with $g^*$ and $g_*$ (resp. $h_!$ and $h^!$) for any morphism $g$ (resp. any separeted morphism of finite type $h$) between regular $\kappa$-schemes. 
%  Moreover, t
  The map
  $v^* \h\Lambda_{Y} \to \h\Lambda_{X}$ is an isomorphism if $v \colon X \to Y$
  is a morphism in $\Reg_\kappa$.
%{\footnote{Conjecturally, this is an isomorphism even if $X$ or $Y$ is not regular
%(cf. \cite[p.~157]{CD-Doc}).}}. 
\end{thm}
%\begin{proof}
%As shown in \cite[\S~2.b]{CD-Doc}, the monoidal structure of $\phi_X$ induces
%  a monoidal functor $t^* \colon \h\Lambda_{X}$-Mod $\to \dm(X, \Lambda)$ which has a
%  right adjoint $t_*$ having the property that $\phi_X = \sO_{\h\Lambda_{X}} \circ
%  t_*$, where $\sO_{\h\Lambda_{X}} \colon \h\Lambda_{X}$-Mod $\to \sh(X)$ is the
%  canonical inclusion (i.e., the forgetful functor).
%  Since $t_*$ is an equivalence of monoidal triangulated categories by Theorem~3.1 of
%  op. cit. if $X \in \reg_\kappa$, it follows that so is $\phi_X$.
%  The second part follows from Corollary~3.6 of op. cit. and this concludes the proof.
%\end{proof}

%\vskip .2cm

The following is well-known to experts.

\begin{lem}\label{lem:DM-SH-1}
  For $X \in \Reg_\kappa$ and $i, j \in \Z$, one has $\phi_X(\Lambda_X(j)[i]) \cong
  \Sigma^{i,j}\h\Lambda_{X}$.
\end{lem}
\begin{proof}
 As shown in \cite[\S~2.b]{CD-Doc}, there is a canonical
 monoidal functor $t_* \colon 
   \dm(X, \Lambda) \to \h\Lambda_{X}$-Mod such that $\phi_X = \sO_{\h\Lambda_{X}} \circ
  t_*$, where $\sO_{\h\Lambda_{X}} \colon \h\Lambda_{X}$-Mod $\to \sh(X)$ is the
  canonical inclusion (i.e., the forgetful functor). Furthermore,
  $t_*$ is an equivalence of monoidal triangulated categories.
  Hence, it suffices the prove that
  $t_*(\Lambda_X(j)[i]) \cong \Sigma^{i,j}\h\Lambda_{X}$.
Now, it is well known that $t_*(\Lambda_X(1)) = \Sigma^{-1}_s \wedge \Sigma_t
  \wedge \h\Lambda_{X}$.
  As the automorphism $\Sigma_s$ in $\h\Lambda_{X}$-Mod
  corresponds to the shift functor
  in $\dm(X, \Lambda)$ and $t_*$ is monoidal, one
  gets
  \[
  t_*(\Lambda_X(j)[i]) =  \Sigma^{i}_s \wedge (t_*(\Lambda_X(1)))^j\wedge \h\Lambda_{X}
  \cong \Sigma^{i}_s \wedge \Sigma^{-j}_s \wedge \Sigma^j_t \wedge \h\Lambda_{X}
  \cong \Sigma^{i,j}\h\Lambda_{X}.
  \]
  This proves the lemma.
  \end{proof}

We now let $k$ be a field containing $\kappa$.
Let $X \in \Sch_k$ with the structure map $f \colon X \to \Spec(k)$.
We let $\h\Lambda_{X/k} = f^* \h\Lambda_{k}$ and
$\h\Lambda_{X/{\kappa}} = f^*_\kappa \h\Lambda_{\kappa}$, where
$f_\kappa \colon X \xrightarrow{f} \Spec(k) \to \Spec(\kappa)$ is the composite map.
We then have the canonical maps $\h\Lambda_{X/{\kappa}} \xrightarrow{\cong}
\h\Lambda_{X/k} \to \h\Lambda_{X}$ in which the first map is isomorphism by
\thmref{thm:CD-5.11}.

For a closed immersion $\iota \colon W \inj X$ in $\Sch_k$ with structure map
$f \colon X \to \Spec(k)$, we let
\begin{equation}\label{eqn:Coh-defn}
  \begin{array}{lll}
    E^i(X, \Lambda(j)) & = & \Hom_{\sh(X)}(\1_X, \Sigma^{i,j} \h\Lambda_{X/k}); \\
    E^i_c(X, \Lambda(j)) & = & \Hom_{\sh(k)}(\1_k, \Sigma^{i,j} f_! \h\Lambda_{X/k}); \\
  E^i_W(X, \Lambda(j)) & = &\Hom_{\sh(X)}(\iota_* \1_W, \Sigma^{i,j}\h\Lambda_{X/k}).
  \end{array}
  \end{equation}

\begin{lem}\label{lem:SH-DM-Maps}
If $X, W \in \Sch_k$ are regular, there are canonical isomorphisms:
  \begin{enumerate}
  \item
$H^i(X, \Lambda(j)) \ {\cong} \ E^i(X, \Lambda(j))$.
\item
$H^i_c(X, \Lambda(j))  \ {\cong} \ E^i_c(X, \Lambda(j))$.
\item
  $H^i_W(X, \Lambda(j)) \ {\cong} \ E^i_W(X, \Lambda(j))$.
  \end{enumerate}
\end{lem}
\begin{proof}
  To prove (1), we compute
  \begin{equation}\label{eqn:SH-DM-Maps-0}
    \begin{array}{lll}
     H^i(X, \Lambda(j)) &{\cong}^1& \Hom_{\dm(k, \Lambda)}(\Lambda_k, f_* \Lambda_{X}(j)[i])\\
     %& {\cong} & \Hom_{\dm_{\cdh}(k, \Lambda)}(f_{\sharp} \Lambda_X,   \Lambda_k(j)[i]) \\
    % & \cong &  \Hom_{\dm_{\cdh}(X, \Lambda)}(\Lambda_X, f^* \Lambda_k(j)[i]) \\
     & {\cong} &  \Hom_{\dm(X, \Lambda)}(\Lambda_X,  \Lambda_X(j)[i]) \\
   %  & {\cong} & \Hom_{\dm(X, \Lambda)}(\Lambda_X, \Lambda_{X}(j)[i]) \\
     & {\cong} &  \Hom_{\dm(X, \Lambda)}(\psi_X(\1_X), \Lambda_{X}(j)[i])\\
    % & {\cong}^{2} & \Hom_{\sh(X)}(\h\Lambda_{X}, \Sigma^{i,j}\h\Lambda_{X}) \\
   %  & {\cong}^{3} & \Hom_{\sh(X)}(\h\Lambda_{X}, \Sigma^{i,j}\h\Lambda_{X/k}) \\ 
    & {\cong} &  \Hom_{\sh(X)}(\1_X, \phi_X(\Lambda_{X}(j)[i]))\\
     & {\cong}^2 & \Hom_{\sh(X)}(\1_X, \Sigma^{i,j}\h\Lambda_{X}) \\
     & {\cong}^3 & \Hom_{\sh(X)}(\1_X, \Sigma^{i,j}\h\Lambda_{X/k}) \\
     & = &  E^i(X, \Lambda(j)),
    \end{array}
    \end{equation}
  where ${\cong}^{1}$ follows from \corref{cor:MC-et-5}, ${\cong}^2$ follows from
  \lemref{lem:DM-SH-1} and ${\cong}^{3}$ follows from  \thmref{thm:CD-5.11}.

  To prove (2), we compute
  \begin{equation}\label{eqn:SH-DM-Maps-1}
    \begin{array}{lll}
      H^i_c(X, \Lambda(j)) & {\cong}^1 &
      \Hom_{\dm(k, \Lambda)}(\Lambda_k, f_{!} \Lambda_X(j)[i]) \\
      & {\cong} & \Hom_{\dm(k, \Lambda)}(\psi_k(\1_k), f_{!} \Lambda_X(j)[i]) \\
      & {\cong} & \Hom_{\sh(k)}(\1_k, \phi_k ( f_{!} \Lambda_X(j)[i] ) ) \\
      & {\cong}^2 & \Hom_{\sh(k)}(\1_k, f_{!}(\phi_X(\Lambda_X(j)[i]))) \\
       & {\cong}^3 & \Hom_{\sh(k)}(\1_k, f_{!}(\Sigma^{i,j}\h\Lambda_{X})) \\
%& {\cong} & \Hom_{\sh(k)}(\1_k,  \Sigma^{i,j} \phi_k f_{!} \Lambda_X) \\
      & {\cong}^4 & \Hom_{\sh(k)}(\1_k, \Sigma^{i,j} f_{!} \h\Lambda_X) \\
 & {\cong}^5 & \Hom_{\sh(k)}(\1_k, \Sigma^{i,j} f_{!} \h\Lambda_{X/k}) \\     
& = & E^i_c(X, \Lambda(j)),
    \end{array}
  \end{equation}
  where ${\cong}^1$ follows from \corref{cor:MC-et-5},  ${\cong}^2$ follows from
  \cite[Cor.~3.2(1)]{CD-Doc}, ${\cong}^3$ follows from \lemref{lem:DM-SH-1},
  ${\cong}^4$ follows from the projection formula for $f_{!}$ in $\sh(-)$
  (cf. \cite[Thm.~2.4.50(5)]{CD-Springer})  and
  ${\cong}^5$ follows from \thmref{thm:CD-5.11}.
 % follow from \thmref{thm:CD-5.11} and ${\cong}^3$ follows from  \cite[Cor.~3.2(1)]{CD-Doc}.
  
To prove (3), we compute
  \begin{equation}\label{eqn:SH-DM-Maps-2}
    \begin{array}{lll}
      H^i_W(X, \Lambda(j))  & {\cong}^1 &
      \Hom_{\dm(X, \Lambda)}(\iota_* \Lambda_W, \Lambda_X(j)[i]) \\
      & {\cong} &  \Hom_{\dm(X, \Lambda)}(\iota_*(\psi_W (\1_W)), \Lambda_X(j)[i]) \\
         & {\cong}^2 &  \Hom_{\dm(X, \Lambda)}(\psi_X( \iota_* \1_W), \Lambda_X(j)[i]) \\
      & {\cong} & \Hom_{\sh(X)}( \iota_* \1_W,
     \phi_X (\Lambda_X(j)[i]) ) \\
       & {\cong}^3 & \Hom_{\sh(X)}(\iota_* \1_W, \Sigma^{i,j}\h\Lambda_{X}) \\
      & {\cong}^4 & \Hom_{\sh(X)}(\iota_* \1_W, \Sigma^{i,j}\h\Lambda_{X/k}) \\
      & = & E^i_W(X, \Lambda(j)), 
    \end{array}
  \end{equation}
  where ${\cong}^1$ follows from \corref{cor:MC-et-5}, ${\cong}^{2}$ follows from  \cite[Cor.~3.2(1)]{CD-Doc},
 ${\cong}^{3}$  follows from \lemref{lem:DM-SH-1}, and ${\cong}^4$  follows from \thmref{thm:CD-5.11}.
\end{proof}

\subsection{Gysin homomorphisms for $E^*(-, \Lambda(\bullet))$}
\label{sec:Gysin**}
We let $k$ be a field containing $\kappa$.
Let $\iota \colon Z \inj X$ be a regular closed
immersion of pure codimension $n$ in $\Sch_k$. Let $W \subset Z$ be a closed subset.
Since $\h\Lambda_{{(-)}/{k}}$ is an oriented absolute spectrum,
it follows from \cite[Defn.~2.8]{Navarro} (see also \cite[Thm.~3.2.21]{DJK})
that there is a fundamental class $\eta^X_Z \in E^{2n}_Z(X, \Lambda(n))$
and a product
$\bigcup \colon E^i_W(Z, \Lambda(j)) \times E^{2n}_Z(X, \Lambda(n))
\to E^{2n+i}_W(X, \Lambda(n+j))$ (cf. \cite[Defn.~1.6]{Navarro}) which together give
rise to the refined Gysin homomorphism
\begin{equation}\label{eqn:Gysin-0}
\iota_* \colon  E^i_W(Z, \Lambda(j)) \to E^{2n+i}_W(X, \Lambda(n+j)); \ \
\alpha \mapsto \alpha \cup \eta^X_Z.
\end{equation}
This map is linear with respect to the action of the (relative) motivic cohomology
ring of $X$. Furthermore, for a flat morphism $f \colon X' \to X$ of $k$-schemes
with $Z' = Z \times_X X'$,
one has $\eta^{X'}_{Z'} = f^*(\eta^X_Z)$ (cf. \cite[Prop.~2.10]{Navarro}).
In particular, the refined Gysin map commutes with the flat
pull-back map $f^*$ on the cohomology of $Z$ and $X$ with support in $W$.

Recall that when $k$ is a perfect field and $X$ is a  regular $k$-scheme, then
there is a natural duality isomorphism
$\Phi_X \colon M^c(X)^*(d)[2d] \xrightarrow{\cong} M(X)$
(cf. \cite[Thm.~5.3.18(iii)]{Kelly}).
Furthermore, if $\iota \colon Z \inj X$ is a closed immersion of regular $k$-schemes of
pure codimension $n$, then one has a purity isomorphism $M_Z(X) \cong M(Z)(n)[2n]$
(cf. \cite[(2.4.39.1)]{CD-Springer}).

When $Z$ is moreover complete (e.g., a closed point of $X$),
the duality and purity isomorphisms yield the Gysin homomorphism
$\iota_* \colon H^i(Z, \Lambda(j)) \to H^{2n+i}_c(X, \Lambda(n+j))$ which is dual to the
push-forward homomorphism $\iota_* \colon H_{d-n-i}(Z, \Lambda(d-n-j)) \to
H_{d-n-i}(X, \Lambda(d-n-j))$. In other  words, there is a commutative diagram
\begin{equation}\label{eqn:Gysin-PF}
  \xymatrix@C1pc{
    H^i(Z, \Lambda(j)) \ar[r]^-{\Phi_Z}_-{\cong} \ar[d]_-{\iota_*} &
    H_{2d-2n-i}(Z, \Lambda(d-n-j)) \ar[d]^-{\iota_*} \\
    H^{2n+i}_c(X, \Lambda(n+j)) \ar[r]^-{\Phi_X}_-{\cong} &
    H_{2d-2n-i}(X, \Lambda(d-n-j)).}
\end{equation}

We shall need a version of Gysin map when $k$ is not perfect and 
$Z$ is not regular. The purity theorem does not hold in this case. However, the
fundamental class and refined Gysin map of ~\eqref{eqn:Gysin-0} allow us
to show the existence of the Gysin homomorphism for $E$-theory without assuming
purity.
This is the context of the following result.

\begin{lem}\label{lem:Gysin-CS}
  Let $\iota \colon Z \inj X$ be a regular closed immersion of pure codimension $n$ in
  $\Sch_k$. Assume that $Z$ is complete. Then there is a natural Gysin homomorphism
  \[
    \iota_* \colon E^i(Z, \Lambda(j)) \to E^{2n+i}_c(X, \Lambda(n+j)).
  \]
\end{lem}
\begin{proof}
By the construction of the refined Gysin homomorphism, it suffices to show that the
  forget support map $E^i_Z(X, \Lambda(j)) \to E^i(X, \Lambda(j))$ factors through
  $E^i_c(X, \Lambda(j))$.
  To show this, we choose a compactification $\ov{X}$ of $X$ and let $u \colon X \inj
  \ov{X}$ be the inclusion. Let $\iota' \colon W = \ov{X} \setminus X \inj \ov{X}$ be
  the inclusion of the complementary closed subset, endowed with the reduced subscheme
  structure. We let $\tau = u \circ \iota \colon Z \inj \ov{X}$. Let
  $f \colon X \to \Spec(k)$ and $f' \colon \ov{X} \to \Spec(k)$ be the structure maps.

Now, we note that
\[
\begin{array}{lll}
\Hom_{\sh(X)}(\iota_* \1_Z, \Sigma^{i,j} \h\Lambda_{X/k}) 
& \cong & 
\Hom_{\sh(X)}(\iota_! \1_Z, \Sigma^{i,j} \h\Lambda_{X/k}) \\
& \cong & 
\Hom_{\sh(X)}(\iota_! \1_Z,  u^* \Sigma^{i,j} \h\Lambda_{\ov{X}/k}) \\
& \cong & \Hom_{\sh(X)}(\iota_! \1_Z, u^! \Sigma^{i,j} \h\Lambda_{\ov{X}/k}) \\
& \cong & \Hom_{\sh(\ov{X})}(u_! \iota_!  \1_Z, 
\Sigma^{i,j} \h\Lambda_{\ov{X}/k}) \\
& \cong & \Hom_{\sh(\ov{X})}(\tau_!  \1_Z, 
\Sigma^{i,j} \h\Lambda_{\ov{X}/k}) \\
& \cong & \Hom_{\sh(\ov{X})}(\tau_*  \1_Z, 
\Sigma^{i,j} \h\Lambda_{\ov{X}/k}),
\end{array}
\]
where the last isomorphism uses $\tau_! \cong \tau_*$, thanks to the completeness of
$Z$. It follows that there is a canonical (excision) isomorphism
\begin{equation}\label{eqn:Gysin-CS-0*}
E^i_Z(\ov{X}, \Lambda(j)) \xrightarrow{\cong} E^i_Z(X, \Lambda(j)).
\end{equation}
Note also that the map
$E^i_Z(\ov{X}, \Lambda(j)) \to E^i(X, \Lambda(j))$ factors through
$E^i_Z(\ov{X}, \Lambda(j)) \to E^i(\ov{X}, \Lambda(j))$.
Hence, we are reduced to showing that the map
$E^i_Z(\ov{X}, \Lambda(j)) \to E^i(\ov{X}, \Lambda(j))$ factors through
$E^i_c({X}, \Lambda(j))$.

Next, we claim that the composite
    \begin{equation}\label{eqn:Gysin-CS-1}
{\rm Maps}_{\sh(\ov{X})}(\tau_! \1_{Z}, \Sigma^{i,j}\h\Lambda_{\ov{X}/k})
    \to {\rm Maps}_{\sh(\ov{X})}(\1_{\ov{X}}, \Sigma^{i,j}\h\Lambda_{\ov{X}/k})
    \hspace*{3cm}
  \end{equation}
  \[
\hspace*{7cm}
      \to {\rm Maps}_{\sh(\ov{X})}(\1_{\ov{X}}, \iota'_*\Sigma^{i,j}\h\Lambda_{W/k})
    \]
    is null-homotopic, where the first arrow is induced by the adjunction
    $\1_{\ov{X}} \to \tau_* \1_Z \cong \tau_! \1_Z$.

    To prove the claim, we note that there is a commutative diagram
    \begin{equation}\label{eqn:Gysin-CS-2}
      \xymatrix@C1pc{
{\rm Maps}_{\sh(\ov{X})}(\tau_! \1_{Z}, \Sigma^{i,j}\h\Lambda_{\ov{X}/k})
        \ar[r] \ar[d] & 
{\rm Maps}_{\sh(\ov{X})}(\tau_! \1_{Z}, \iota'_*\Sigma^{i,j}\h\Lambda_{W/k})
        \ar[d] \\
        {\rm Maps}_{\sh(\ov{X})}(\1_{\ov{X}}, \Sigma^{i,j}\h\Lambda_{\ov{X}/k})
        \ar[r] & 
{\rm Maps}_{\sh(\ov{X})}(\1_{\ov{X}}, \iota'_*\Sigma^{i,j}\h\Lambda_{W/k}).}
\end{equation}
It suffices therefore to show that
    ${\rm Maps}_{\sh(\ov{X})}(\tau_! \1_{Z}, \iota'_*\Sigma^{i,j}\h\Lambda_{W/k})$
    is contractible.

    Finally, we observe that there are canonical weak equivalences of
   spectra
    \[
      \begin{array}{lll}
{\rm Maps}_{\sh(\ov{X})}(\tau_! \1_{Z}, \iota'_*\Sigma^{i,j}\h\Lambda_{W/k}) & \cong &
{\rm Maps}_{\sh(Z)}(\1_{Z}, \tau^! \iota'_*\Sigma^{i,j}\h\Lambda_{W/k}) \\
      & \cong &
      {\rm Maps}_{\sh(Z)}(\1_{Z}, \beta_* \alpha^!\Sigma^{i,j}\h\Lambda_{W/k}),
      \end{array}
    \]
    where $\alpha \colon W \times_{\ov{X}} Z \to W$ and $\beta \colon 
    W \times_{\ov{X}} Z \to Z$ are the canonical maps. The claim now follows because
    $W \times_{\ov{X}} Z = \emptyset$.
Using the claim and the distinguished triangle of spectra (cf. \cite[Prop.~2.2.10]{DJK})

\vskip .2cm

\begin{equation}\label{eqn:Gysin-CS-0}
    {\rm Maps}_{\sh(\ov{X})}(\1_{\ov{X}}, u_! \Sigma^{i,j}\h\Lambda_{X/k})
    \to {\rm Maps}_{\sh(\ov{X})}(\1_{\ov{X}}, \Sigma^{i,j}\h\Lambda_{\ov{X}/k})
    \hspace*{3cm}
  \end{equation}
  \[
\hspace*{8cm}
      \to {\rm Maps}_{\sh(\ov{X})}(\1_{\ov{X}}, \iota'_*\Sigma^{i,j}\h\Lambda_{W/k}),
    \]  
    we see that the map
    \[
{\rm Maps}_{\sh(\ov{X})}(\tau_! \1_{Z}, \Sigma^{i,j}\h\Lambda_{\ov{X}/k})
\to {\rm Maps}_{\sh(\ov{X})}(\1_{\ov{X}}, \Sigma^{i,j}\h\Lambda_{\ov{X}/k})
\]
factors through
\begin{equation}\label{eqn:Gysin-CS-4}
  {\rm Maps}_{\sh(\ov{X})}(\tau_! \1_{Z}, \Sigma^{i,j}\h\Lambda_{{\ov{X}}/k}) \to
  {\rm Maps}_{\sh(\ov{X})}(\1_{\ov{X}}, u_! \Sigma^{i,j}\h\Lambda_{X/k}).
\end{equation}
Considering the induced maps on $\pi_0$ and noting that
\[
\begin{array}{lll}
  \pi_0({\rm Maps}_{\sh(\ov{X})}(\1_{\ov{X}}, u_! \Sigma^{i,j}\h\Lambda_{X/k})) & = &
  \Hom_{\sh(\ov{X})}(\1_{\ov{X}}, u_! \Sigma^{i,j}\h\Lambda_{X/k})
  \\
  & \cong &
  \Hom_{\sh(\ov{X})}({f'}^*(\1_k), u_! \Sigma^{i,j}\h\Lambda_{X/k}) \\
  & \cong & \Hom_{\sh(k)}(\1_{k}, f'_* u_! \Sigma^{i,j}\h\Lambda_{X/k}) \\
  & \cong & \Hom_{\sh(k)}(\1_{k}, f'_{!} u_! \Sigma^{i,j}\h\Lambda_{X/k}) \\
  & \cong & \Hom_{\sh(k)}(\1_{k}, ({f'} \circ u)_! \Sigma^{i,j}\h\Lambda_{X/k}) \\
  & \cong & \Hom_{\sh(k)}(\1_{k}, {f}_! \Sigma^{i,j}\h\Lambda_{X/k}) = E^i_c(X, \Lambda(j)),
\end{array}
\]
we get the  desired factorization
\begin{equation}\label{eqn:Gysin-CS-3}
  E^i_Z(\ov{X}, \Lambda(j)) \to E^i_c(X, \Lambda(j)) \to E^i(\ov{X}, \Lambda(j)).
\end{equation}
This concludes the proof.
\end{proof}

\subsection{A functorial property of Gysin map}\label{sec:Gysin-PB*}
Let $k$ be a field containing $\kappa$.
One knows that the Gysin homomorphism commutes with the pull-back maps on
motivic cohomology in a transverse square. But we need to know the relation between
these two homomorphisms when the square is not Cartesian.
We consider a special case of this situation in this subsection.

We assume that $k$ is a perfect field. We first assume that 
$X \in \Sm_k$ is an integral scheme of dimension one. We let $Z \subset  X$ be a
closed subscheme such that $Z_\red$ is a closed point $x$ of $X$.
Let $S = \Spec(k(x))$ and let $S \xrightarrow{\iota} Z \xrightarrow{\tau} X$ be the
inclusions and let  $e = \ell(\sO(Z))$. 
%Write $v = \tau \circ \iota$ and $e = \ell(\sO(Z))$. 
Assume that $\tau$ is a regular closed immersion. Since $X$ is regular, $ \tau \iota$ is also a regular closed immersion.
We thus have a diagram
\begin{equation}\label{eqn:Gysin-1}
  \xymatrix@C1pc{
    E^i(Z, \Lambda(j)) \ar[r]^-{{\tau}_*} \ar[d]_-{{\iota}^*}^-{\cong} & 
    E^{2+i}_S(X, \Lambda(1+j)) \\
E^i(S, \Lambda(j)). \ar[ur]_-{( \tau  \iota)_*} & }
\end{equation}

\begin{lem}\label{lem:Gysin-2}
  Let $m$ be an integer prime to $p$ and $\Lambda = {\Z}/m$.
  Then ${\tau}_* = e ((\tau \iota)_* \circ {\iota}^*)$.
\end{lem}
\begin{proof}
Suppose that we know the equality $\eta^X_Z = e \eta^X_S$. 
 Then, for any $i, j$ and 
  $\alpha \in E^i(Z, \Lambda(j))$, we get
 \[
   {\tau}_*(\alpha) = \alpha \cup \eta^X_Z = {\iota}^*(\alpha) \cup \eta^X_Z =
   {\iota}^*(\alpha) \cup e \eta^X_S = e (\tau \iota)_*(h^*(\alpha)),
 \]
 where the second identity follows from the construction of the
 cup product in \cite[Defn.~1.6, p.~509]{Navarro} and the isomorphism (of the unit of
 adjunction) $\1_Z \xrightarrow{\cong} {\iota}_*\1_S$ in $\sh(Z)$.
 %Check this !!!!
It remains to show that $\eta^X_Z = e \eta^X_S$. By definition of the
 fundamental classes (cf. \cite[Defn.~2.1]{Navarro}), this is equivalent to showing that
 $c_1(\sL_Z) = e (c_1(\sL_S))$, where $c_1(\sL_D)$ denotes the `Chern class with
 support' of the divisor $D \subset X$ in $E^{2}_D(X, \Lambda(1))$. We are now done by 
 \cite[Prop.~1.43]{Navarro} because $\h\Lambda_{X/k} \cong \h\Lambda_X$ is an
 oriented absolute ring spectrum with the additive group law 
 (cf. \cite[Rem.~3.7]{CD-Doc}). 
\end{proof}

Suppose now that $k$ is an imperfect field and $X \in \Sm_k$ is an integral scheme
of dimension $d \ge 1$. We let $x \in X$ be a closed point and let $Z =
\Spec(k(x)) \times_X X^\hs$ (cf. \S~\ref{sec:Cont}).
We let $y$ be the unique point in $X^\hs$ lying over $x$
so that $\Spec(k(y)) = Z_\red$. Let $e = \ell(\sO(Z))$. Let $\alpha \colon
\Spec(k(x)) \inj X$ and $\beta \colon \Spec(k(y)) \inj X^\hs$ be the inclusions.
Let $f \colon X^\hs \to X$ and $g \colon \Spec(k(y)) \to \Spec(k(x))$ be the
projections. We have the refined Gysin maps
$\alpha_* \colon E^i(k(x), \Lambda(j)) \to E^{2d+i}_x(X, \Lambda(d+j))$ and
$\beta_* \colon E^i(k(y), \Lambda(j)) \to E^{2d+i}_y(X^\hs, \Lambda(d+j))$.

\begin{lem}\label{lem:Gysin-perf}
  Suppose $(p,m) =1$ and $\Lambda = {\Z}/m$. Then $e \in \Lambda^\times$
  and one has a commutative diagram
  \begin{equation}\label{eqn:Gysin-perf-0}
    \xymatrix@C1pc{
      E^i(k(x), \Lambda(j))   \ar[r]^-{\alpha_*} \ar[d]_-{eg^*} &
      E^{2d+i}_x(X, \Lambda(d+j))
    \ar[d]^-{f^*} \\
    E^i(k(y), \Lambda(j)) \ar[r]^-{\beta_*} & E^{2d+i}_y(X^\hs, \Lambda(d+j)).}
\end{equation}
\end{lem}
\begin{proof}
By the Bertini theorem of Altman-Kleiman \cite[Thm.~1]{AK}, we can find an integral
  curve $C \subset X$ which passes through $x$ and is smooth over $k$.
We now consider the diagram 
\begin{equation}\label{eqn:Gysin-perf-1}
    \xymatrix@C1pc{
      E^i(k(x), \Lambda(j)) \ar[r] \ar[d]_-{eg^*} & E^{2+i}_x(C, \Lambda(1+j)) \ar[r]
      \ar[d] & E^{2d+i}_x(X, \Lambda(d+j)) \ar[d] \\
    E^i(k(y), \Lambda(j)) \ar[r] & E^{2+i}_y(C^\hs, \Lambda(1+j)) \ar[r]
    &  E^{2d+i}_y(X^\hs, \Lambda(d+j)),}
\end{equation}
where the middle and the right vertical arrows are the pull-back maps,
and the horizontal arrows are the refined Gysin maps of ~\eqref{eqn:Gysin-0}.
We note that $C^\hs = C \times_X X^\hs$. It follows from
\cite[Prop.~1.8, (1) and (3)]{Navarro} that $\alpha_*$ (resp. $\beta_*$) is the composition of
the two upper (resp. lower) horizontal arrows, and the right square is commutative.
It suffices therefore to show that the left square is commutative. This allows us to
assume that $d = 1$.

We now let $\Spec(k(y)) \xrightarrow{\iota} Z \xrightarrow{\tau} X^\hs$ be the
inclusions and $h \colon Z \to \Spec(k(x))$ the projection.
Since $f$ is flat, $\tau$ must be a regular closed immersion.
We consider the diagram
\begin{equation}\label{eqn:Gysin-perf-2}
    \xymatrix@C1pc{
 E^i(k(x), \Lambda(j)) \ar[rr]^-{\alpha_*} \ar[dd]_-{eg^*} \ar[dr]^-{h^*} & & 
 E^{2+i}_x(X, \Lambda(1+j)) \ar[dd]^-{f^*} \\
 & E^i(Z, \Lambda(j)) \ar[dr]^-{\tau_*} \ar[dl]_-{\iota^*} & \\
  E^i(k(y), \Lambda(j)) \ar[rr]^-{\beta_*} & & 
  E^{2+i}_y(X^\hs, \Lambda(1+j)).}
\end{equation}
Since $Z = \Spec(k(x)) \times_X X^\hs$, the argument for the commutativity of the right
square in ~\eqref{eqn:Gysin-perf-1} also shows that $f^* \circ \alpha_* =
\tau_* \circ h^*$.
Since $g^* = \iota^* \circ h^*$, it suffices to show that
$\tau_* = e (\beta_* \circ \iota^*)$. But this follows from \lemref{lem:Gysin-2}.
It remains to show that $e \in \Lambda^\times$.

Now, we have shown in the proof of \lemref{lem:Continuity} that there exists a
finite purely inseparable extension ${k'}/k$ and a unique closed point $y'
\in X' := X \times_{\Spec(k)} \Spec(k')$ such that $\Spec(k(y)) = \Spec(k(y')) \times_{X'}
X^\hs$. It follows then from \lemref{lem:length} that
$e = \ell(\sO(Z'))$, where $Z' = \Spec(k(x)) \times_X X'$.
On the other hand, we know that
\[
[k': k] = \dim_{k(x)}(\sO(Z'))  =
[k(y'):k(x)] \ell(\sO(Z')).
\]
where the second equality follows from \cite[Lem.~A.1.3]{Fulton}. Since $[k': k]$ is a
power of $p$, it follows that so is $\ell(\sO(Z'))$ and we are done.
\end{proof}

\begin{cor}\label{cor:Gysin-perf-3}
With $\Lambda$ as in \lemref{lem:Gysin-perf}, one has a commutative diagram
  \begin{equation}\label{eqn:Gysin-perf-4}
    \xymatrix@C1pc{
    H^i(k(x), \Lambda(j))   \ar[r]^-{\alpha_*} \ar[d]_-{eg^*} & H^{2d+i}_c(X, \Lambda(d+j))
    \ar[d]^-{f^*} \\
    H^i(k(y), \Lambda(j)) \ar[r]^-{\beta_*} & H^{2d+i}_c(X^\hs, \Lambda(d+j)).}
\end{equation}
\end{cor}
\begin{proof}
  We choose an integral compactification $\ov{X}$ of $X$. Let
  $\iota^\hs \colon Z \inj X^\hs, \ u' \colon X \inj \ov{X}$  and
  $u^\hs \colon X^\hs \inj \ov{X}^\hs$ be the
  inclusions, where $Z = \Spec(k(x)) \times_X X^\hs$. We set
  $\tau^\hs = u^\hs \circ \iota^\hs$.
  In view of Lemmas~\ref{lem:SH-DM-Maps}, \ref{lem:Gysin-CS} and ~\ref{lem:Gysin-perf},
  we only need to show that the diagram
\begin{equation}\label{eqn:Gysin-perf-4-0}
\xymatrix@C1pc{
  E^{2d+i}_x(X, \Lambda(d+j)) \ar[r] \ar[d]_-{f^*} & E^{2d+i}_c(X, \Lambda(d+j))
  \ar[d]^-{f^*} \\
  E^{2d+i}_Z(X^\hs, \Lambda(d+j)) \ar[r] & E^{2d+i}_c(X^\hs, \Lambda(d+j))}
\end{equation}
is commutative, where the horizontal arrows are constructed in the proof of
\lemref{lem:Gysin-CS}. But this is equivalent to the commutativity of
the diagram
\begin{equation}\label{eqn:Gysin-perf-4-1}
\xymatrix@C1pc{
\Hom_{\sh(\ov{X})}(\tau_! \1_{x}, \Sigma^{i,j}\h\Lambda_{\ov{X}/k})
\ar[r] \ar[d]_-{f^*} &  \Hom_{\sh(\ov{X})}(\1_{\ov{X}}, u'_! \Sigma^{i,j}\h\Lambda_{X/k})
\ar[d]^-{f^*} \\
\Hom_{\sh(\ov{X}^\hs)}(\tau^\hs_! \1_{Z}, \Sigma^{i,j}\h\Lambda_{{\ov{X}^\hs}/k}) \ar[r] &
\Hom_{\sh(\ov{X}^\hs)}(\1_{\ov{X}}, u^\hs_! \Sigma^{i,j}\h\Lambda_{X^\hs/k}),}
\end{equation}
which is clear by the compatibility of the lower shriek functors with the
pull-back (cf. \cite[\S~A.5.1(5)]{CD-Springer}).
\end{proof}

\subsection{Milnor $K$-theory and motivic cohomology}\label{sec:MK-MC}
Let $k$ be a field containing $\kappa$.
Let $\alpha = \{f_1, \ldots , f_n\} \in K^M_n(k)$ be a Steinberg symbol.
Then each $f_i$ defines a morphism $f_i \colon \Spec(k) \to \G_{m,k}$.
Taking the induced map between the motives and using
  \cite[Prop.~11.2.11]{CD-Springer}, we get a
  morphism $f_i \colon \Z \to \Z(1)[1]$ in  $\dm(k, \Z)$.
  Taking the tensor product over $1 \le i \le n$, we get a morphism
  $\alpha \colon \Z \to \Z(n)[n]$ in  $\dm(k, \Z)$.
  Equivalently, $\alpha$ defines a unique element $\Psi_k(\alpha) \in
  H^n(k, \Z(n))$. Hence, we get a map (a priori, only of sets)
\begin{equation}\label{eqn:MC-field-1}
  \Psi_k \colon K^M_1(k) \times \cdots \times K^M_1(k) \to  H^n(k, \Z(n)),
\end{equation}
where recall that $H^n(k, \Z(n)) = \Hom_{\dm(k, \Z)}(\Z_k, \Z_k(j)[i])$
(cf. \cite[\S~11.0.1, Defn.~11.2.1]{CD-Springer}).
  This is clearly natural with respect to the pull-back maps via inclusion of fields.
  For $n =1$, the following result follows by combining
  \cite[Prop.~8.1, Cor.~8.2]{CD-Doc} and \cite[Thm.~11.2.14]{CD-Springer}.
  It provides the base case of the more general isomorphism
  (with prime-to-$p$ coefficients) between the
tame class group and a motivic cohomology group with compact support for smooth
schemes over a local field that we shall prove in the next subsection.

\begin{lem}\label{lem:MC-field}
 If $m \in k^\times$ and $n \ge 0$, the map
  $\Psi_k$ induces a canonical group isomorphism
  \begin{equation}\label{eqn:MC-field-0}
    \Psi_k \colon {K^M_n(k)}/m \xrightarrow{\cong} H^n(k, {\Z}/m(n))
  \end{equation}
  which is natural with respect to the pull-back maps on both sides via inclusion of
  fields.
\end{lem}
\begin{proof}
  If $k$ is perfect, it is well known that $\Psi_k$, as defined in
  ~\eqref{eqn:MC-field-1},
  induces an isomorphism $\Psi_k \colon {K^M_n(k)} \xrightarrow{\cong} H^n(k, {\Z}(n))$.
  Since $H^{n+1}(k, \Z(n)) \cong \CH^n(k, n-1) = 0$ (cf. \cite{Voe-imrn}),
  the universal coefficient theorem
  implies that the change of coefficients map
  ${H^n(k, {\Z}(n))}/m \to  H^n(k, {\Z}/m(n))$
is an isomorphism. This implies ~\eqref{eqn:MC-field-0}.

Suppose now that $k$ is not perfect and look at the diagram
  \begin{equation}\label{eqn:MC-field-2}
    \xymatrix@C1pc{
      {K^M_n(k)}/m \ar@{.>}[r]^-{\Psi_k} \ar[d]_-{\pi^*} &
      H^n(k, {\Z}/m (n)) \ar[d]^-{\pi^*}
      \\
     {K^M_n(k^{\heartsuit})}/m \ar[r]^-{\Psi_{k^{\heartsuit}}} & 
     H^n(k^{\heartsuit}, {\Z}/m (n)),}
   \end{equation}
   where $\pi \colon \Spec(k^{\heartsuit}) \to \Spec(k)$ is the projection.
   We have seen that $\Psi_k$ is defined on the Steinberg symbols in $K^M_n(k)$ and
   the diagram commutes on these symbols.
   It follows by a combination of \lemref{lem:SH-DM-Maps} and
   \cite[Cor.~2.1.7]{Elmanto-Khan} that the right vertical arrow in the above
   diagram is an isomorphism. The left vertical
   arrow is an isomorphism by \lemref{lem:Milnor-K-perf}. Since $\Psi_{k^{\heartsuit}}$
   is an isomorphism, we conclude that $\Psi_k$,   as defined in
   ~\eqref{eqn:MC-field-1},
   induces a map $\Psi_k \colon {K^M_n(k)}/m \to H^n(k, {\Z}/m(n))$, which is
   an isomorphism of groups.
\end{proof}

\subsection{Motivic cohomology and idele class group}\label{sec:MH}
Assume that $k$ is an infinite field containing $\kappa$ and
let $\Lambda' := \Z[\tfrac{1}{p}]$.
Let $X \in \Sm_k$ be integral of dimension $d$ with structure map
$f \colon X \to \Spec(k)$. We fix an integer $m$ prime to $p$
and let $\Lambda = {\Z}/m$. The goal of this subsection is to
prove an isomorphism between the mod-$m$ idele class group of $X$ and
$H^{2d+1}_c(X, \Lambda(d+1))$. This will play a key role in the proofs of the main
results of the paper. For $i,j \in \Z$, we let
$H_{i}(X, {\Z}(j)) := \Hom_{\dm(k, \Z)}(\Z_k(j)[i], M_k(X))$,
where recall that $M_k(X) = f_\sharp \Z_X$ which is defined since $X \in \Sm_k$
(e.g., see the proof of \thmref{thm:CD-cdh} which shows that $f_\sharp$ exists at
the level of integral motives when $f$ is smooth).

\begin{lem}\label{lem:KP}
    Assume that $k$ is perfect. Then the change of coefficients homomorphism
    ${H_{-1}(X, {\Z}(-1))}/m \to H_{-1}(X, {\Z}/m(-1))$ is an isomorphism.
  \end{lem}
  \begin{proof}
    We can replace $\Z$ by $\Lambda'$.
    By the universal coefficient theorem, it
    suffices to show that $H_{-2}(X, {\Lambda}'(-1)) = 0$.
    By \cite[Thm.~5.3.18(iii)]{Kelly}, this is equivalent to showing that
    $H^{2d+2}_c(X, {\Lambda}'(d+1)) = 0$.
To this end, we choose an integral compactification
    $\ov{X}$ of $X$ and let $Z = \ov{X} \setminus X$ with the reduced closed subscheme
    structure. We then have the localization exact sequence (cf. \lemref{lem:Ex-seq})
\[
  H^{2d+1}(Z, \Lambda'(d+1)) \to H^{2d+2}_c(X, \Lambda'(d+1)) \to
  H^{2d+2}(\ov{X}, \Lambda'(d+1)),
\]
in which the terms on the two ends are zero by 
\cite[Thm.~5.1]{Krishna-Pelaez-AKT}.
The lemma follows.
\end{proof}

%\enlargethispage{15pt}

\begin{prop}\label{prop:Tame-MCCS}
    There is an isomorphism
    \[
      \Psi_X \colon {C(X)}/m \xrightarrow{\cong} H^{2d+1}_c(X, \Lambda(d+1))
    \]
    such that the diagram
    \begin{equation}\label{eqn:Tame-MCCS-0}
      \xymatrix@C1pc{
        {K^M_1(k(x))}/m \ar[r]^-{\Psi_{k(x)}} \ar[d]_-{\iota^x_*} &
        H^1(k(x), \Lambda(1)) \ar[d]^-{\iota^x_*} \\
        {C(X)}/m \ar[r]^-{\Psi_X} & H^{2d+1}_c(X, \Lambda(d+1))}
    \end{equation}
    is commutative for every inclusion $\iota^x \colon \Spec(k(x)) \inj X$ of
    a closed point.
  \end{prop}
\begin{proof}
By Lemma~\ref{lem:Idele-Tame-prime-to-p},
    we can replace ${C(X)}/m$ by ${C^\tm(X)}/m$. We shall work with the latter
    group throughout the proof.
We first assume that $k$ is perfect. For any $x \in X_{(0)}$, we let
    $\Theta_{k(x)}$ denote the composition of the isomorphisms 
    \begin{equation}\label{eqn:Tame-MCCS-1}
      K^M_1(k(x))[\tfrac{1}{p}] \xrightarrow{\Psi_{k(x)}}
      H^1(k(x), \Lambda'(1)) \xrightarrow{\Phi_{k(x)}}
      H_{-1}(k(x), \Lambda'(-1)),
\end{equation}
where $\Phi_{k(x)}$ is the duality isomorphism (cf. ~\eqref{eqn:Gysin-PF}).
We then get a diagram
  \begin{equation}\label{eqn:Tame-MCCS-2}  
    \xymatrix@C1pc{
      {\underset{x \in X_{(0)}}\bigoplus} {K^M_1(k(x))}/m \ar[r] \ar@{->>}[d] &
      {\underset{x \in X_{(0)}}\bigoplus} H^1(k(x), \Lambda(1)) \ar[r] \ar[d] &
      {\underset{x \in X_{(0)}}\bigoplus}  H_{-1}(k(x), \Lambda(-1)) \ar[d] \\
    {C^\tm(X)}/m \ar@{.>}[r]^-{\Psi_X} \ar@/_2pc/[rr]^-{\Theta_X} &    
    H^{2d+1}_c(X, \Lambda(d+1)) \ar[r]^-{\Phi_X}_-{\cong} &  H_{-1}(X, \Lambda(-1)),}
\end{equation}
where the top horizontal arrows are from ~\eqref{eqn:Tame-MCCS-1}, the left vertical
arrow is the canonical surjection, the middle one is the sum of Gysin maps
for closed immersions of
regular schemes (cf. Lemmas~\ref{lem:SH-DM-Maps} and ~\ref{lem:Gysin-CS}) and
the right vertical arrow is the sum of push-forward maps. The right square is
commutative by ~\eqref{eqn:Gysin-PF}. 
It follows from \cite[Thm.~1.3]{Yamazaki} and \lemref{lem:KP} that the composition of
the right vertical arrow with the two top horizontal arrows induces $\Theta_X$ which
is an isomorphism.  Since $\Phi_X$ is an isomorphism, it follows that there is a unique
isomorphism $\Psi_X$ such that the left square commutes.

Suppose now that $k$ is not perfect and let $f \colon X^\hs \to X$ be the projection
map. For $x \in X_{(0)}$ and the unique point $y \in X^\hs$ lying over $x$, we let
$f^*_x \colon H^1(k(x), \Lambda(1)) \to H^1(k(y), \Lambda(1))$ be given by
$f^*_x(a) = \ell(S_x) \tau^*_x(a)$, where $S_x = \Spec(k(x)) \times_X X^\hs$ and
$\tau_x \colon \Spec(k(y)) \to \Spec(k(x))$ is the projection map.
The map $f^*_x \colon {K^M_1(k(x))}/m \to {K^M_1(k(y))}/m$ has been earlier defined in
the same manner in ~\eqref{eqn:Defn-PB}.
We let $\wt{f}^* = {\underset{x \in X_{(0)}}\bigoplus} f^*_x$.
We let $\alpha = {\underset{x \in X_{(0)}}\bigoplus} \Psi_{k(x)}$
(cf. \lemref{lem:MC-field}). The map $\alpha^\hs$ is defined in the similar
manner.

We now consider the diagram
\begin{equation}\label{eqn:Tame-MCCS-3}
  \xymatrix@C1pc{
    {\underset{x \in X_{(0)}}\bigoplus} {K^M_1(k(x))}/m \ar[rr]^-{\alpha}
    \ar@{->>}[dr]  \ar[dd]_-{\wt{f}^*} & &
    {\underset{x \in X_{(0)}}\bigoplus} H^1(k(x), \Lambda(1)) \ar[dd]^>>>>>>>>{\wt{f}^*}
    \ar[dr] & \\
& {C^\tm(X)}/m \ar[dd]_>>>>>>>>{f^*} \ar@{.>}[rr] & & H^{2d+1}_c(X, \Lambda(d+1))
\ar[dd]^-{f^*} \\
{\underset{x \in X^\hs_{(0)}}\bigoplus} {K^M_1(k(x))}/m \ar[rr]^->>>>>>>>>{\alpha^\hs}
\ar@{->>}[dr]  & &
{\underset{x \in X^\hs_{(0)}}\bigoplus} H^1(k(x), \Lambda(1)) \ar[dr] & \\
& {C^\tm(X^\hs)}/m \ar[rr]^-{\Psi_{X^\hs}} & &  H^{2d+1}_c(X^\hs, \Lambda(d+1)).}
\end{equation}

We have shown above that $\Psi_{X^\hs}$ is an isomorphism. The right vertical arrow of
the front face is an isomorphism by \lemref{lem:SH-DM-Maps} and
\cite[Cor.~2.1.7]{Elmanto-Khan}. Its left vertical arrow is an
isomorphism by \corref{cor:Continuity-*}. It follows that there is a unique
isomorphism $\Psi_X \colon {C^\tm(X)}/m \xrightarrow{\cong} H^{2d+1}_c(X, \Lambda(d+1))$
such that the front face of the diagram commutes. To complete the proof of the
proposition, it remains to show that the top face of ~\eqref{eqn:Tame-MCCS-3} commutes.
Since the right vertical arrow of the front face is an isomorphism, it suffices to show
that all other faces of ~\eqref{eqn:Tame-MCCS-3} commute.

Now, the left face of ~\eqref{eqn:Tame-MCCS-3} commutes by the construction of the
pull-back map $f^*$ between the tame class groups. The right face commutes by
\corref{cor:Gysin-perf-3}.
The back face commutes by definitions of all its arrows. The bottom face commutes by
definition of $\Psi_{X^\hs}$ (cf. the diagram ~\eqref{eqn:Tame-MCCS-2}). The front face
has been shown to commute. This concludes the proof.
\end{proof}

%\enlargethispage{10pt}

\section{Reciprocity and {\'e}tale realization}\label{sec:RR}
The goal of this section is to describe the reciprocity homomorphism of the tame
class group in terms of the realization map from the motivic cohomology with compact
support to the {\'e}tale cohomology with compact support.
This will allow us to paraphrase the statements of some of the main results of this
paper as statements about the {\'e}tale realization map.
We let $\kappa$ be a prime field of characteristic exponent $p \ge 1$
and fix a field $k$ containing $\kappa$ throughout this section.
We fix an integer $m \in \kappa^\times$ and let $\Lambda = {\Z}/m$.
We begin with a rapid recollection of the {\'e}tale realization functor.

\subsection{{\'E}tale realization of motivic cohomology}\label{sec:ER}
For $X \in \Sch_k$, let $\dm_\et(X, \Lambda)$ be the triangulated
category of {\'e}tale sheaves of $\Lambda$-modules with transfer on $\Sm_X$.
Let $\dr_\et(X, \Lambda)$ be
the unbounded derived category of sheaves of $\Lambda$-modules on $X_\et$.
There is a commutative diagram of triangulated categories
(cf. \cite[\S~9]{CD-Doc} and \cite[\S~1,2]{CD-Comp})
\begin{equation}\label{eqn:RE-0}
  \xymatrix@C1pc{
    \dm(X, \Lambda) \ar[r] \ar[d]_-{\vartheta^*_X} &
    \dm_\et(X, \Lambda) \ar[r]^-{\tau_X}_-{\cong} \ar[d]^-{\omega^*_X}_-{\cong}
    & \dr_\et(X, \Lambda) \\
    \dm_\cdh(X, \Lambda) \ar@{.>}[ur]^-{\epsilon^*_X} \ar[r]_-{\epsilon '^*_X} &
    \dm_{\rm h}(X, \Lambda), &} 
\end{equation}
where $\dm_{\rm h}(X, \Lambda)$ is the triangulated category of {\'e}tale motives with
respect to Voevodsky's h-topology \cite[\S~3]{Voev-Selecta}.
The functor $\tau_X$ is obtained by forgetting the transfer
structure and the remaining functors are induced by sheafification with respect
to the underlying topologies on their targets.

One knows that $\omega^*_X$ and $\tau_X$ are equivalences of monoidal triangulated
categories (cf. \cite[Thm.~4.5.2, Cor.~5.5.4]{CD-Comp}).
The equivalence of $\omega^*_X$ implies that
$\epsilon'^*_X$ factors through
\begin{equation}\label{eqn:cdh-et}
\epsilon^*_X \colon \dm_\cdh(X, \Lambda) \to \dm_\et(X, \Lambda)
\end{equation}
such that two triangles in ~\eqref{eqn:RE-0} commute.
Both categories in ~\eqref{eqn:cdh-et} satisfy the six functor formalism
(cf. \cite[Thm.~2.4.50]{CD-Springer}) and the functor $\epsilon^*_X$ is 
compatible with the underlying functors when restricted to constructible objects
(cf. \cite[Rem.~9.6]{CD-Doc}).
Henceforth, we shall make no distinction between
$\dm_\et(X, \Lambda)$, \
$\dm_{\rm h}(X, \Lambda)$ and $\dr_\et(X, \Lambda)$ in order to
simplify notations.

For any $f \colon X \to \Spec(k)$ in $\Sch_k$ and a closed immersion
$\iota \colon Z \inj X$, we let
\begin{equation}\label{eqn:RE-1}
  H^i_{\et, \sM}(X, \Lambda(j)) =  \Hom_{\dm_\et(k, \Lambda)}(\Lambda_k, f_*\Lambda_X(j)[i]);
  \end{equation}
\[
  H^i_{\et, \sM, c}(X, \Lambda(j)) = \Hom_{\dm_\et(k, \Lambda)}(\Lambda_k, f_!\Lambda_X(j)[i]);
\]
\[
H^i_{\et, \sM, Z}(X, \Lambda(j)) =
\Hom_{\dm_\et(X, \Lambda)}(\iota_* \Lambda_Z, \Lambda_X(j)[i]).
  \]

  Since $\epsilon^*_X$ is monoidal, \lemref{lem:MC-et} provides
  natural {\'e}tale realization maps
\begin{equation}\label{eqn:RE-2}
  H^i_Z(X, \Lambda(j)) \xrightarrow{\epsilon^*_X} H^i_{\et, \sM, Z}(X, \Lambda(j)); \ \
  H^i_c(X, \Lambda(j)) \xrightarrow{\epsilon^*_X} H^i_{\et, \sM, c}(X, \Lambda(j)).
\end{equation}

Before we proceed further, we recall the following.

\begin{lem}\label{lem:Roots-PB}
  Let $k_0$ be any field and $X$ any Noetherian $k_0$-scheme with structure map
  $f \colon X \to \Spec(k_0)$. Then the canonical map $f^* \mu_{m,k_0} \to \mu_{m,X}$ is an
  isomorphism for all integers $m \in (k_0)^\times$.
\end{lem}
\begin{proof}
  It is easy to check that $f^* \mu_{m,k_0} \to \mu_{m,X}$ is injective. To prove its
  surjectivity, 
  it suffices to show that for any $k_0$-algebra $A$ and $\alpha \in \mu_m(A)$, there is
  a factorization $k_0 \inj A' \xrightarrow{\phi} A$ and $\alpha' \in  \mu_m(A')$ such
  that $A'$ is {\'e}tale over $k_0$ and $\phi(\alpha') = \alpha$.
  To show this, we let $\phi' \colon k_0[t] \to A$ be the $k_0$-algebra homomorphism
  defined by setting $\phi'(t) = \alpha$. Then $\phi'$ has a factorization
  $\phi \colon \frac{k_0[t]}{(t^m-1)} \to A$. Since $m \in (k_0)^\times$, it follows that
  $A' := {k_0[t]}/{(t^m-1)}$ is {\'e}tale over $k_0$. Letting $\alpha'$ be residue
  class of $t$ in $A'$, we conclude the proof.
\end{proof}

For $j \in \Z$, we let $\Lambda'(j)$ be the sheaf on $X_\et$ defined so that
$\Lambda'(0)= \Z/m$, $\Lambda'(j) = \mu_m^{\otimes j}$ when $j >0$,
$\Lambda'(j) = \un{\Hom}(\Lambda'(-j), \Z/m)$ when $j<0$, where the latter is the
internal hom in the category of sheaves of abelian groups on $X_\et$ 
(cf. \cite[\S~V.1, p.~163]{Milne-etale}).

\begin{lem}\label{lem:EFT}
  For any $X \in \Sch_k$, the functor $\tau_X$ induces an isomorphism
  \[
    H^i_{\et, \sM}(X, \Lambda(j)) \to H^i_{\et}(X, \Lambda'(j)).
  \]
  The same holds for cohomology with compact support and cohomology with support
  in a closed subscheme.
\end{lem}
\begin{proof}
  The second part of the lemma follows from its first part using the long exact
  sequences of cohomology and naturality of $\tau_X$.
  So we only need to prove the first part. By ~\eqref{eqn:RE-1}, it suffices to show
  that 
  the natural map $\Lambda'_X(j) \to \tau_X (\Lambda_X(j))$ is an isomorphism. 
  Let $f \colon X \to \Spec(\kappa)$ be the projection.
  Since the pull-back map $f^* \Lambda_{\kappa}(j) \to \Lambda_X(j)$
  is clearly an isomorphism, and the same holds for $\Lambda'(j)$ by
  \lemref{lem:Roots-PB}, it suffices to prove the above isomorphism when
  $X = \Spec(\kappa)$. However, as $\kappa$ is a perfect field, this isomorphism is
  well known (cf. \cite[Prop.~10.6]{MVW}). 
  \end{proof}

Note: In the light of \lemref{lem:EFT}, we shall suppress $\sM$ from the notations in
\eqref{eqn:RE-1} and make no distinction between the {\'e}tale motivic
cohomology (and their variants) with coefficients in $\Lambda(j)$ and the corresponding
{\'e}tale sheaf cohomology with coefficients in $\Lambda'(j)$ in the rest of this paper.

\subsection{Gysin homomorphisms in motivic and {\'e}tale
    cohomology}\label{sec:GEC}
Let $k$ and $\Lambda$ be as above.
Suppose now that $\iota \colon Z \inj X$ is a closed immersion of
pure codimension $n$ in $\Sch_k$ such that $X$ and $Z$ are both regular.
Recall from \S~\ref{sec:Gysin**} (cf. Lemma~\ref{lem:SH-DM-Maps}) that
there is a fundamental class $\eta^X_Z \in H^{2n}_Z(X, \Lambda(n))$
such that the cup product
$\bigcup \colon H^i(Z, \Lambda(j)) \times H^{2n}_Z(X, \Lambda(n))
\to H^{2n+i}_Z(X, \Lambda(n+j))$ gives rise to the refined Gysin homomorphism
\begin{equation}\label{eqn:Gysin-RE}
\iota_* \colon  H^i(Z, \Lambda(j)) \to H^{2n+i}_Z(X, \Lambda(n+j)); \ \
\alpha \mapsto \alpha \cup \eta^X_Z.
\end{equation}
This construction of the refined Gysin homomorphism is in fact obtained by mimicking 
the construction of already existing
fundamental classes and refined Gysin homomorphism in {\'e}tale
cohomology by Gabber-Riou (cf. \cite{Fujiwara} and \cite{Riou}). 

\begin{lem}\label{lem:Gysin-Nis-etale}
There is a commutative diagram
  \begin{equation}\label{eqn:Gysin-Nis-etale-00}
    \xymatrix@C1pc{
      H^i(Z, \Lambda(j)) \ar[r]^-{\epsilon^*_Z} \ar[d]_-{\iota_*} &
      H^i_\et(Z, \Lambda(j)) \ar[d]^-{\iota_*} \\
      H^{2n+i}_Z(X, \Lambda(n+j)) \ar[r]^-{\epsilon^*_X} &
      H^{2n+i}_{Z,\et}(X, \Lambda(n+j)).}
  \end{equation}
\end{lem}
\begin{proof}
Since the {\'etale} realization functor is monoidal, $\epsilon^*_X$ preserves
  cup products. Hence, we only need to show that it also preserves the fundamental
  class $\eta^X_Z$. But this is part of the construction of the refined
  Gysin homomorphism in the motivic cohomology. We briefly recall this construction
  for the sake of completeness of our arguments.

Let $X'$ be the blow-up of $X$ along $Z$ with exceptional divisor $E \cong \P_Z(\sF)$.
Let $\pi \colon X' \to X, \ \pi' \colon E \to Z$ be the projections and
$\iota' \colon E \inj X'$ the
  inclusion. Let $\sF$ be the conormal sheaf of the embedding $Z \inj X$. 
Then the pull-back map $\pi^* \colon H^{2n}_Z(X, \Lambda(n)) \to
  H^{2n}_E(X', \Lambda(n))$ is an inclusion \cite[Cor.~2.6]{Navarro} (cf.
  \cite{Riou} for the {\'e}tale case). Furthermore, there is a class $\omega^X_Z \in
  H^{2n-2}(E, \Lambda(n-1))$ which is a linear combination of powers of
$c_1(\iota'^*\sL_E)$ and
  the Chern classes of $\pi'^*(\sF)$ such that the image of $\omega^X_Z$ under
  the refined Gysin homomorphism $\iota_* \colon  H^{2n-2}(E, \Lambda(n-1)) \to
  H^{2n}_E(X', \Lambda(n))$ lies in $H^{2n}_Z(X, \Lambda(n))$ and this image is the
  class $\eta^X_Z$. Hence, all one needs to check is that $\epsilon^*_{X'}(c_1(\sL_E))$
  and $\epsilon^*_{Z}(c_i(\sF))$ coincide with $c^\et_1(\sL_E)$ and
$c^\et_i(\sF)$.

Since the ordinary Chern classes of $\sF$ are preserved under $\epsilon^*_{Z}$
  (as is clear from \cite[Defn.~1.32]{Navarro}), we have to check the same for
  the `Chern class with support' $c_1(\sL_E)$. 
But this is also well known. Indeed, recall that
\begin{equation}\label{eqn:Gysin-Nis-etale-01}
  \begin{array}{lll}
    H^2_E(X', \Lambda(1)) & {\cong}^\dagger &
    \Hom_{\dm(X', \Lambda)}(\iota'_! \Lambda_{E}, \Lambda_{X'}(1)[2]) \\
  & \cong & \Hom_{\dm(X', \Lambda)}(\iota'_!\Lambda_{E}, \G_{m,X'} \otimes \Lambda(1)) \\
& \cong & H^1_E(X'_\nis, {\sO^\times_{X'}}/m),
  \end{array}
  \end{equation}
where ${\cong}^\dagger$ is \corref{cor:MC-et-5} and
the isomorphism $\Lambda_{X'}(1)[1] \cong \G_{m,X'} \otimes \Lambda$ is shown in the
proof of \cite[Prop.~11.2.11]{CD-Springer}.
The last isomorphism in~\eqref{eqn:Gysin-Nis-etale-01} is a classical fact in
$\A^1$-homotopy theory of schemes (cf. \cite[Ex.~10.4.4]{CD-Springer}).

Now, it follows from the Cartier exact sequence that $H^1_E(X'_\nis, {\sO^\times_{X'}})$
is the group of Cartier divisors on $X'$ which are supported on $E$.
Since $E \subset X'$ is a Cartier divisor, it has a uniquely defined class
(say, $[E]$) in $H^1_E(X'_\nis, {\sO^\times_{X'}}/m )$. Furthermore, 
$c_1(\sL_E) \in H^2_E(X', \Lambda(1))$ corresponds to
$[E]$ under the composite isomorphism of ~\eqref{eqn:Gysin-Nis-etale-01}
(cf. \cite[Defn.~1.41]{Navarro}).
As the class $c^\et_1(\sL_E) \in H^2_{\et, E}(X', \Lambda(1))
\cong H^2_E(X'_\et, \Lambda'(1))$ is defined in the identical way, one gets
 $\epsilon^*_{X'}(c_1(\sL_E)) = c^\et_1(\sL_E)$.
\end{proof}

\begin{cor}\label{cor:Gysin-Nis-etale-1}
 If $Z$ is complete in \lemref{lem:Gysin-Nis-etale}, there is a commutative 
diagram
\begin{equation}\label{eqn:Gysin-Nis-etale-2}
    \xymatrix@C1pc{
      H^i(Z, \Lambda(j)) \ar[r]^-{\epsilon^*_Z} \ar[d]_-{\iota_*} &
      H^i_\et(Z, \Lambda(j)) \ar[d]^-{\iota_*} \\
      H^{2n+i}_c(X, \Lambda(n+j)) \ar[r]^-{\epsilon^*_X} &
      H^{2n+i}_{\et, c}(X, \Lambda(n+j)).}
  \end{equation}
\end{cor}
\begin{proof}
This follows from \lemref{lem:Gysin-Nis-etale} because
  ~\eqref{eqn:Gysin-Nis-etale-2} is the composition of ~\eqref{eqn:Gysin-Nis-etale-00}
  and
  \begin{equation}\label{eqn:Gysin-Nis-etale-3}
    \xymatrix@C1pc{
      H^{2n+i}_Z(X, \Lambda(n+j)) \ar[r]^-{\epsilon^*_X} \ar[d]
      & H^{2n+i}_{\et, Z}(X, \Lambda(n+j)) \ar[d] \\
       H^{2n+i}_c(X, \Lambda(n+j)) \ar[r]^-{\epsilon^*_X} &
       H^{2n+i}_{\et, c}(X, \Lambda(n+j)),}
     \end{equation}
  where the vertical arrows are the forget support maps
  (cf. Lemmas~\ref{lem:SH-DM-Maps} and ~\ref{lem:Gysin-CS}).
     But this diagram commutes by the naturality of $\epsilon^*_X$ because the
     right vertical arrow is obtained from the left vertical arrow by applying
     $\epsilon^*_X$ to ~\eqref{eqn:Gysin-CS-2}.
\end{proof}

\subsection{Saito-Tate duality in {\'e}tale cohomology}\label{sec:SD}
We shall now assume that $k$ is a local field of characteristic exponent $p \ge 1$.
Let $m \in k^\times$ and $\Lambda = {\Z}/m$. We consider the commutative diagram
\begin{equation}\label{eqn:SD-1}
  \xymatrix@C1.5pc{
    Z \ar[r]^-{g'} \ar[d]_-{\iota} \ar[dr]^-{g} & \Spec(k') \ar[d]^-{\pi} \\
    X \ar[r]^-{f} & \Spec(k).}
\end{equation} 
Here, $X \in \Sm_k$ is an integral scheme of
dimension $d$, the left vertical arrow is a closed immersion of codimension $n$,
the top horizontal arrow is smooth and projective with $Z$ integral, and the right
vertical arrow is induced by
a finite field extension ${k'}/k$. In particular, $Z$ is a projective and regular
$k$-scheme. We let $e = d-n$.

In the above setting, there are canonical maps
 \begin{equation}\label{eqn:SD-0}
 f_! \Lambda_X \cong f_! f^* \Lambda_k \cong f_! f^!\Lambda_k(-d)[-2d]
   \xrightarrow{\alpha_{X/k}} \Lambda_k(-d)[-2d]
 \end{equation}
 in $\dm_\et(k, \Lambda)$, where the middle isomorphism is a consequence of the well
 known relation between
 $f^!$ and $f^*$ for smooth morphisms and $\alpha_{X/k}$ is the counit of adjunction. 
 The composite $f_* \colon f_! \Lambda_X \to  \Lambda_k(-d)[-2d]$ 
 is the Gabber-Riou Gysin map \cite{Riou}.
 We similarly have $g'_* \colon  g'_! \Lambda_Z \to \Lambda_{k'}(-e)[-2e]$.

We have maps $\iota_! \Lambda_Z \xrightarrow{\cong}
 \iota_! \iota^! \Lambda_X(n)[2n] \xrightarrow{\alpha_{X/Z}}
 \Lambda_X(n)[2n]$, where the first map is Gabber's purity isomorphism
 \cite{Fujiwara} and $\alpha_{X/Z}$ is the counit of adjunction. 
 The composite map $\iota_* \colon g_* \Lambda_Z \to f_! \Lambda_X$
 induces the Gysin homomorphism $\iota_* \colon 
 H^{i}_{\et}(Z, \Lambda(j)) \to H^{2n+i}_{\et, c}(X, \Lambda(n+j))$
 (cf. \corref{cor:Gysin-Nis-etale-1}).
 The finite map $\pi \colon \Spec(k') \to \Spec(k)$ induces the
 counit of adjunction $\pi_* \pi^* \Lambda_{k} \to \Lambda_k$ which
 yields the Gysin map $\pi_* \colon \pi_! \Lambda_{k'} \to \Lambda_k$.
 Let $g_* \colon g_! \Lambda_Z \to \Lambda_{k}(-e)[-2e]$ denote the
 Gabber-Riou Gysin homomorphism associated to the projective morphism
 $g: Z \to \Spec(k)$ between regular $k$-schemes.

\begin{lem}\label{lem:Tr-Gysin-commute}
  The diagram
\begin{equation}\label{eqn:SD-2}
  \xymatrix@C1pc{
    g_! \Lambda_Z(e+1)[2e+2] \ar[r]^-{\cong}  \ar[d]_-{\pi_! \circ \iota_*} 
    \ar@/^2pc/[rr]^-{\pi_! \circ g'_*} \ar[drr]^-{g_*} & \pi_! g'_! \Lambda_Z(e+1)[2e+2] 
\ar[r] & \pi_! \Lambda_{k'}(1)[2] \ar[d]^-{\pi_*}   \\
f_! \Lambda_X(d+1)[2d+2] \ar[rr]_-{f_*} & & \Lambda_k(1)[2]}
\end{equation}
is commutative.
\end{lem}
\begin{proof}
  The proof of the lemma is part of the construction of the Gysin homomorphism
  $g_*$. More precisely, \cite[Thm.~2.5.12]{Riou}
  (see also \cite[\S~3]{Deglise-Oriented}) shows that
  $\pi_* \circ g'_* = g_*$ and $f_* \circ \iota_* = g_*$.
\end{proof}

To translate ~\eqref{eqn:SD-2} in the language of cohomology groups, we note that 
  \[
    \begin{array}{lll}
      \Hom_{\dm_\et(k, \Lambda)}(\Lambda_k, \pi_! g'_! \Lambda_Z(e+1)[2e+2]) & \cong &
 \Hom_{\dm_\et(k, \Lambda)}(\Lambda_k, \pi_* g'_! \Lambda_Z(e+1)[2e+2]) \\    
& \cong & \Hom_{\dm_\et(k', \Lambda)}(\pi^* \Lambda_k, g'_! \Lambda_Z(e+1)[2e+2]) \\ 
 & \cong & \Hom_{\dm_\et(k', \Lambda)}(\Lambda_{k'}, g'_! \Lambda_Z(e+1)[2e+2]).\\
    \end{array}
  \]
  Similarly, we have
  \[
    \Hom_{\dm_\et(k, \Lambda)}(\Lambda_{k}, \pi_! \Lambda_{k'}(1)[2]) \cong
    \Hom_{\dm_\et(k', \Lambda)}(\Lambda_{k'}, \Lambda_{k'}(1)[2]) \cong
    H^2_\et(k', \Lambda(1)).
    \]
    In particular, the counit of adjunction $\pi_* \Lambda_{k'} \to \Lambda_k$
    induces $\pi_* \colon H^i_\et(k', \Lambda(j)) \to
H^i_\et(k, \Lambda(j))$.
Hence, by considering the morphisms in $\dm_\et(k, \Lambda)$,
induced by ~\eqref{eqn:SD-2}, we get:

\begin{lem}\label{lem:Gysin-Trace}
There is a commutative diagram
  \begin{equation}\label{eqn:SD-3}
    \xymatrix@C1pc{
      H^{2e+i}_{\et}(Z, \Lambda(e+j)) \ar[r]^-{g'_*} \ar[d]_-{\iota_*} &
      H^i_\et(k', \Lambda(j)) \ar[d]^-{\pi_*} \\
      H^{2d+i}_{\et,c}(X, \Lambda(d+j)) \ar[r]^-{f_*} & H^i_\et(k, \Lambda(j)).}
  \end{equation}
\end{lem}

\vskip .2cm

Let $\ff$ denote the residue field of $k$. Then there is a Witt-residue map
${\rm Inv}_k \colon {\rm Br}(k) \to H^1_\et(\ff, {\Q}/{\Z})$, which is an isomorphism
(cf. \cite[Defn.~1.4.11]{CTS}).
In particular, the composite map
$H^2_\et(k, \Lambda(1)) \xrightarrow{\cong} {}_m {\rm Br}(k)
\xrightarrow{{\rm Inv}_k} H^1_\et(\ff, \Lambda)$
is an isomorphism. Composing this with the canonical composite isomorphism
\begin{equation}\label{eqn:Tr-fin}
  \tr_\ff \colon H^1_\et(\ff, \Lambda) \cong \Hom_\Tab(G_\ff, \Lambda)
  \xrightarrow{\cong} \Lambda,
  \end{equation}
  we get an isomorphism $\tr_k \colon H^2_\et(k, \Lambda(1)) \xrightarrow{\cong}
  \Lambda$.
This map was used by Tate \cite{Tate} as the trace homomorphism for his
duality theorem for {\'e}tale cohomology over local fields.

\begin{lem}\label{lem:Trace-PF}
  Let ${k'}/k$ be a finite extension of local fields and let $\pi \colon
  \Spec(k') \to \Spec(k)$ be the projection map. Then we have a commutative diagram
  \begin{equation}\label{eqn:Trace-PF-0}
    \xymatrix@C1.5pc{
      H^2_\et(k', \Lambda(1)) \ar[r]^-{\pi_*} \ar@/_2pc/[rr]^-{\tr_{k'}} & 
      H^2_\et(k, \Lambda(1)) \ar[r]^-{\tr_{k}} & \Lambda.}
  \end{equation}
  \end{lem}
\begin{proof}
Let $\ff'$ denote the residue field of $k'$ and let $\pi' \colon \Spec(\ff') \to
  \Spec(\ff)$ be the projection. By definition, the composition of the top (resp.
  bottom) row in the diagram   
\begin{equation}\label{eqn:Trace-PF-3}
    \xymatrix@C1.5pc{
H^2_\et(k', \Lambda(1)) \ar[r]^-{{\rm Inv}_{k'}} \ar[d]_-{\pi_*} &
H^1_\et(\ff', \Lambda) \ar[d]^-{\pi'_*}  \ar[r]^-{\tr_{\ff'}}   & \Lambda \ar@{=}[d] \\
H^2_\et(k, \Lambda(1)) \ar[r]^-{{\rm Inv}_k} &
H^1_\et(\ff, \Lambda)  \ar[r]^-{\tr_{\ff}}  & \Lambda}
\end{equation}
is $\tr_{k'}$ (resp. $\tr_{k}$). 
Moreover, the left square commutes by \cite[Exc.~XII.3.2]{Serre-LF}.
It suffices therefore to show that the right square is commutative.

For this, we recall that $\tr_{\ff}$ is the same as the composite
  isomorphism $H^1_\et(\ff, \Lambda) \cong H^1(G_\ff, H^0_\et(\ov{\ff}, \Lambda))
  \cong \Hom_\Tab(G_\ff, \Lambda) \cong \Lambda$, where the second term is the
  (continuous) Galois cohomology and the first isomorphism is the edge map in the
  Hochschild-Serre spectral sequence. The same holds for $\tr_{\ff'}$ too.
  Since $\ov{\ff'} = \ov{\ff}$, it suffices to show that  
   \begin{equation}\label{eqn:Trace-PF-5}
    \xymatrix@C1.5pc{
     \Hom_\Tab(G_{\ff'}, \Lambda) \ar[d]_-{\pi'_*} \ar[r]^-{\cong}& \Lambda
      \ar@{=}[d] \\
      \Hom_\Tab(G_{\ff}, \Lambda) \ar[r]^-{\cong} & \Lambda,} 
  \end{equation}
  commutes. But this follows from \cite[\S~3.2, p.~661, Cor.~1]{Kato80}.
\end{proof}

We now recall the following duality theorem due to Tate \cite{Tate} and
Saito \cite[Lem.~2.9]{Saito-Duality}.
 For $X \in \Sm_k$ integral of dimension $d$, let
 $\tr_{X/k}$ denote the composite
 $H^{2d+2}_{\et, c}(X, \Lambda(d+1)) \xrightarrow{f_*} H^2_\et(k, \Lambda(1))
 \xrightarrow{\tr_k} \Lambda$. We shall call this the trace map for $X$.

\begin{thm}\label{thm:Saito-D}
   Let $X \in \Sm_k$ be an integral scheme of dimension $d$ with structure map
   $f \colon X \to \Spec(k)$. Let $K \in \dr_\et(X, \Lambda)$ be a bounded complex with
   constructible cohomology sheaves. Then the following hold.
   \begin{enumerate}
   \item
$H^i_{\et}(X, K)$ are finite for all $i \in \Z$.
\item
  The trace map $\tr_{X/k} \colon H^{2d+2}_{\et, c}(X, \Lambda(d+1)) \to \Lambda$
  is an isomorphism.
\item
 We have a perfect paring of finite abelian groups
  \[
    H^i_{\et, c}(X, K) \times H^{2d+2-i}_{\et}(X, \un{Hom}_X(K, \Lambda(d+1))) \to
    H^{2d+2}_{\et,c}(X, \Lambda(d+1)) \xrightarrow{\tr_{X/k}} \Lambda.
  \]
\end{enumerate}
 \end{thm} 

 We let $\omega_X \colon H^i_\et(X, \Lambda(j))^\vee \xrightarrow{\cong}
 H^{2d+2-i}_{\et, c}(X, \Lambda(d+1-j))$ be the Saito-Tate duality map obtained by
 taking $K = \Lambda(d+1-j)$
 %(cf. \S~\ref{sec:ER} for latter's definition)
 in \thmref{thm:Saito-D}(3). 
 
Suppose now that $\iota \colon Z \inj X$ is as in~\eqref{eqn:SD-1}.
We then get a diagram
 \begin{equation}\label{eqn:Gysin-SD*}
   \xymatrix@C1pc{
     H^1_\et(Z, \Lambda)^\vee \ar[r]^-{(\iota^*)^\vee} \ar[d]_-{\omega_{Z}} &
     H^1_\et(X, \Lambda)^\vee \ar[d]^-{\omega_{X}} \\
     H^{2e+1}_\et(Z, \Lambda(e+1)) \ar[r]^-{\iota_*} &
     H^{2d+1}_{\et,c}(X, \Lambda(d+1)).}
 \end{equation}
 Here, $(\iota^*)^\vee$ is the dual of the pull-back
 $\iota^* \colon H^1_\et(X, \Lambda) \to H^1_\et(Z, \Lambda)$, the vertical arrows
 are the unique maps induced by the Saito-Tate duality for $Z$ and $X$
 (cf. \thmref{thm:Saito-D}) and $\iota_*$ is the Gysin homomorphism.

\begin{lem}\label{lem:Gysin-SD}
   The diagram ~\eqref{eqn:Gysin-SD*} is commutative.
 \end{lem}
 \begin{proof}
   Since the Gysin map $\iota_*$ is linear with respect to the action of the
   {\'e}tale cohomology ring of $X$ on the {\'e}tale cohomology of $Z$ via $\iota^*$,
   we have a commutative diagram
\begin{equation}\label{eqn:Gysin-SD-0}   
  \xymatrix@C1pc{
    H^1_\et(Z, \Lambda) \times H^{2e+1}_{\et}(Z, \Lambda(e+1)) \ar[r] 
    \ar@<5ex>[d]^-{\iota_*} & H^{2e+2}_{\et}(Z, \Lambda(e+1)) \ar[d]^-{\iota_*} \\
  H^1_\et(X, \Lambda) \times H^{2d+1}_{\et,c}(X, \Lambda(d+1)) \ar[r]    
  \ar@<7ex>[u]^-{\iota^*} &  H^{2d+2}_{\et,c}(X, \Lambda(d+1)).}
\end{equation}

We now note that the map $\Hom(H^1_\et(X, \Lambda), H^{2d+2}_{\et,c}(X, \Lambda(d+1)))
\to H^1_\et(X, \Lambda)^\vee$, obtained by composing with $\tr_{X/k}$, is an isomorphism.
The same is also true for $Z$. Hence, ~\eqref{eqn:Gysin-SD-0} reduces the
proof of the lemma to showing that the diagram
\begin{equation}\label{eqn:Gysin-SD-1}   
  \xymatrix@C1.2pc{
H^{2e+2}_{\et}(Z, \Lambda(e+1)) \ar[r]^-{g'_*} \ar[d]_-{\iota_*} &
      H^2_\et(k', \Lambda(1)) \ar[d]_-{\pi_*} \ar[dr]^-{\tr_{k'}} & \\
      H^{2d+2}_{\et,c}(X, \Lambda(d+1)) \ar[r]^-{f_*} & H^2_\et(k, \Lambda(1))
      \ar[r]^-{\tr_k} & \Lambda}
  \end{equation}
  is commutative. But this follows by Lemmas~\ref{lem:Gysin-Trace} and
  ~\ref{lem:Trace-PF}.  
\end{proof}

\subsection{Comparison of reciprocity and {\'e}tale realization maps}
\label{sec:RERM}
We shall now prove the main result of this section.
Let us recall that $k$ is a local field of characteristic exponent $p \ge 1$ and
$\Lambda = {\Z}/m$, where $m \in k^\times$.
We need the following base case.

\begin{lem}\label{lem:Rec-Real-field}
We have a commutative diagram
\begin{equation}\label{eqn:Rec-Real-field-0}
\xymatrix@C1pc{
  {K^M_1(k)}/m \ar[r]^-{\rho_k} \ar[d]_-{\Psi_k} & H^1_\et(k, \Lambda)^\vee
  \ar[d]^-{\omega_k} \\
  H^1(k, \Lambda(1)) \ar[r]^-{\epsilon^*_k} & H^1_\et(k, \Lambda(1)).}
\end{equation}
\end{lem}
\begin{proof}
Recall that $\rho_k$ is defined via the pairing
  \[
    {K^M_1(k)}/m \times H^1_\et(k, \Lambda) \xrightarrow{{\rm NR}_k \times {\rm id}}
    H^1_\et(k, \Lambda(1)) \times
    H^1_\et(k, \Lambda) \to H^2_\et(k, \Lambda(1)) \xrightarrow{\tr_k} \Lambda,
  \]
  where ${\rm NR}_k \colon {K^M_n(k)}/m \xrightarrow{\cong} H^n_\et(k, \Lambda(n))$
  is the Norm-residue isomorphism.
  In particular, ${\rm NR}_k = \omega_k \circ \rho_k$.
  We thus have to show that $\epsilon^*_k \circ \Psi_k = {\rm NR}_k$.
 But this is well known and elementary using the description of
  $\Psi_k$ given in \S~\ref{sec:MK-MC}.
\end{proof}

The following is the main result of \S~\ref{sec:RR}.

\begin{prop}\label{prop:Rec-Real}
  Let $X \in \Sm_k$ be of pure dimension $d$ and let $m \in k^\times$ be an integer. 
  Then the  diagram
  \begin{equation}\label{eqn:Rec-Real-0}
    \xymatrix@C1pc{
      {C(X)}/m \ar[r]^-{\rho_X} \ar[d]_-{\Psi_X} & {\pi^{\ab}_1(X)}/m
      \ar[d]^-{\omega_X} \\
      H^{2d+1}_{c}(X, \Lambda(d+1)) \ar[r]^-{\epsilon^*_X} &
      H^{2d+1}_{\et,c}(X, \Lambda(d+1))}
\end{equation}
is commutative and the vertical arrows are isomorphisms.
\end{prop}
\begin{proof}
We can assume that $X$ is integral. Using the natural isomorphism
\begin{equation}\label{eqn:Rec-Real-1}
{\pi^{\ab}_1(X)}/m \xrightarrow{\cong} \Hom(H^1_\et(X, \Lambda), \Lambda)
= H^1_\et(X, \Lambda)^\vee,
\end{equation}
we can replace ${\pi^{\ab}_1(X)}/m$ by $H^1_\et(X, \Lambda)^\vee$.
The right vertical arrow is now the Saito-Tate duality isomorphism.
The left vertical arrow is an isomorphism by \propref{prop:Tame-MCCS}.
It remains to show that ~\eqref{eqn:Rec-Real-0} is commutative.

Since the map
  ${\underset{x \in X_{(0)}}\bigoplus} {K^M_1(k(x))}/m \to {C(X)}/m$ is surjective by
  \lemref{lem:Idele-Tame-prime-to-p}, it suffices to show that ~\eqref{eqn:Rec-Real-0}
  commutes after we compose it with the canonical map
  ${K^M_1(k(x))}/m \to {C(X)}/m$ for every $x \in X_{(0)}$.
  We therefore fix an arbitrary closed point $x \in X$ and let
  $\iota \colon \Spec(k(x)) \inj X$ be the inclusion.

We now look at the diagram
\begin{equation}\label{eqn:Rec-Real-2}
\xymatrix@C1pc{  
  {K^M_1(k(x))}/m \ar[rr]^-{\rho_x} \ar[dd]_-{\Psi_x} \ar[dr]^-{\iota_*} & &
  H^1_\et(k(x), \Lambda)^\vee \ar[dr]^{(\iota^*)^\vee} \ar[dd]_->>>>>>{\omega_x} & \\
  & {C(X)}/m \ar[dd]_->>>>>>{\Psi_X} \ar[rr]^->>>>>>>>>>>>>>{\rho_X} & &
  H^1_\et(X, \Lambda)^\vee \ar[dd]^->>>>>>>>{\omega_X} \\
  H^1(k(x), \Lambda(1)) \ar[dr]_-{\iota_*} \ar[rr]^->>>>>>>>>>>{\epsilon^*_x} & &
  H^1_\et(k(x), \Lambda(1)) \ar[dr]^-{\iota_*} & \\
  & H^{2d+1}_{c}(X, \Lambda(d+1)) \ar[rr]^-{\epsilon^*_X} & &
  H^{2d+1}_{\et,c}(X, \Lambda(d+1)).}
\end{equation}
We need to show that $\omega_X \circ \rho_X \circ \iota_* = \epsilon^*_X \circ
\Psi_X \circ \iota_*$. For this, it suffices to show that all faces of the above cube,
except possibly the front face, commute.

The top face commutes by the construction of the reciprocity map $\rho_X$.
Note here that under the isomorphism ~\eqref{eqn:Rec-Real-1}, the map
$(\iota^*)^\vee$ is identified with the push-forward map
$\iota_* \colon {\pi^{\ab}_1(\Spec(k(x)))}/m \to  {\pi^{\ab}_1(X)}/m$.
The bottom face commutes by \corref{cor:Gysin-Nis-etale-1}.
The left face commutes by \propref{prop:Tame-MCCS} and the right face commutes by
\lemref{lem:Gysin-SD}. Finally, the back face commutes by \lemref{lem:Rec-Real-field}.
This concludes the proof.
\end{proof}

\section{Finiteness of motivic and {\'e}tale cohomology}
\label{sec:BMC}
Our goal in this section is to prove some finiteness results for motivic and
{\'e}tale cohomology over local fields. We begin by proving a general result.

\begin{comment}
  For a Noetherian scheme $X$ over a field $k$ of characteristic exponent $p$,
we let
\begin{equation}\label{eqn:Fin-mot-coh}
  H^i(X, j) = {\underset{\ell \neq p}\bigoplus} H^i(X, {\Q_\ell}/{\Z_\ell}(j))
  = {\underset{\ell \neq p}\bigoplus} \ {\underset{n \ge 1}\varinjlim} \
  H^i(X, {\Z}/{\ell^n}(j));
\end{equation}
\[
  H^i_\et(X, j) = {\underset{\ell \neq p}\bigoplus} H^i_\et(X, {\Q_\ell}/{\Z_\ell}(j))
  = {\underset{\ell \neq p}\bigoplus} \ {\underset{n \ge 1}\varinjlim} \
  H^i_\et(X, {\Z}/{\ell^n}(j))
  \]
  for any pair of integers $i\geq 0$ and $j \in \Z$, 
where the sums are over all primes different from $p$. We define $ H^i_c(X, j)$ and
$H^i_{\et, c}(X, j)$ in similar fashion. 
\end{comment}

\subsection{A property of the realization map}\label{sec:R-iso}
Let $k$ be a field of characteristic exponent $p$.
We fix an integer $m \in k^\times$ and set $\Lambda = {\Z}/m$.
Let $cd(k)$ denote the {\'e}tale cohomological dimension of $k$ for 
sheaves which are torsion of finite exponents prime-to-$p$.

Recall from ~\eqref{eqn:RE-2} that
for any $X \in \Sch_k$, there is a natural {\'e}tale realization homomorphism
$\epsilon^*_X \colon H^i(X, \Lambda(j)) \to  H^i_\et(X, \Lambda(j))$. The naturality
of this map implies that we have a commutative diagram
\begin{equation}\label{eqn:Real-0}
  \xymatrix@C1pc{
H^i(X, \Lambda(j)) \ar[r]^-{\epsilon^*_X} \ar[d]_-{\pi^*}  & H^i_\et(X, \Lambda(j))
    \ar[d]^-{\pi^*} \\
    H^i(X', \Lambda(j)) \ar[r]^-{\epsilon^*_{X'}} & H^i_\et(X', \Lambda(j)),}
\end{equation}
where $\pi \colon X' = X \times_{\Spec(k)} \Spec(k') \to X$ is the projection map
for any field extension ${k'}/k$.

\begin{lem}\label{lem:Real-iso}
  Let $X \in \Sm_k$ be of pure dimension $d$. Then $\epsilon^*_X$ is an isomorphism
  when $j \ge {\min}\{i, d + cd(k)\}$. This map is injective when $i = j+1$.
\end{lem}
\begin{proof}
We can assume that $\Lambda = {\Z}/{\ell^n}$, where $\ell \neq \Char(k)$ is any prime
  and $n \ge 1$ is any integer.
  By \lemref{lem:SH-DM-Maps} and \cite[Cor.~2.1.7]{Elmanto-Khan},
  the motivic cohomology of $X$ is invariant under perfection.
  The similar result for {\'e}tale cohomology is classical
  (cf. \cite[Exp.~VIII, Thm.~1.1]{SGA4-Tome2}).
  In particular, $cd(k) = cd(k^\hs)$.
  Using ~\eqref{eqn:Real-0}, we can therefore assume that $k$ is perfect.

  Now, the Beilinson-Lichtenbaum conjecture
  (which is a now a theorem and follows from \cite{Voe-BK})
says that the {\'e}tale realization map
\begin{equation}\label{eqn: Real-iso-0}
  \epsilon^* \colon
  \Lambda(j) \to \tau_{\le j} {\bf R}\epsilon_* (\Lambda(j)_\et) \xrightarrow{\cong}
  \tau_{\le j} {\bf R}\epsilon_*(\Lambda(j))
\end{equation}
is an isomorphism for $j \ge 0$ in $\dm(k, \Lambda)$, where $\epsilon \colon 
\Sm_{k, \et} \to \Sm_{k, \nis}$ is the canonical morphism
from the {\'e}tale to the Nisnevich site of essentially of finite type smooth schemes
over $k$. We can now conclude the lemma by observing that
${\bf R}\epsilon_*(\Lambda(j)) \to \tau_{\le j} {\bf R}\epsilon_*(\Lambda(j))$
is an isomorphism when $j \ge d + cd(k)$ because the {\'e}tale cohomological
dimension of a $d$-dimensional finite type affine $k$-scheme is bounded by
$d + cd(k)$.
\end{proof}

\begin{lem}\label{lem:Real-iso-sing}
  Let $X \in \Sch_k$ be of dimension $d$. Then $\epsilon^*_X$ is an
  isomorphism
  when $j \ge {\min}\{i, d + cd(k)\}$. This map is injective when $i = j+1$.
\end{lem}
\begin{proof}
As in the proof of \lemref{lem:Real-iso}, we can assume that $k$ is perfect.
If $d = 0$, then $X$ must be smooth over $k$ and we are done by
\lemref{lem:Real-iso}.
  We shall now prove the general case by induction on $d$. Since the motivic
  and {\'e}tale cohomology are both nil-invariant, we can assume that $X$ is
  reduced.
  Since $\dm_\cdh(k, \Lambda)$ clearly satisfies the Mayer-Vietoris property for
  closed covers, and so does the {\'e}tale cohomology
  (cf. \cite[Exc.~VI.3.4]{Milne-etale}), we can apply this property to
  immediately reduce the proof to the case when $X$ is integral.

If $p =1$, then we can find a resolution of singularities
  $\pi \colon W \to X$. If we let $Z = X_\sing$ and $E = \pi^{-1}(Z)$, then
  $\{W \amalg Z \rightarrow X\}$ is a cdh cover of $X$ such that
  $\dim(Z) < \dim(X)$ and $\dim(E) < \dim(X)$.
 A combination of proper base change theorem and
support cohomology exact sequence  in {\'e}tale cohomology
(cf. \cite[Prop.~III.1.25, Cor.~VI.2.3]{Milne-etale}) implies that
$\dr^c_\et(k, \Lambda)$ satisfies
cdh-descent, where $\dr^c_\et(k, \Lambda)$ is the full subcategory of
$\dr_\et(k, \Lambda)$ consisting of complexes with constructible cohomology sheaves.
As $\dm_\cdh(k, \Lambda)$ clearly satisfies cdh-descent,
we therefore get an exact sequence 
\begin{equation}\label{eqn:Real-iso-3}
    \xymatrix@C.8pc{
{\begin{array}{l}
        H^{i-1}(W, \Lambda(j)) \\
       \hspace*{1.2cm} \oplus \\
       H^{i-1}(Z, \Lambda(j))
     \end{array}}
   \ar[r] &   H^{i-1}(E, \Lambda(j)) \ar[r] & H^i(X, \Lambda(j))
\ar[r] & {\begin{array}{l}
        H^{i}(W, \Lambda(j)) \\
       \hspace*{1.2cm} \oplus \\
       H^{i}(Z, \Lambda(j))
     \end{array}}
   \ar[r] &  H^{i}(E, \Lambda(j)),}
\end{equation}
which maps to the similar exact sequence of {\'e}tale cohomology. 
Comparing the two exact sequences and using
induction on $d$, it suffices to prove that $\epsilon^*_W$ is an isomorphism
for $i \le j$ and injective for $i = j+1$.
But this follows from \lemref{lem:Real-iso}.

We now assume that $p \ge 2$.
Using Temkin's strengthening of the alteration theorems of de Jong and Gabber
(cf. \cite[Thm.~1.2.5]{Temkin}) and the flatification theorem of Raynaud-Gruson
\cite[Thm.~5.2.2]{Gruson-Raynaud}, it is shown in the proof of
\cite[Prop.~6.2]{Binda-Krishna-JEP} that there exists a commutative diagram
\begin{equation}\label{eqn:Real-iso-1}
  \xymatrix@C1pc{
    W' \ar[r]^-{\pi'} \ar[d]_-{f'} & X' \ar[d]^-{f} \\
    W \ar[r]^-{\pi} & X,}
\end{equation}
where $f$ is the blow-up along a nowhere dense closed subscheme $Z \inj X$ and $f'$ is
the strict transform of $W$ along $f$. The map $\pi$ is a $p$-alteration of $X$
such that $W$ is integral and smooth over $k$ and $\pi'$ is a finite and flat
morphism between integral schemes of degree $p^r$ for some $r \ge 0$.

Since $\{X' \amalg Z \rightarrow X\}$ is a cdh cover of $X$, we can apply the above
characteristic zero
argument to reduce the proof to showing that the lemma holds for
${X'}$.
By applying this argument to $f'$ and using induction,
we see that the lemma holds for $W'$ because it
holds for $W$ by \lemref{lem:Real-iso}.

The assertion for $\epsilon^*_{X'}$ now follows from \cite[Prop.~6.3]{CD-Doc}
and \cite[Prop.~6.1.8(4)]{CD-Comp} which together say that there is a commutative
diagram
\begin{equation}\label{eqn:Real-iso-2}
  \xymatrix@C1pc{
    H^i(X', \Lambda(j)) \ar[r]^-{\pi'^*} \ar[d]_-{\epsilon^*_{X'}} &
    H^i(W', \Lambda(j)) \ar[r]^-{\pi'_*} \ar[d]^-{\epsilon^*_{W'}} &
    H^i(X', \Lambda(j)) \ar[d]^-{\epsilon^*_{X'}} \\
    H^i_\et(X', \Lambda(j)) \ar[r]^-{\pi'^*} &
    H^i_\et(W', \Lambda(j)) \ar[r]^-{\pi'_*}  &
    H^i_\et(X', \Lambda(j))}
\end{equation}
such that the composition of the horizontal arrows on each of the two rows is
multiplication by $p^r$. This concludes the proof.
\end{proof}

\begin{cor}\label{cor:Real-iso-sing-0}
  Let $X \in \Sch_k$ be of dimension $d$. Then the {\'e}tale realization
  map $\epsilon^*_X \colon H^i_c(X, \Lambda(j)) \to  H^i_{\et,c}(X, \Lambda(j))$ is an
  isomorphism when $j \ge {\min}\{i, d + cd(k)\}$. This map is injective when
  $i = j+1$. 
\end{cor}
\begin{proof}
  Let  $u \colon X \inj \ov{X}$ be a compactification and let $Z = \ov{X} \setminus X$
with the reduced subscheme structure.  
\lemref{lem:Ex-seq}(2) and
 its analogue in {\'e}tale cohomology 
  (cf. \cite[Rem.~III.1.30]{Milne-etale})
  allow us to assume that $X$ is
  complete. The latter case follows from \lemref{lem:Real-iso-sing}.
\end{proof}

\subsection{Finiteness of {\'e}tale cohomology over local fields}
\label{sec:BEC}
We now specialize to the case of schemes over a local field.
We fix a local field $k$ having the 
  ring of integers $\sO_k$, maximal ideal $\fm$  and residue field $\ff$.
  We let $p$ denote the characteristic of $\ff$
  (equivalently, the characteristic exponent of $\ff$, or
  the residue characteristic of $k$).
  We let $S = \Spec(\sO_k)$
  with generic point $\eta$ and closed point $s$. Recall that
  $cd(k) = 2$ and $cd(\ff) = 1$. We fix an integer $i \ge 0$.

  If $X$ is either a $k$-scheme or an $\ff$-scheme and $j \in \Z$, we let
\begin{equation}\label{eqn:Fin-mot-coh}
  H^i(X, j) = {\underset{\ell \neq p}\bigoplus} H^i(X, {\Q_\ell}/{\Z_\ell}(j))
  = {\underset{\ell \neq p}\bigoplus} \ {\underset{n \ge 1}\varinjlim} \
  H^i(X, {\Z}/{\ell^n}(j)).
\end{equation}
We define $ H^i_c(X, j)$ in a similar fashion. 
For any $S$-scheme $\sX$ (which includes $k$-schemes as well as $\ff$-schemes),
  we let
  \[
  H^i_\et(\sX, j) = {\underset{\ell \neq p}\bigoplus} H^i_\et(\sX, {\Q_\ell}/{\Z_\ell}(j))
  = {\underset{\ell \neq p}\bigoplus} \ {\underset{n \ge 1}\varinjlim} \
  H^i_\et(\sX, {\Z}/{\ell^n}(j)).
  \]
  We define $H^i_{\et, c}(\sX, j)$ in a similar fashion. 
  The direct sums in these definitions are over all primes different from $p$.

The key ingredient to prove our finiteness theorem for {\'e}tale cohomology
over $k$ (cf. \thmref{thm:Boundedness}) is the following analogous result
over finite fields due to Colliot-Th{\'e}l{\`e}ne-Sansuc-Soul{\'e}
\cite[Thm.~2]{CSS} and Kahn \cite[Thm.~1, 2]{Kahn-03}
(the reader may check
    that in the latter reference, the word `variety' means any finite type
    separated $\F_q$-scheme).

\begin{thm}\label{thm:Kahn}
  Let $X \in \Sm_{\F_q}$ be a smooth scheme of dimension $d$, where $q = p^r$ for some
  $r \ge 1$ and let $n$ be any integer.
  Then $H^i_\et(X, n)$ is finite if $n \notin [0,d]$ or $i \notin [n, 2n+1]$.
  If $X \in \Sch_{\F_q}$ is any scheme of dimension $d$, then $H^i_\et(X, n)$ is finite
  if $n \notin [0,d]$ or $i \notin [n, n+d+1]$.
  \end{thm}

We shall now prove the finiteness of {\'e}tale cohomology over local fields
  in a few steps,
  beginning with the following base case. 
%\vskip .2cm
We begin with the following extension of \lemref{lem:Roots-PB}.

\begin{lem}\label{lem:Roots-PB-local}
    Given a commutative diagram of Noetherian schemes
    \begin{equation}\label{eqn:Roots-PB-local-0}
      \xymatrix@C1pc{
        X \ar[r]^-{\iota'} \ar[d]_-{\pi'} & \sX \ar[d]^-{\pi} \\
        \Spec(\ff) \ar[r]^-{\iota} & S}
    \end{equation}
    with $\pi$ dominant,
    the canonical map $\iota'^* \mu_{m,\sX} \to \mu_{m,X}$ is an isomorphism on
    $X_\et$ if $m \in \ff^\times$.
  \end{lem}
  \begin{proof}
The proof of \lemref{lem:Roots-PB} applies verbatim to show that
    $\pi^* \mu_{m,S} \xrightarrow{\cong} \mu_{m,\sX}$. Since
    $\pi'^* \mu_{m,\ff} \xrightarrow{\cong} \mu_{m,X}$ by \lemref{lem:Roots-PB}, it
    suffices to show that $\iota^* \mu_{m,S} \xrightarrow{\cong} \mu_{m,\ff}$. But this
    is an elementary exercise using our hypothesis, which implies that
    $m \in \sO^\times_k$ and hence $(1 + \fm)$ is $m$-divisible,
    see the last part of the proof of \lemref{lem:Rigidity-Milnor}.
\end{proof}

\begin{lem}\label{lem:Etale-0}
    The group $H_{\et}^i(k, n)$ is finite when $n \neq 0,1$.
  \end{lem}
  \begin{proof}
    By \cite[Prop.~III.1.25]{Milne-etale} and
    \lemref{lem:Roots-PB-local}, we have an exact sequence
   \begin{equation}\label{eqn:Etale-0-0}
       H^i_\et(S, n) \to H^i_\et(k, n) \to H^{i+1}_{\et, s}(S, n).
 \end{equation}
   By the purity \cite[Thm.~2.1.1]{Fujiwara} and rigidity
   \cite[Cor.~VI.2.7]{Milne-etale} theorems, we have isomorphisms 
$H^{i-1}_{\et}(\ff, n-1) \xrightarrow{\cong} H^{i+1}_{\et, s}(S, n)$ and $H^i_\et(S, n) \xrightarrow{\cong} H^i_{\et}(\ff, n)$, respectively. The lemma now follows by using
these isomorphisms in \eqref{eqn:Etale-0-0} and applying \thmref{thm:Kahn}.
\end{proof}

Recall that a reduced Noetherian scheme $X$ of pure dimension $d$ with
 irreducible components $\{X_1, \ldots , X_r\}$ is called a normal crossing scheme
 if for every nonempty subset $J \subset [1, r]$, the scheme theoretic 
intersection $X_J := {\underset{i \in J}\bigcap} X_i$ is a regular scheme
which is either empty or of pure dimension $d + 1 - |J|$. It is clear from this
definition that if $X$ is a normal crossing scheme, then
$X_i \bigcap \ (\cup_{j \neq i} X_j)$ is also a normal crossing scheme (of smaller
dimension than that of $X$).

\begin{lem}\label{lem:Etale-1}
  Let $f \colon \sX \to S$ be a projective and flat morphism of relative dimension
  $d$ from a regular integral scheme. Let $Y \subset \sX$ be a normal crossing
  closed subscheme of pure codimension $c$ such that  $Y \subset \sX_s$. Then
  $H^i_{\et,Y}(\sX, n)$ is finite if $n \notin [c, d+1]$.
    \end{lem}
\begin{proof}
We shall prove the lemma by a double induction on the dimension and the number of
      irreducible components of $Y$. We can assume that $Y$ is reduced.
      If $Y$ is regular (e.g., if $\dim(Y) = 0$), then 
      %Gabber's purity theorem 
      \cite[Thm.~2.1.1]{Fujiwara} 
      implies that $H^i_{\et,Y}(\sX, n) \cong H^{i-2c}_{\et}(Y, n-c)$ and the
      latter group is finite in the given range by \thmref{thm:Kahn} because
      $n-c \notin [0, \dim(Y)]$.

In the general case, we write
      $Y = Y' \bigcup Z$, where $Y'$ is an irreducible component of $Y$ and $Z$ is the
      union of its remaining irreducible components. We have
      the Mayer-Vietoris exact sequence (cf. \cite[Exc.~VI.3.4]{Milne-etale})
      \[
       H^{i}_{\et, Y'}(\sX, n) \oplus  H^{i}_{\et, Z}(\sX, n) \to
       H^{i}_{\et, Y}(\sX, n) \to H^{i+1}_{\et, Y' \cap Z}(\sX, n).
     \]
The term on the left is finite by induction on the number of irreducible components
     of $Y$ and the term on the right is finite by induction on $\dim(Y)$ because
     $Y' \cap Z$ is a normal crossing scheme of strictly smaller dimension as we
     observe before. It follows that the middle term is finite.
\end{proof}

Recall that an integral and smooth  projective $k$-scheme $X$ is said to admit a
semistable reduction if there exists a regular integral scheme $\sX$ together with
a faithfully flat projective morphism $f \colon \sX \to S$ such that $\sX_\eta = X$
and $(\sX_s)_\red$ is a strict normal crossing divisor on $\sX$. It is a consequence of
    \cite{Deligne-Mumford} and \cite[Thm.~0.1, Cor.~0.4]{Cossart-J-S}
    (see also \cite{Cossart-J-S-2}) that $X$ admits a semistable reduction if
its dimension is at most one.

\begin{lem}\label{lem:Etale-fin-ss}
  Let $X \in \Sm_k$ be an integral projective scheme of dimension $d$ which admits a
  semistable reduction. Then $H^i_\et(X, n)$ is finite if $n \notin [0,d+1]$.
\end{lem}
\begin{proof}
Let $f \colon \sX \to S$ be a semistable reduction of $X$ with the reduced closed 
fiber $Y$. The cohomology with support exact sequence and the rigidity theorem
(cf. \cite[Prop.~III.1.25, Cor.~VI.2.7]{Milne-etale}) give an exact sequence
      \begin{equation}\label{eqn:Etale-1-0}
        H^i_\et(Y, n) \to H^i_\et(X, n) \to H^{i+1}_{\et, Y}(\sX, n).
      \end{equation}
      In the given range, the left term is finite by \thmref{thm:Kahn} and the
      right term is finite by \lemref{lem:Etale-1}. Hence, the middle term is finite.
    \end{proof}

\begin{lem}\label{lem:Etale-fin}
      Let $X \in \Sch_k$ be a projective scheme of dimension $d$. Then
      $H^i_\et(X, n)$ is finite if $n \notin [0,d+1]$.
    \end{lem}
    \begin{proof}
We shall prove the lemma by induction on $d$. If $d = 0$, then
      we are done by \lemref{lem:Etale-0}. We shall now assume that $d > 0$.

By \cite[Thm.~4.3.1]{Temkin}, we can find a finite field extension ${k'}/k$
      and a smooth projective integral $k'$-scheme $W$ together with a projective
      morphism $\pi \colon W \to X$ such that the following hold.
      \begin{enumerate}
      \item
        $k' = k$ and $\pi$ is birational if $\Char(k) = 0$.
      \item
        The degrees  $[k':k]$ is some power of $p$
        if $\Char(k) = p$.
      \item The morphism 
        $\pi$ is a $p$-alteration (an alteration {\`a} la de Jong whose degree is a
        $p$-power) if $\Char(k) = p$.
      \item
        $W \to \Spec(k')$ admits a semistable reduction.
      \end{enumerate}

     \lemref{lem:Etale-fin-ss} implies that $H^i_\et(W, n)$ is finite.
     If $\Char(k) = 0$, then a Mayer-Vietoris argument (cf. ~\eqref{eqn:Real-iso-3}
     in the
      proof of \lemref{lem:Real-iso-sing}) and induction on $d$ reduce the proof to
      the finiteness of $H^i_\et(W, n)$. 
      If $\Char(k) = p$, we can again repeat
      the second part of the proof of \lemref{lem:Real-iso-sing}
      (cf. ~\eqref{eqn:Real-iso-1} and ~\eqref{eqn:Real-iso-2}) to
      conclude that $H^i_\et(X,n)$ is finite.
 \end{proof} 

\vskip .2cm

We can now prove the main finiteness result over local fields.

\begin{thm}\label{thm:Boundedness}
  Let $k$ be a local field and $X \in \Sch_k$ be of dimension $d \ge 0$. Let $n$ be any
  integer. Then $H^i_c(X, n)$ and $H^i_{\et,c}(X, n)$ are finite if $n \ge d+2$.
\end{thm}
\begin{proof}
  We prove the theorem by induction on $d$.
  If $d= 0$, we are through by Lemmas~\ref{lem:Real-iso-sing} and
  ~\ref{lem:Etale-0}. 
  In general, we can find a nowhere dense closed subscheme $Z \subset X$
  whose complement is affine. By using induction and \lemref{lem:Ex-seq}, we can 
  assume that $X$ is quasi-projective. In the latter case, we can use a similar
  argument to assume that $X$ is projective. The theorem now follows by
  Lemmas~\ref{lem:Real-iso-sing} and ~\ref{lem:Etale-fin} because $cd(k) = 2$. 
\end{proof}

\begin{cor}\label{cor:Boundedness-0}
  Let the notations be as in \thmref{thm:Boundedness} and let $n\ge d+2$ be an
  integer. Then there exists an integer
  $M(n) \gg 0$ such that $|H^i_c(X, {\Z}/m(n))| \le M(n)$ and
  $|H^i_{\et,c}(X, {\Z}/m(n))| \le M(n)$
  for all integers $m$ not divisible by $p$.
\end{cor}
\begin{proof}
  Use \thmref{thm:Boundedness} and the exact triangle
  ${\Z}/m(n) \to ({\Q}/{\Z})'(n) \xrightarrow{m} ({\Q}/{\Z})'(n)$,
  where $({\Q}/{\Z})' = {\underset{\ell \neq p}\bigoplus} {\Q_\ell}/{\Z_\ell}$.
\end{proof}

\subsection{Finiteness in characteristic zero}\label{sec:Fin-0}
Let us now assume that $k$ is a local field of characteristic zero. In this case,
\corref{cor:Boundedness-0} is not quite sufficient for us because it excludes
the $p$-primary part of the motivic and {\'e}tale cohomology. We shall prove the
following weaker version of \corref{cor:Boundedness-0} which will be enough for our
purpose.

\begin{prop}\label{prop:Fin-char-0}
Let $X \in \Sch_k$ be of dimension $d \ge 0$. Then there exists an integer
  $M \gg 0$ such that $|H^{2d+2}_c(X, {\Z}/m(d+2))| \le M$ and
  $|H^{2d+2}_{\et,c}(X, {\Z}/m(d+2))| \le M$
  for all integers $m \neq 0$.
\end{prop}
\begin{proof}
  We can assume that $X$ is a reduced scheme. By \corref{cor:Real-iso-sing-0},
  we need to prove the proposition only for the {\'e}tale cohomology.
  Using the argument of \thmref{thm:Boundedness} and the fact that
  $cd_k(Z) \le 2\dim(Z) + 2$ for any $Z \in \Sch_k$
  (cf. \cite[Chap.~II, Prop.~12]{Serre-GC} and \cite[Exp.~X, Cor.~4.3]{SGA4}),
  we reduce to the case when $X$ is projective over $k$.

Assume first that $d = 0$. In this case, it suffices to prove the proposition
  for $\Spec(k)$. 
  The norm-residue isomorphism ${\rm NR}_k \colon {K^M_2(k)}/{m}
 \xrightarrow{\cong} H^{2}_{\et}(k, {\Z}/m(2))$ now reduces the problem to showing the
uniform boundedness of ${K^M_2(k)}/{m}$. 
To this end, note that we can write $K^M_2(k) = F \oplus D$, where $F$ is a finite
cyclic group and $D$ is a divisible group by \cite{Tate-Kyoto} (cf. \cite{Merkurjev}
for a stronger statement).
Hence, we have a surjection $F \surj H^2_\et(k, {\Z}/m(2))$ for all integers
$m > 0$. Therefore, we get $|H^2_\et(k, {\Z}/m(2))| \le |F|$ for all $m > 0$.

We now assume that $X$ is a smooth projective scheme over $k$ and prove the
proposition by induction on $d = \dim(X)$. We can assume $X$ to be integral.
If $d = 1$, we choose a dense affine open
$U \subset X$ and let $S$ be the complementary closed subset with the reduced
subscheme structure. We then get an exact sequence
\begin{equation}\label{eqn:Surface-0-0}
  H^4_{\et, S}(X, {\Z}/m(3)) \to H^4_\et(X, {\Z}/m(3)) \to H^4_\et(U, {\Z}/m(3)).
\end{equation}
Since $U$ is an affine curve over $k$, its {\'e}tale cohomological
dimension is at most three (cf. \cite[Chap.~II, Prop.~12]{Serre-GC} and
\cite[Tag~0F0P]{SP}). In particular, the last term in the above
exact sequence is zero. On the other hand,
we have $H^4_{\et, S}(X, {\Z}/m(3)) \cong H^2_\et(S, {\Z}/m(2))$
by excision (cf. \cite[Prop.~III.1.27]{Milne-etale})
and purity for {\'e}tale cohomology. Since $S$ is the spectrum of a
finite product of local fields, we are done by the $d = 0$ case solved above.

If $d \ge 2$, we can use the Altman-Kleiman Bertini theorem \cite[Thm.~1]{AK}
to find a smooth hypersurface section (in a chosen embedding of $X$ in some
projective space over $k$) $Y \subset X$. Note that $Y$ is necessarily integral
by Grothendieck's Lefschetz hyperplane theorem for $\pi_0(-)$. Furthermore,
$U := X \setminus Y$ is affine. We thus have an exact sequence
\begin{equation}\label{eqn:Surface-0-1}
  H^{2d}_\et(Y, {\Z}/m(d+1)) \to H^{2d+2}_\et(X, {\Z}/m(d+2)) \to
  H^{2d+2}_\et(U, {\Z}/m(d+2))
\end{equation}
by the localization and purity theorems. The last term is zero because
$2d + 2 > d+2 = cd_k(U)$ and the first term is uniformly bounded by induction on $d$.
This proves the proposition when $X$ is smooth and projective over $k$.

If $X$ is projective over $k$ but not necessarily smooth, we can choose
a resolution of singularities $\pi \colon W \to X$. Then the
resulting exact sequence ~\eqref{eqn:Real-iso-3} reduces the problem
to showing the boundedness for $W$ because $cd_k(E) < 2d+1$ and $cd_k(Z) < 2d+2$.
This concludes the proof.
\end{proof}

\begin{remk}\label{remk:Fin-char-1}
  We do not expect \propref{prop:Fin-char-0} to be valid
  under the weaker conditions of \corref{cor:Boundedness-0}.
  For instance, $H^1_\et(k, {\Q_p}/{\Z_p}(2))$ is not a finite group if $k$ is a
  $p$-adic field (cf. \cite[Cor.~7.4.1]{K-book}).
 \end{remk}

\section{Some properties of $\L$-completions}\label{sec:Kernel}
The goal of this section is to prove \thmref{thm:Main-2} and
part (1) of \thmref{thm:Main-3}. We shall also prove some properties of
the completions of the class groups and {\'e}tale
fundamental groups with respect to some
sets of primes. These results will be key steps in the proofs of the
remaining main results.
We begin by recalling some known elementary results which will be used frequently
in the rest of the paper.

\subsection{Some results on completions of abelian groups}
\label{se:Completion}
Let $\M$ be a set of prime numbers
and $I_{\M}$ the set of all integers whose prime divisors belong to  $\M$.
Recall from \S~\ref{sec:Notn} that for an abelian group $A$,
we write $A_{\M} = {\varprojlim}_{m \in I_\M} A/{m}$.
We shall call this the $\M$-completion of
$A$. We shall say that $A$ is $\M$-divisible 
if it is divisible by every prime
in $\M$ (equivalently, divisible by every integer in $I_\M$).
$A$ will be called  uniquely $\M$-divisible if it is $\M$-divisible and
$\M$-torsion-free. We shall say that
$A$ is $\M$-torsion if every element of $A$ is annihilated by some element of $I_\M$.
We let $A_{\M-\tor}$ denote the $\M$-torsion subgroup of $A$.
The following two results are \cite[Lem.~7.7, 7.8]{JS-Doc}.

\begin{lem}\label{lem:JS-0}
    Let $A$ be an abelian group, $\{B_m\}_{m \in I_\M}$ a pro-abelian group and
    $\phi \colon \{A/m\} \to \{B_m\}$ a morphism of pro-abelian groups. Let
    $\wh{\phi} \colon {A}_\M \to \wh{B}$ be the induced homomorphism between the
    limits, where $\wh{B} = {\varprojlim}_{m \in I_\M} B_m $.
    Assume that there exists $N \in I_\M$ such that $N \wh{B}_\tor =
    N \Ker(\wh{\phi}) = 0$. Then $\bigcap_{m \in I_\M} mA$ is the maximal $\M$-divisible
    subgroup of $A$.
    \end{lem}

\begin{lem}\label{lem:JS-1}
  Let $0 \to D \to A \xrightarrow{\pi} T \to 0$ be a short exact sequence of
  abelian groups such that $D$ is $\M$-divisible and $T$ is torsion. Let
  $T = T' \oplus T''$, where $T''$ is $\M$-torsion and $T'$ has no $\M$-torsion
  element. Let $D' = \pi^{-1}(T')$. Then $D'$ is $\M$-divisible and
  $A = D' \oplus T''$.
    \end{lem}

\begin{lem}\label{lem:Finite-compln}
      Let $\{G_i\}_{i \ge 1}$ be an inverse system of finite abelian groups.
      Assume that there exists an integer $M > 0$ such that $|G_i| \le M$
      for all $i \ge 1$. Let $ G = {\varprojlim}_i \ G_i$ be the inverse
      limit. Then the projection map $G \to G_i$ is injective for all $i \gg 1$.
    \end{lem}
    \begin{proof}
      %This is elementary and we omit the proof (cf. \cite[Lem.~11.3]{GKR-arxiv}).
      One can easily check that there is an inverse system $\{G'_i\}$ with surjective
      transition maps such that
      $G'_i \subset G_i$ for every $i \ge 1$ and $G = {\varprojlim}_i \ G'_i$.
      The lemma easily follows from this observation.
    \end{proof}

\begin{cor}\label{cor:DIV-ab}
      Let $A$ be an abelian group and $\M$ a set of prime numbers. Suppose that
      there exists an integer $M \gg 0$ such that $|{A}/{m}| \le M$ for all
      $m \in I_\M$. Then $A \cong F \bigoplus D$, where $F$ is a finite group of
      order lying in $I_\M$ and $D$ is $\M$-divisible.
    \end{cor}
    \begin{proof}
      Combine Lemmas~\ref{lem:JS-0}, ~\ref{lem:JS-1} and ~\ref{lem:Finite-compln}.
      \end{proof}

We shall also need the following result.

\begin{lem}\label{lem:PB-PF-EFG}
  Let $f \colon X' \to X$ be a finite {\'e}tale morphism between
  Noetherian schemes such that $X$ is integral.
  Then there exists a continuous homomorphism $f^* \colon \pi^{\ab}_1(X) \to
  \pi^{\ab}_1(X')$
  such that $f_* \circ f^*$ is multiplication by $\deg(f)$ on $\pi^{\ab}_1(X)$.
  Furthermore, given a morphism of integral Noetherian schemes $\pi \colon Y \to X$
  and the canonical
  projections $X' \xleftarrow{\pi'} Y' = Y \times_X X' \xrightarrow{g} Y$, one has a commutative  diagram
  \begin{equation}\label{eqn:PB-Cart}
    \xymatrix@C1pc{
      \pi^{\ab}_1(Y) \ar[r]^-{g^*} \ar[d]_{\pi_*} & \pi^{\ab}_1(Y') \ar[d]^-{\pi'_*} \\
      \pi^{\ab}_1(X) \ar[r]^-{f^*} & \pi^{\ab}_1(X').}
  \end{equation}
\end{lem}
\begin{proof} We first assume that $X'$ is integral. 
Let $\sU(X, x)$ denote the Galois category of finite {\'e}tale covers of $X$ (with
  respect to a chosen base point $x$). We let $\sU(X', x')$ be defined similarly.
  Then the operation of composition with $f$ defines a canonical inclusion functor
  $f_* \colon \sU(X', x') \to \sU(X, f(x'))$. This in turn gives rise to a
  homomorphism $f^* \colon \pi^{\ab}_1(X) \to \pi^{\ab}_1(X')$. 
   For a non-connected  finite {\'e}tale
cover $f \colon X' \to X$, we define the desired pull-back homomorphism by considering the irreducible components of $X'$.
  The commutativity of ~\eqref{eqn:PB-Cart} follows because one easily checks that
  the corresponding diagram of the Galois categories of finite
  {\'e}tale covers commutes.

To show the remaining part of the lemma, we let $K$ (resp. $K'$) denote the
  function field of $X$ (resp. the ring of total quotients of $X'$).
  We then obtain a diagram
  \begin{equation}\label{eqn:PB-PF-EFG-0}
    \xymatrix@C1pc{
      G_{K} \ar[r]^-{f^*} \ar[d] & G_{K'} \ar[r]^-{f_*} \ar[d] & G_{K} \ar[d] \\
      \pi^{\ab}_1(X) \ar[r]^-{f^*} & \pi^{\ab}_1(X') \ar[r]^-{f_*} &
      \pi^{\ab}_1(X),}
  \end{equation}
  where the vertical arrows are induced by the inclusions of generic points of
  $X$ and $X'$. Note that $X'$ is a disjoint union of integral normal schemes and
  $K'$ is the product of the function fields of these integral schemes.
  The right square clearly commutes and the left square commutes by
  ~\eqref{eqn:PB-Cart} because $X' \times_X \Spec(K) \cong \Spec(K')$.

Since the vertical arrows in ~\eqref{eqn:PB-PF-EFG-0}
  are surjective because $X$ (and hence $X'$) is normal,
  it is enough to show that the composition of the top horizontal arrows is
  multiplication by $\deg(f)$. Since $\deg(f)$ is the sum of degrees of the
  irreducible components of $X'$ over $X$, it suffices to show that if
  $X'_i$ is an irreducible component of $X'$ with function field $K'_i$, then
  the composite map $G_K \xrightarrow{f^*} G_{K'_i}  \xrightarrow{f_*} G_K$ is
  multiplication by $\deg({X'_i}/X)$. We can therefore assume that $X'$ is integral.

Now, it is classical that under the Pontryagin duality, the maps
  $G_K \xrightarrow{f^*} G_{K'}  \xrightarrow{f_*} G_K$
  are dual to the maps between Galois cohomology groups
  \begin{equation}\label{eqn:PB-PF-EFG-1}
H^1(G_{K}, {\Q}/{\Z}) \xrightarrow{{\rm Res}} H^1(G_{K'}, {\Q}/{\Z}) 
\xrightarrow{{\rm Cor}} H^1(G_{K}, {\Q}/{\Z}),
\end{equation}
where ${\rm Res}$ and ${\rm Cor}$ are, respectively, the
restriction and corestriction homomorphisms between the Galois cohomology groups.
Hence, it suffices to show that ${\rm Cor} \circ {\rm Res}$ is multiplication by
$\deg(f)$.
To show this, note that we can replace the abelian Galois groups in
~\eqref{eqn:PB-PF-EFG-1} by the corresponding absolute (nonabelian) Galois groups
without changing the Galois cohomology.
The claim then follows by \cite[Prop.~4.2.10]{Gille-Szamuely}.
\end{proof}

\subsection{The finiteness theorem for $\pi^{\ab, \tm}_1(X)_0$}
\label{sec:GYT}
For the rest of \S~\ref{sec:Kernel}, we fix a local field $k$, a smooth and connected
projective $k$-scheme $\ov{X}$ of dimension
$d$ and a nonempty open immersion $j \colon X \inj \ov{X}$. We let
$\iota \colon Z \inj \ov{X}$ be the inclusion of the complementary closed subset
endowed with the reduced subscheme structure. 
We let $\ff$ be the residue field of $k$.
Let $\P$ be the set of primes different from
$\Char(k)$ and $\L$ the set of primes 
different from $\Char(\ff)$. In particular, $\P = \L$ if $\Char(k) > 0$ and
$\P = \L \cup \{\Char(\ff)\}$ if $\Char(k) = 0$.

\vskip .2cm

We shall now prove \thmref{thm:Main-2} and part of \thmref{thm:Main-3},
and give some applications.
The precise result is the following.

\begin{thm}\label{thm:J-fin}
    The group $\pi^{\ab, \tm}_1(X)$ satisfies the following properties.
    \begin{enumerate}
    \item
      In the short exact sequence
    \[
      0 \to J^\tm(X) \to \pi^{\ab, \tm}_1(X) \xrightarrow{\tau_{\ov{X}}} \pi^{\ab}_1(\ov{X})
      \to 0,
    \]
    the left group $J^\tm(X)$ is finite whose order is invertible in $k$.
    \item
    If $X$ is geometrically connected, then
    \[
      \pi^{\ab, \tm}_1(X)_0 \cong F \oplus \wh{\Z}^r,
    \]
    where $F$ is a finite group and $r$ is the $\ff$-rank of the special fiber of
    the N{\'e}ron model of $\Alb(\ov{X})$.
  \end{enumerate}
\end{thm}
\begin{proof}
  We shall follow the notations of \S~\ref{sec:Reln} while proving (1).
  By \corref{cor:Pi-decom-0}, we have a short exact sequence
  \begin{equation}\label{eqn:J-fin-0}
    0 \to J^\tm(X) \to {\pi^{\ab, \tm}_1(X)}_{\P} \xrightarrow{\tau_{\ov{X}}}
    {\pi^{\ab}_1(\ov{X})}_{\P} \to 0.
  \end{equation}

A combination of
    \propref{prop:Rec-Real}, (the {\'e}tale analogue of) \lemref{lem:Ex-seq},
    \corref{cor:Boundedness-0} and \propref{prop:Fin-char-0} implies that we
    have a compatible system of short exact sequences
 \begin{equation}\label{eqn:J-fin-1}
    0 \to A_m \to {\pi^{\ab, \tm}_1(X)}/m \xrightarrow{\tau_{\ov{X}}}
    {\pi^{\ab}_1(\ov{X})}/m \to 0
    \end{equation}   
    as $m$ varies through integers invertible in $k$ such that
    $|A_m| \le M$, where $M$ is an integer not depending on $m$.
    Taking the inverse limit, we get an exact
    sequence
    \begin{equation}\label{eqn:J-fin-2}
      0 \to {\varprojlim}_m A_m \to {\pi^{\ab, \tm}_1(X)}_{\P}
      \xrightarrow{\tau_{\ov{X}}} {\pi^{\ab}_1(\ov{X})}_{\P} \to 0.
    \end{equation}   
    The uniform boundedness of $\{|A_m|\}$ implies that
    ${\varprojlim}_m A_m$ is finite (cf. \lemref{lem:Finite-compln}).
    Comparing this sequence with
    ~\eqref{eqn:J-fin-0}, we conclude that $J^\tm(X)$ is a finite group whose order
    is invertible in $k$.

    To prove part (2), note that our hypothesis implies that $\ov{X}$ is geometrically
    connected. We also note that 
\begin{equation}\label{eqn:J-fin-3}
  0 \to J^\tm(X) \to \pi^{\ab, \tm}_1(X)_0 \xrightarrow{\tau_{\ov{X}}}
  \pi^{\ab}_1(\ov{X})_0 \to 0
\end{equation} is exact. 
%Since $\wh{\Z}^r$ is a projective pro-finite group, \cite[Thm.~1.1]{Yoshida03} yield that $\pi^{\ab}_1(\ov{X})_0 \cong F_1\oplus \wh{\Z}^r$, where is $r$ is as in the theorem and $F_1$ is a finite group. We therefore get a surjection 
%$\theta\colon \pi^{\ab, \tm}_1(X)_0 \to \wh{\Z}^r$ with an exact sequence 
%$ 0 \to J^t(X) \to \Ker(\theta) \to F_1 \to 0$.
%We now apply part (1) of the proposition to conclude that $\Ker(\theta)$ is a finite group. The 
%desired assertion about
%$\pi^{\ab, \tm}_1(X)_0$ follows because $\wh{\Z}^r$ is a projective pro-finite group. 
%This  conclude the proof of the theorem. 
By combining  ~\eqref{eqn:J-fin-3} with \cite[Thm.~1.1]{Yoshida03}, the snake lemma gives a surjection 
$\theta\colon \pi^{\ab, \tm}_1(X)_0 \surj \wh{\Z}^r$ together with an exact sequence 
$ 0 \to J^t(X) \to \Ker(\theta) \to F_1 \to 0$, where $r$ is as in the theorem and $F_1$ is a finite group. By part (1), we 
obtain an exact sequence
\begin{equation*}\label{eqn:J-fin-3.5}
  0 \to F \to \pi^{\ab, \tm}_1(X)_0 \to \wh{\Z}^r  \to 0,
\end{equation*}
where $F= \Ker(\theta)$ is a finite group. The desired decomposition of $\pi^{\ab, \tm}_1(X)_0$  now  follows easily from the universal property of $\wh{\Z}^r$
(e.g., apply \cite[Lem.~3.2.1]{Pro-fin} with $G = \mathbb{Z}^r, \ H =\pi^{\ab, \tm}_1(X)_0 $
and $\phi$ being a homomorphism induced by the canonical inclusion $\Z^r
\inj \wh{\Z}^r$ and by the  freeness of $\Z^r$).
\end{proof}

The following corollary proves part of \thmref{thm:Main-3} whose complete proof will
follow from  \thmref{thm:Main-8*}.

\begin{cor}\label{cor:J-fin-1}
  Assume that $X$ is geometrically connected. Then the image of the degree zero
  reciprocity map
  $\rho^t_{X,0} \colon C^\tm(X)_0 \to \pi^{\ab, \tm}_1(X)_0$ is finite.
  Furthermore, $\coker_{\rm top}(\rho^\tm_X)$ is the direct sum of $\wh{\Z}^r$ and a
  finite group.
\end{cor}
\begin{proof}
  Combine \thmref{thm:J-fin} and \cite[Cor.~2.3(1,4)]{Forre-Crelle}.
  Note that the proof of the cited result actually shows that
  $\coker_{\rm top}(\rho_{\ov{X}})$ is the direct sum of $\wh{\Z}^r$ and a finite group.
\end{proof}

The following application of \thmref{thm:J-fin} strengthens
\propref{prop:TFG-5}.

\begin{thm}\label{thm:Tame-nr-main}
There is an exact sequence of abelian groups
  \[
    {\underset{C \in \sC(X)}\bigoplus} \frac{G^{(0)}_{\Delta(C_n)}}{G^{(1)}_{\Delta(C_n)}} \to
    \pi^{\ab, \tm}_1(X) \xrightarrow{\tau_{\ov{X}}}
    \pi^{\ab}_1(\ov{X}) \to 0.
  \]
\end{thm}
\begin{proof}
  Combine \propref{prop:TFG-5} and part (1) of \thmref{thm:J-fin}.
  \end{proof}

\subsection{Torsion in the image of degree zero tame reciprocity}
\label{sec:L-compl-tor}
We continue with the setup of \S~\ref{sec:GYT}. We begin by stating the following
consequence of the construction of the reciprocity homomorphism
in \propref{prop:REC-T}. We shall refer to the diagram ~\eqref{eqn:REC-Diag} very
often in the rest of the paper.
\begin{cor}\label{cor:REC-T-Norm}
 There is a commutative diagram
  \begin{equation}\label{eqn:REC-Diag}
    \xymatrix@C1pc{
      0 \ar[r] & C^\tm(X)_0 \ar[r] \ar[d]_-{\rho^\tm_{X,0}} & C^\tm(X) \ar[r]^-{N_X}
      \ar[d]^-{\rho^\tm_X} & k^\times \ar[d]^-{\rho_k} \\
      0 \ar[r] & \pi^{\ab,\tm}_1(X)_0 \ar[r] & \pi^{\ab,\tm}_1(X) \ar[r]^-{\pi_X} & G_k.}
  \end{equation}
\end{cor}
\begin{proof}
We only need to show the commutativity of the right
  square.
  But this follows from the construction of various maps in the square
  and \cite[\S~3.2, Cor.~1]{Kato80}.
\end{proof}

Suppose that ${k'}/k$ is a finite extension and let
$X' = X \times_{\Spec(k)} \Spec(k')$.
Let $\pi \colon \Spec(k') \to \Spec(k)$ and $f \colon X' \to X$ be the projections.
Then we get a diagram
\begin{equation}\label{eqn:PF-Rec*}
  \xymatrix@C1.5pc{
    C^\tm(X') \ar[rr]^-{f_*} \ar[dd]_-{\rho^\tm_{X'}} \ar[dr]^-{N_{X'}} & &
    C^\tm(X) \ar[dd]_->>>>>>{\rho^\tm_X} \ar[dr]^-{N_X} & \\
    & k'^\times \ar[rr]^->>>>>>>>{N_{{k'}/k}} \ar[dd]_->>>>>>>{\rho_{k'}} & &
    k^\times \ar[dd]^-{\rho_k} \\
    \pi^{\ab,\tm}_1(X') \ar[rr]^->>>>>>>{f_*} \ar[dr]_-{\pi_{X'}} & &
    \pi^{\ab,\tm}_1(X) \ar[dr]^-{\pi_X} & \\
    & G_{k'} \ar[rr]^-{\pi_*} & & G_k.}
  \end{equation}

\begin{lem}\label{lem:PF-Rec-Commute}
  The above diagram is commutative.
\end{lem}
\begin{proof}
  The front, back, left and right faces commute by \corref{cor:REC-T-PF} and
  ~\eqref{eqn:REC-Diag}.
  The bottom face clearly commutes. To show the commutativity of the top face,
  we can assume that $X$ is the spectrum of a field, in which case it is a standard.
\end{proof}

\begin{lem}\label{lem:Coker-0}
  The image of $\rho^\tm_{X,0}$ is a torsion group of finite exponent.
  In particular, $({\rm Image}(\rho^\tm_X))_\tor$ has finite exponent.
  \end{lem}
  \begin{proof}
    The second assertion follows from the first since $(G_k)_\tor$ is finite.
    We shall therefore prove the first assertion.
By \thmref{thm:J-fin}, \corref{cor:REC-T-PF} and the surjection
  $C^\tm(X)_0 \surj C(\ov{X})_0$, it suffices to show that
  the image of $\rho_{\ov{X},0}$ is a torsion group of finite exponent.
If $\ov{X}$ is geometrically connected, then the image of $\rho_{\ov{X},0}$
  is a finite group by \cite[Cor.~2.3(1)]{Forre-Crelle}.
  If $\ov{X}$ is not geometrically connected, we can find a finite Galois
  extension ${k'}/{k}$ such that $\ov{X}' := \ov{X}_{k'}$ is the disjoint union of
  its irreducible components $\ov{X}'_1, \ldots , \ov{X}'_r$ each of which is
  geometrically connected. In particular, the lemma holds for $\ov{X}'$.

We let $f \colon \ov{X}' \to \ov{X}$ be the projection. By \lemref{lem:PF-Rec-Commute},
  we have a commutative diagram
  \begin{equation}\label{eqn:Coker-0-0}
    \xymatrix@C1pc{
      C(\ov{X}')_0 \ar[r]^-{f_*} \ar[d]_-{\rho_{\ov{X}',0}} & C(\ov{X})_0
      \ar[d]^-{\rho_{\ov{X},0}} \\
      \pi^{\ab}_1(\ov{X}')_0 \ar[r]^-{f_*} & \pi^{\ab}_1(\ov{X})_0.}
  \end{equation}
It follows from \corref{cor:Sm-PB} that the cokernel of the top horizontal
  arrow is of exponent $[k':k]$. Since the image of the left vertical arrow
  is finite, the image of the right vertical arrow must have finite exponent.
\end{proof}

\subsection{Torsion in the fundamental groups}
\label{sec:Tor-L-compl}
We shall now prove that the torsion subgroup of
the abelian fundamental group of a smooth scheme over a local field has
finite exponent. This will be a key step in proving \thmref{thm:Main-1}.
%We first prove this in the projective case.

\begin{lem}\label{lem:Proj-compl}
  The groups $\pi^{\ab}_1(\ov{X})_\tor$ and $(\pi^{\ab}_1(\ov{X})_\P)_\tor$ have finite
  exponents.
\end{lem}
\begin{proof}
  By \lemref{lem:Prod-dec}, it suffices to prove the assertion for
  $\pi^{\ab}_1(\ov{X})_\tor$.
We first assume that $\ov{X}$ is geometrically connected. In this case, the sequence
  \begin{equation}\label{eqn:Proj-compl-0}
    0 \to \pi^{\ab}_1(\ov{X})_0 \to \pi^{\ab}_1(\ov{X}) \to G_k \to 0
  \end{equation}
  is exact.
 We also have an exact sequence
 \begin{equation}\label{eqn:Proj-compl-1} 
   0 \to k^\times \xrightarrow{\rho_k} G_k \to {\wh{\Z}}/{\Z} \to 0
 \end{equation}
 which shows that $(k^\times)_\tor \cong (G_k)_\tor$.
 We can now apply \cite[Thm.~1.1]{Yoshida03} and \cite[Prop.~II.5.7]{Neukirch}
to conclude the proof.

If $\ov{X}$ is not geometrically connected, we have seen in the proof of
 \lemref{lem:Coker-0} that there exists a finite separable (in fact, Galois)
 extension ${k'}/k$ such that each irreducible component of 
 $\ov{X} \times_{\Spec(k)} \Spec(k')$ is geometrically connected. We choose one such
 irreducible component $\ov{X}'$ and let $f \colon \ov{X}' \to \ov{X}$ be the
 projection map. Then $f$ is a finite {\'e}tale morphism of $k$-schemes.
By \lemref{lem:PB-PF-EFG}, the composite map
 $\pi^{\ab}_1(\ov{X}) \xrightarrow{f^*} \pi^{\ab}_1(\ov{X}') \xrightarrow{f_*}
 \pi^{\ab}_1(\ov{X})$ is multiplication by $\deg(f)$. This implies that the
 composite map
 \begin{equation}\label{eqn:Proj-compl-4}
%   (\pi^{\ab}_1(\ov{X}))_\tor \xrightarrow{f^*} (\pi^{\ab}_1(\ov{X}'))_\tor
%   \xrightarrow{f_*} (\pi^{\ab}_1(\ov{X}))_\tor;
% 
% \[
\pi^{\ab}_1(\ov{X})_\tor \xrightarrow{f^*} \pi^{\ab}_1(\ov{X}')_\tor
   \xrightarrow{f_*} \pi^{\ab}_1(\ov{X})_\tor
 \end{equation}
 is also multiplication by $\deg(f)$. Since the middle term has finite exponent, it
 follows that the same holds for $\pi^{\ab}_1(\ov{X})_\tor$.
\end{proof}

\begin{lem}\label{lem:Open-compl}
  The groups $\pi^{\ab, \tm}_1({X})_\tor$ and $(\pi^{\ab}_1({X})_\P)_\tor$ have finite
  exponents.
\end{lem}
\begin{proof}
By \propref{prop:Pi-decom}, it suffices to prove the
  claim for $\pi^{\ab, \tm}_1({X})_\tor$.  But this follows by combining
  \thmref{thm:J-fin} and \lemref{lem:Proj-compl}, as
  the torsion functor is left exact.
\end{proof}

\begin{cor}\label{cor:Coker-1}
  The groups $\coker(\rho^\tm_{X,0})_\tor$ and $\coker(\rho^\tm_{X})_\tor$ 
have finite exponents.
\end{cor}
\begin{proof}
We first argue for the degree zero part.
  We let $F'$ be the image of $\rho^\tm_{X,0}$ and $F = \coker(\rho^\tm_{X,0})$ so that
  there is an exact sequence
  \begin{equation}\label{eqn:Coker-1-0}
    0 \to F' \to \pi^{\ab, \tm}_1(X)_0 \to F \to 0.
  \end{equation}
  Using \lemref{lem:Coker-0}, we see that $F' \otimes {\Q}/{\Z} = 0$. 
This implies that
  \[
    0 \to {F'}_\tor \to  (\pi^{\ab, \tm}_1(X)_0)_\tor \to F_\tor \to 0
  \]
  is exact. Hence, it suffices to show that $(\pi^{\ab, \tm}_1(X)_0)_\tor$ has finite
  exponent. But this follows from \lemref{lem:Open-compl}.

To prove the finite exponent property for $\coker(\rho^\tm_{X})_\tor$, we let
  $F = {\rm Image}(N_X)$ and $F' = {\rm Image}(\pi_X)$. Since $\rho_k$ is injective,
  its cokernel is uniquely divisible and ${\rm Coker}(N_X)$ is finite, it follows that
  \begin{equation}\label{eqn:Coker-1-1}
    0 \to \coker(\rho^\tm_{X,0}) \to \coker(\rho^\tm_{X}) \to {F'}/F \to 0
  \end{equation}
  is exact and $({F'}/F)_\tor$ is finite. The desired claim now follows from the
  left exactness of the torsion functor and the
  finite exponent of $\coker(\rho^\tm_{X,0})_\tor$ we showed above.
\end{proof}

\begin{remk}\label{remk:Fin-exp-fin}
  We shall show in \S~\ref{sec:FIN-EXP-FIN} that the finite exponent claims in
  the results of \S~\ref{sec:L-compl-tor} and \S~\ref{sec:Tor-L-compl} can be replaced
  by a finiteness assertion, at least away from $\Char(k)$.
\end{remk}

\subsection{The kernel of the reciprocity}\label{sec:Ker-Rec}
We continue with the setup of \S~\ref{sec:GYT}. We let $\rho^m_{X} \colon {C(X)}/m \to
{\pi^{\ab}_1(X)}/m$ be the reciprocity homomorphism mod-$m$, where $m > 0$ is an
integer. We define $\rho^{\tm, m}_X$ and $\rho^m_{\ov{X}}$ similarly.

\begin{lem}\label{lem:Ker-bound}
There exists an integer $M \gg 0$ such that $|\Ker(\rho^m_{X})| \le M$ for all
$m \in I_{\L}$.
\end{lem}
\begin{proof}
 For $\ov{X}$, the lemma follows from \cite[Thm.~5.4]{Forre-Crelle}. Indeed,
  note that none of the groups $C(\ov{X}) \cong SK_1(\ov{X})$ and $\pi^{\ab}_1(\ov{X})$
  depends on $k$. Hence, we can replace $k$ by $k' := H^0(\ov{X}, \sO_{\ov{X}})$.
Observe that $k'$ is a finite separable field extension of $k$ because
$X$ is integral and smooth over $k$, in particular, geometrically reduced.
Moreover, the structure map $X \to \Spec(k)$ canonically factors through $\Spec(k')$
and $\ov{X}$ is geometrically connected over $k'$ (cf. \cite[Tag~0366]{SP}). We
can therefore assume that $\ov{X}$ is geometrically connected.
In this case, there is an exact sequence
  \[
    {\Ker(\rho_{\ov{X}})}/m \to \Ker(\rho^m_{\ov{X}}) \to
    \ov{{}_m \coker(\rho_{\ov{X}})} \to 0
  \]
  for every integer $m > 0$, where the term on the right is a quotient of
  ${}_m \coker(\rho_{\ov{X}})$.
  The term on the left is uniformly finite as $m$ runs through $I_{\L}$
  by \cite[Thm.~5.4]{Forre-Crelle} and the
    term on the right is uniformly finite over $I_\P$ by \cite[Thm.~1.1]{Yoshida03}.

We shall now prove the lemma for $X$.
By Propositions~\ref{prop:Tame-MCCS} and ~\ref{prop:Rec-Real}, the lemma is equivalent
to the statement in motivic cohomology that the kernel of the {\'e}tale realization map
  $\epsilon^*_X \colon H^{2d+1}_c(X, {\Z}/m(d+1)) \to H^{2d+1}_{\et, c}(X, {\Z}/m(d+1))$
  is uniformly finite as $m$ runs through $I_{\L}$. To prove this new assertion, we
  consider the commutative diagram
  \begin{equation}\label{eqn:Ker-bound-0}
    \xymatrix@C1pc{
H^{2d}(Z, {\Z}/m(d+1)) \ar[r] \ar[d]_-{\epsilon^*_Z} & H^{2d+1}_c(X, {\Z}/m(d+1))
\ar[r] \ar[d]^-{\epsilon^*_X} & H^{2d+1}(\ov{X}, {\Z}/m(d+1))
\ar[d]^-{\epsilon^*_{\ov{X}}} \ar[r] & 0 \\
 H^{2d}_\et(Z, {\Z}/m(d+1)) \ar[r] & H^{2d+1}_{\et, c}(X, {\Z}/m(d+1))
 \ar[r] & H^{2d+1}_\et(\ov{X}, {\Z}/m(d+1)) \ar[r] & 0.}
\end{equation}

The top row is exact by \lemref{lem:Ex-seq} and the bottom row is the 
localization sequence in {\'e}tale cohomology. Since $H^{2d+1}(\ov{X}, {\Z}/m(d+1))$
remains invariant under perfection by \lemref{lem:SH-DM-Maps} and
\cite[Cor.~2.1.7]{Elmanto-Khan}, this term is zero by 
\cite[Thm.~5.1]{Krishna-Pelaez-AKT}.
The term on the right end of the bottom row is zero because $cd_k(Z) \le 2d$.
We have shown above (this also uses Propositions~\ref{prop:Tame-MCCS} and
~\ref{prop:Rec-Real}) that $\Ker(\epsilon^*_{\ov{X}})$ is uniformly finite
over $I_{\L}$. We now apply \corref{cor:Boundedness-0} and conclude.
\end{proof}

For any set of primes $\M$, we let ${\Ker}(\rho_X)^\dagger_\M =
{\varprojlim}_{m \in I_\M} \Ker(\rho^m_{X})$.
It follows from Lemmas~\ref{lem:Finite-compln} and ~\ref{lem:Ker-bound} that
\begin{equation}\label{eqn:Ker-bound-1}
|{\Ker}(\rho_X)^\dagger_{\L}| \le M  \ \mbox{for \ some \ integer} \  M \gg 1.
\end{equation}
Furthermore, there exists an exact sequence
\begin{equation}\label{eqn:Ker-bound-2}
  0 \to {\Ker}(\rho_X)^\dagger_{\L} \to C(X)_{\L} \to \pi^{\ab}_1(X)_{\L}.
\end{equation}

\begin{lem}\label{lem:Ker-div}
  The kernels of the ${\L}$-completion maps $C(X) \to C(X)_{\L}$ and
$C^\tm(X) \to C^\tm(X)_{\L}$ are ${\L}$-divisible.
\end{lem}
\begin{proof}
We only need to show that the kernel of $C(X) \to C(X)_{\L}$ is ${\L}$-divisible.
This is proven by using $\rho_X$ and \lemref{lem:JS-0} as follows. We set $A = C(X)$,
$B_m = {\pi^{\ab}_1(X)}/m$ and $\phi = \{\rho^m_X\}$ with $m \in I_{\L}$.
Then we get $\wh{B} = {\pi^{\ab}_1(X)}_{\L}$ in the notations of
\lemref{lem:JS-0}. It follows from \lemref{lem:Open-compl} that
$M \wh{B}_\tor = 0$ for some integer $M \gg 0$. Using the left exactness of the
inverse limit functor and ~\eqref{eqn:Ker-bound-1}, we can choose $M$ large enough
so that we also have $M (\Ker(\wh{\phi})) = 0$.
Note that it follows from ~\eqref{eqn:Ker-bound-2} that
$M$ necessarily lies in $I_{\L}$ in this case. We now apply \lemref{lem:JS-0} to
conclude the proof.
\end{proof}

\section{End of proofs under bad reduction}\label{sec:End-1}
In this section, we shall complete the proofs of all our main results except
Theorem~\ref{thm:Main-4}. We shall also prove refinements of
the results of \S~\ref{sec:L-compl-tor} and \S~\ref{sec:Tor-L-compl}.
We let $k$ be a local field with residue field $\ff$ and $\Char(\ff) = p$.
We maintain the set-up of \S~\ref{sec:GYT}.

\subsection{Proof of \thmref{thm:Main-1}}\label{sec:PF-1}
%We begin by proving \thmref{thm:Main-1}.
The following result proves part of \thmref{thm:Main-1}.

\begin{thm}\label{thm:Main-1-pf}
  $\Ker(\rho_X)$ is a direct sum of a finite group and an $\L$-divisible group.
  The same holds also for $\Ker(\rho^\tm_X)$.
\end{thm}
\begin{proof}
We first show the assertion for $\Ker(\rho_X)$.
  We look at the commutative diagram
  \begin{equation}\label{eqn:Main-1-pf-0}
    \xymatrix@C1pc{
      0 \to \Ker(\rho_X) \ar[r] \ar[d] & C(X) \ar[d]
      \ar[r]^-{\rho_X} & \pi^{\ab}_1(X) \ar@{->>}[d] \\
      0 \to {\Ker}(\rho_X)^\dagger_\L \ar[r] &  C(X)_\L \ar[r]^-{(\rho_{X})_\L} &
      \pi^{\ab}_1(X)_\L,}
    \end{equation}
    where the vertical arrows are induced by the $\L$-completion. Note that
    the left vertical arrow is the composition $ \Ker(\rho_X) \to  \Ker(\rho_X)_\L \to
     {\Ker}(\rho_X)^\dagger_\L$. It follows from \lemref{lem:Prod-dec} that
 the right vertical arrow is surjective and its kernel is $\pi^{\ab}_1(X)_p$.
In particular, this kernel is a $\Z_p$-module, and hence uniquely $\L$-divisible.

We let $D$ and $T$ denote the kernel and the image of the left vertical arrow,
    respectively,
    so that we have a short exact sequence
    \begin{equation}\label{eqn:Main-1-pf-1}
      0 \to D \to  \Ker(\rho_X) \to T \to 0.
    \end{equation}
    It follows from ~\eqref{eqn:Ker-bound-1} and  ~\eqref{eqn:Ker-bound-2}
    that $T$ is a finite group of exponent lying in $I_\L$. By \lemref{lem:JS-1},
    it remains to show that $D$ is $\L$-divisible. 
    Using \lemref{lem:Ker-div}, it suffices to show that
    the kernel of the right vertical arrow in ~\eqref{eqn:Main-1-pf-0} has no
    $\L$-torsion. But we have seen this in the previous paragraph.
The claim that $\Ker(\rho^\tm_X)$ has the same property has an identical proof.
\end{proof}

\subsection{\thmref{thm:Main-1} for surfaces}\label{sec:Surface**}
Assume in \thmref{thm:Main-1} that $\dim(X) \le 2$. We wish to prove the
following.

\begin{lem}\label{lem:Surface-0}
  One has $\Ker(\rho_X) = F' \oplus D$, where $F'$ is a finite group and $D$ is
  divisible by every integer invertible in $k$. The same holds also for
  $\Ker(\rho^\tm_X)$.
\end{lem}
\begin{proof}
We shall give a proof for $\Ker(\rho_X)$ which also works verbatim for 
$\Ker(\rho^\tm_X)$.  We can assume that $\Char(k) = 0$, else the lemma already 
follows from \thmref{thm:Main-1-pf}.
Suppose that we can prove \lemref{lem:Ker-bound}
by replacing $I_\L$ with $I_{\P}$. 

%We let $\L'$ be the set of all prime numbers.
It is a straightforward checking that 
%Lemmas~\ref{lem:Proj-compl} and
%~\ref{lem:Open-compl} hold (with the same proof) even if we replace $\L$ by
%$\L'$. In particular, 
\lemref{lem:Ker-div} now also holds if we replace $\L$ by
$\P$.
We now look at the diagram ~\eqref{eqn:Main-1-pf-0}. The right vertical arrow is
now an isomorphism and the kernel of the middle vertical arrow is $\P$-divisible. It
follows that the kernel of the left vertical arrow is $\P$-divisible.
Now, we apply \lemref{lem:JS-1} and argue exactly as in the proof of
\thmref{thm:Main-1-pf} to conclude that $\Ker(\rho_X)$ (which is the same as
$\Ker(\rho_{X,0})$) has the desired form.
It remains therefore to prove the improved version of \lemref{lem:Ker-bound}, namely,
there exists an integer $M \gg 0$ such that $|\Ker(\rho^m_X)| \le M$ for all
$m > 0$.

To prove the above assertion, we shall argue as we did in the proof of
\lemref{lem:Ker-bound} and make necessary changes.
For $\ov{X}$, only change we need to make is to use \cite[Thm.~1.8]{JS-Doc}
instead of \cite[Thm.~5.4]{Forre-Crelle}. No other change is required.
For $X$, we only need to make a change in the last line of the proof,
where we need to replace \corref{cor:Boundedness-0} by its improved version.
Hence, we are finally left with showing that
$|H^4(Z, {\Z}/m(3))| \le M$ for all $m > 0$. But this follows from
\propref{prop:Fin-char-0}.
\end{proof}

\begin{remk}\label{remk:JS-higher-dim}
In view of \propref{prop:Fin-char-0}, the proof of \lemref{lem:Surface-0} shows that
its assertion can be extended to any higher dimension if \cite[Thm.~1.8]{JS-Doc}
is valid in that dimension.
\end{remk}

\subsection{\thmref{thm:Main-1} for curves}\label{sec:Curves*}
Assume now in \thmref{thm:Main-1} that $\dim(X) = 1$.
%We wish to prove the following.

\begin{lem}\label{lem:Curve*-0}
  $\Ker(\rho^\tm_X)$ is divisible.
\end{lem}
\begin{proof}
Using the argument of \lemref{lem:Ker-bound}, we can assume that $X$ is 
geometrically connected.
Suppose first that $m \in k^\times$.
\corref{cor:Real-iso-sing-0} says that the map
  $\epsilon^*_X \colon H^3_c(X, {\Z}/m(2)) \to  H^3_{\et,c}(X, {\Z}/m(2))$ is
  injective. Combining this with \propref{prop:Rec-Real},
  we get that the map $\rho^{\tm, m}_X \colon {C^\tm(X)}/m \to {\pi^{\ab,t}_1(X)}/m$
  is injective.
  If $\Char(k) = p$, then it follows from \corref{cor:Rec-tame-unr} (which
  is much easier for curves) and
  \cite[Prop.~3]{Kato-Saito-1} that 
  ${C^\tm(X)}/{p^n} \to {\pi^{\ab,\tm}_1(X)}/{p^n}$
  is injective for all $n \ge 1$.
 Taking the limits, it follows that the map
  \begin{equation}\label{eqn:Curve*-0-0}
   ({\varprojlim} \ \rho^{\tm, m}_X) \colon \ {\varprojlim}_{m > 0} {C^\tm(X)}/m \to
    {\varprojlim}_{m > 0} {\pi^{\ab,t}_1(X)}/m
    \cong \pi^{\ab,t}_1(X)
  \end{equation}
  is injective in any characteristic. 
  In particular, $\Ker(\rho^\tm_X) = \bigcap_{m >0} \ m C^{\tm}(X)$. 

  To finish the proof, we can now apply Lemma~\ref{lem:JS-0} with
  $\M = \{\mbox{all \ primes}\}$,
  $A = C^\tm(X)$ and $B_m = {\pi^{\ab,\tm}_1(X)}/m$ (with $m > 0$).
  Lemma~\ref{lem:Open-compl} and ~\eqref{eqn:Curve*-0-0} ensure that the
  conditions of Lemma~\ref{lem:JS-0} are met.
  \end{proof}

We finally have:

\vskip .2cm

{\sl Proof of \thmref{thm:Main-1}:}
Combine \propref{prop:REC-T}, \thmref{thm:Main-1-pf} as well as
Lemmas~\ref{lem:Surface-0} and ~\ref{lem:Curve*-0}.
$\hfill\qed$

\subsection{Proof of \thmref{thm:Main-5}}\label{sec:Pf-5}
We shall now prove \thmref{thm:Main-5}. We restate it here
for reader's convenience. Recall that
$E^\tm(X)$ and $J^\tm(X)$ are the kernels of the canonical surjections
$C^\tm(X) \surj C(\ov{X})$ (cf. \propref{prop:ML-main-3}) and $\pi^{\ab, \tm}_1(X) \surj \pi^{\ab}_1(\ov{X})$,
respectively. 

\begin{thm}\label{thm:Main-5-1}
  The map $\rho^\tm_X \colon E^\tm(X)_\tor \to J^\tm(X)$ is surjective.
\end{thm}
\begin{proof}
  The proof of the theorem is easily reduced to the case when $\dim(X) = 1$ by
  \propref{prop:ML-main-3}, \corref{cor:REC-T-PF}
  and \thmref{thm:Tame-nr-main}.
We thus assume that $\dim(X) =1$. 
  We let $K = k(X)$ and $K_x$ the completion of $K$ at a closed point $x$.

We recall at this point that the local reciprocity 
    $\rho_{K_x} \colon K^M_2(K_x) \to G_{K_x}$ maps $U'_iK^M_2(K_x)$ to the ramification
    subgroup $G^{(i)}_{K_x}$ for every $i \ge 0$ by \propref{prop:Kato-REC}.
Using Definition~\ref{defn:Idel-X-0} and \lemref{lem:Approximation},
  we get a commutative diagram
  \begin{equation}\label{eqn:Main-5-1-0}
    \xymatrix@C1pc{
      {\underset{x \in Z}\bigoplus} \frac{U'_0K^M_2(K_x)}{U'_1K^M_2(K_x)} \ar@{->>}[r]
      \ar[d] & E^\tm(X) \ar[d]^-{\rho^\tm_X} \\
      {\underset{x \in Z}\bigoplus} \frac{G^{(0)}_{K_x}}{G^{(1)}_{K_x}} \ar@{->>}[r]
      & J^\tm(X),}
    \end{equation}
    where $Z = \ov{X} \setminus X$ and
    the left vertical arrow is the sum of local reciprocity maps for $K_x$
    as $x$ runs through $Z$.
    The horizontal arrows are surjective by \propref{prop:ML-main-3} and
    \thmref{thm:Tame-nr-main}. Since $\frac{U'_0K^M_2(K_x)}{U'_1K^M_2(K_x)} \cong
    K^M_2(k(x))$, it suffices to
  show that $\rho_{K_x}$ induces a surjection
  $K^M_2(k(x))_\tor \surj {G^{(0)}_{K_x}}/{G^{(1)}_{K_x}}$ for every $x \in \ov{X}_{(0)}$.
    
To prove the above, we first note that the induced map 
   $\rho^\vee_{K_x} \colon ({G^{(0)}_{K_x}}/{G^{(1)}_{K_x}})^\vee \cong
   \frac{\Fil_1 H^1(K_x)}{\Fil_0 H^1(K_x)} \to
   \Hom_{\Tab}( \frac{U'_0K^M_2(K_x)}{U'_1K^M_2(K_x)}, {\Q}/{\Z})$
   is injective by ~\eqref{eqn:fil_n-dual} and \propref{prop:Kato-REC}. Equivalently,
   the map $\frac{U'_0K^M_2(K_x)}{U'_1K^M_2(K_x)}
\to {G^{(0)}_{K_x}}/{G^{(1)}_{K_x}}$ has dense image by
\cite[Lemma~7.11]{Gupta-Krishna-REC}.
On the other hand, ${G^{(0)}_{K_x}}/{G^{(1)}_{K_x}}$ is finite by \lemref{lem:Local-case}
and $\frac{U'_0K^M_2(K_x)}{U'_1K^M_2(K_x)} \cong K^M_2(k(x))$ is a
direct sum of divisible and finite groups by \cite{Tate-Kyoto}. This forces
  $K^M_2(k(x))_\tor \to {G^{(0)}_{K_x}}/{G^{(1)}_{K_x}}$ to be surjective.
 \end{proof}

\begin{cor}\label{cor:Main-5-2}
  The maps $C^\tm(X)_\tor \to C(\ov{X})_\tor$ and $\Ker(\rho^\tm_X)_\tor \to
  \Ker(\rho_{\ov{X}})_\tor$ are surjective.
\end{cor}
\begin{proof}
  We have seen in the proof of \thmref{thm:Main-5-1} that $E^\tm(X)$ is a quotient of
  a direct sum of divisible and finite groups. In particular,
  $E^\tm(X) \otimes {\Q}/{\Z}
  = 0$. This implies that $C^\tm(X)_\tor \to C(\ov{X})_\tor$ is surjective.
  It follows from \thmref{thm:Main-5-1} that
  \begin{equation}\label{eqn:Main-5-2-0}
    0 \to \Ker(E^\tm(X) \to J^\tm(X)) \to  \Ker(\rho^\tm_X) \to \Ker(\rho_{\ov{X}})
    \to 0
  \end{equation}
  is exact. Since $E^\tm(X)_\tor \surj J^\tm(X)$ and
  $E^\tm(X) \otimes {\Q}/{\Z} = 0$, it follows that
  $(\Ker(E^\tm(X) \to J^\tm(X))) \otimes {\Q}/{\Z} = 0$.
  We conclude that $\Ker(\rho^\tm_X)_\tor \to
  \Ker(\rho_{\ov{X}})_\tor$ is surjective.
\end{proof}

\subsection{Finiteness of the torsion subgroups of fundamental groups}
\label{sec:FIN-EXP-FIN}
We shall now prove a refinement 
of the results of \S~\ref{sec:L-compl-tor} and \S~\ref{sec:Tor-L-compl}. We remark
though that this refinement already follows from
\thmref{thm:J-fin} when $X$ is geometrically connected.
At the same time, we also remark that \thmref{thm:Main-8*} was not earlier known
even for projective schemes over $k$ in any dimension without the assumption of
geometrically connectedness.

The key lemma for proving \thmref{thm:Main-8*} is the following.

\begin{lem}\label{lem:p-tor-fin}
  The groups $\pi^{\ab}_1(X)\{\ell\}$ and $\pi^{\ab, \tm}_1(X)\{\ell\}$ 
  are finite for every $\ell \in \P$.
\end{lem}
\begin{proof}
  By \propref{prop:Pi-decom} and \thmref{thm:J-fin}, it suffices to prove the
  assertion of the lemma for $\pi^{\ab}_1(\ov{X})\{\ell\}$. We can therefore assume that
  $X$ is projective over $k$. By \lemref{lem:Proj-compl}, it suffices
  to show that $\ell^n$-torsion subgroup of $\pi^{\ab}_1(X)$ is finite for every
  $n \ge 1$. By \lemref{lem:Prod-dec}, we can replace $\pi^{\ab}_1(X)$ by
  $\pi^{\ab}_1({X})_\ell$.

  We now compute

      \[
        \begin{array}{lll}
          {}_{\ell^n} \pi^{\ab}_1({X})_\ell & = &
          \Hom({\Z}/{\ell^n}, \pi^{\ab}_1({X})_\ell) \\ 
                   & = &
\Hom({\Z}/{\ell^n}, \Hom(H^1_\et(X, {\Q_\ell}/{\Z_\ell}), {\Q_\ell}/{\Z_\ell})) \\
 & = &  \Hom({H^1_\et(X, {\Q_\ell}/{\Z_\ell})}/{\ell^n},
                                                {\Q_\ell}/{\Z_\ell}),
        \end{array}
      \]
      where all homomorphisms are in the category of abelian groups.
      It suffices therefore to show that ${H^1_\et(X, {\Q_\ell}/{\Z_\ell})}/{\ell^n}$
      is finite.

The exact sequence
      \[
        0 \to {\Z}/{\ell^n} \to {\Q_\ell}/{\Z_\ell} \xrightarrow{\ell^n}
        {\Q_\ell}/{\Z_\ell} \to 0
      \]
      implies that ${H^1_\et(X, {\Q_\ell}/{\Z_\ell})}/{\ell^n} \inj
      H^2_\et(X, {\Z}/{\ell^n})$.
      We thus need to show that $H^2_\et(X, {\Z}/{\ell^n})$ is finite.
      But this is a special case of the general fact that 
      $H^i_\et(X, {\Z}/m(j))$ is finite for any integers $i, j \ge 0$ and $m \ge 1$
      such that $m \in k^\times$.
      The latter result is well known. For instance, it is
      an easy consequence of the Hochschild-Serre spectral
      sequence using that $cd(k) = 2$ (cf. \cite[Rem.~3.5(c)]{Jannsen-MA}).
      \end{proof}

We can now prove the following. This also completes the proof of
      \thmref{thm:Main-3}.

\begin{thm}\label{thm:Main-8*}
Let $F$ be any of the following groups:
  \[
    (1) \ \pi^{\ab, \tm}_1({X})_\tor, \
    (2) \ {\rm Image}(\rho^\tm_{X,0}), \ (3) \ ({\rm Image}(\rho^\tm_{X}))_\tor,
  \]
  \[
\ (4) \ {\rm Coker}(\rho^{\tm}_{X,0})_\tor, \
    \mbox{and} \ (5) \ {\rm Coker}(\rho^\tm_X)_\tor.
  \]
  Then $|F| < \infty$ if $\Char(k) = 0$ and $F$ has finite exponent 
with $|F\{p'\}| < \infty$ if $\Char(k) = p$.
\end{thm}
\begin{proof}
  All cases except possibly (4) and (5) follow by combining
  \lemref{lem:p-tor-fin} with Lemmas~\ref{lem:Coker-0}
  and ~\ref{lem:Open-compl}. We next note that (5) follows from (4)
  by ~\eqref{eqn:Coker-1-1} and the latter follows from the case (1) because
  we showed in the proof of \corref{cor:Coker-1} that
  $(\pi^{\ab, \tm}_1({X})_0)_\tor \surj {\rm Coker}(\rho^{\tm}_{X,0})_\tor$.
\end{proof}

\subsection{Proof of \thmref{thm:Main-9}}\label{sec:M-9*}
In this subsection, we shall prove \thmref{thm:Main-9} which we restate below for
reader's convenience. We keep the notations of \S~\ref{sec:GYT}.
In particular, we let $X$ be as in \thmref{thm:Main-9} with smooth compactification $\ov{X}$.
For $f \colon Y \to \Spec(k)$ in $\Sch_k$,
we let $H^i(Y, \Z(j)) =
\Hom_{\dm(k, \Z)}(\Z_{\rm tr}(Y), \Z(j)[i])$ and
$H^i_c(Y, \Z(j)) =
\Hom_{\dm(k, \Z)}(f_* f^! \Z_k, \Z(j)[i])$, where $\Z(1) =
{\Z_{\rm tr}(\P^1_k)}/{\Z_{\rm tr}(\Spec(k))}$. Note that the premotivic category
$Y \mapsto \dm(Y, \Z)$
is equipped with the six functors by \cite{CD-Springer}.

\begin{thm}\label{thm:Main-9**}
  We have the following.
  \begin{enumerate}
    \item
  There is a decomposition $H^{2d+2}_c(X, {\Z}(d+2)) \cong F \bigoplus D$, where
  $F$ is a finite group of order lying in $I_\P$ and $D$ is $\P$-divisible.
  \item
  For $n \ge 3$, the group $H^{2d+n}_c(X, {\Z}(d+n))$ is $\P$-divisible.
  \item
  $\CH^{d+n}(X, n)$ is divisible if $n \ge 3$ and $X$ is projective over $k$.
\end{enumerate}
\end{thm}
\begin{proof}
Using the inclusion
  ${H^{2d+2}_c(X, {\Z}(d+2))}/m \inj H^{2d+2}_c(X, {\Z}/m(d+2))$ and \corref{cor:DIV-ab},
we see that (1) will follow if show that there exists an integer
  $M \gg 0$ such that $|{H^{2d+2}_c(X, {\Z}/m(d+2))}| \le M$ for all $m \in I_\P$.
But this follows from \corref{cor:Boundedness-0} and \propref{prop:Fin-char-0}.

To prove (2), note that ${H^{2d+n}_c(X, {\Z}(d+n))}/m \inj H^{2d+n}_c(X, {\Z}/m(d+n))$.
It is therefore enough to show that $ H^{2d+n}_c(X, {\Z}/m(d+n)) = 0$ if
$m \in k^\times$. To check the latter statement, 
we let $Z = \ov{X} \setminus X$. Noting that $\dm(k, {\Z}/m) \xrightarrow{\simeq}
\dm_\cdh(k, {\Z}/m)$, \lemref{lem:Ex-seq} says that there is an
exact sequence
\[
  {H^{2d+n-1}(Z, {\Z}/m(d+n))} \to {H^{2d+n}_c(X, {\Z}/m(d+n))} \to
  {H^{2d+n}(\ov{X}, {\Z}/m(d+n))},
\]
By \lemref{lem:Real-iso-sing}, the realization map
$H^{2d+n-1}(Z, {\Z}/m(d+n)) \to H^{2d+n-1}_\et(Z, {\Z}/m(d+n))$ is an
isomorphism. But the latter group is zero because $cd_k(Z) \le 2d$.
This implies that the term on the left in the above exact sequence is zero.
It suffices therefore to show that
${H^{2d+n}(\ov{X}, {\Z}/m(d+n))} = 0$. But this is a direct consequence of
\lemref{lem:Real-iso-sing} and the fact that $cd_k(\ov{X}) \le 2d+2$.

We now prove (3). By \cite[Thm.~1.1]{Akhtar}, we have a canonical isomorphism
$\CH^{d+n}(X, n) \cong H^d_\zar(X, \sK^M_{d+n,X})$. On the other hand,
the Gersten resolution for Milnor $K$-theory
%\cite{Kerz-Invent}
implies that there is a surjection ${\underset{x \in X_{(0)}}\bigoplus}
K^M_n(k(x)) \surj H^d_\zar(X, \sK^M_{d+n, X})$. We are now done because
each $K^M_n(k(x))$ is a divisible group for $n \ge 3$ by
\cite[Thm.~IX.4.11]{Fesenko-Vostokov}. 
\end{proof}

\section{The case of good reduction}\label{sec:GRed}
The goal of this section is to prove \thmref{thm:Main-4}. We shall work
under the following setup throughout this section.
We let $k$ be a local field with the ring of integers $\sO_k$ and the residue field
$\ff$. We let $\ov{X}$ be a smooth projective scheme over $k$ of pure dimension $d$
and let $j \colon X \inj \ov{X}$ be an open immersion with dense image.
We let $K = k(X)$. We fix an integer $m$ invertible in $k$.
Recall that $\ov{X}$ is said to admit a good reduction if 
there exists a smooth projective morphism $f \colon \ov{\sX} \to \Spec(\sO_k)$
whose generic fiber is $\ov{X}$. We let $\ov{X}_s$ denote the special fiber of $f$.
We first do some preparation.

\subsection{A vanishing lemma for surfaces}\label{sec:dim<3}
We let $\sH^j({\Z}/m(n))$ be the sheaf on the big Zariski site of $\Sm_k$
  associated to the presheaf $U \mapsto H^j(U, {\Z}/m(n))$.
  Similarly, we let $\sH^j_\et({\Z}/m(n))$ be the sheaf on the big Zariski site of
  $\Sm_k$ associated to the presheaf $U \mapsto H^j_\et(U, {\Z}/m(n))$.

\begin{lem}\label{lem:KH-Surface}
   If  $d = 2$ and $\ov{X}$ admits a good reduction, then
   $H^0_\zar(\ov{X}, \sH^4_\et({\Z}/m)(3)) = 0$.
\end{lem}
\begin{proof}
  Using the exact sequence (cf. \cite[Cor.~2.2.2]{CHK})
\begin{equation}\label{eqn:GR-Surface-2}
  0 \to H^0_\zar(\ov{X}, \sH^4_\et({\Z}/m)(3)) \to H^4_\et(K, {\Z}/m(3))
  \xrightarrow{\partial_{\ov{X}}} {\underset{x \in \ov{X}^{(1)}}\bigoplus}
    H^3_\et(k(x), {\Z}/m(2)),
    \end{equation}
the lemma is equivalent to the assertion that $\Ker(\partial_{\ov{X}}) = 0$.

We now recall the Kato complex  $K(\ov{X}, r,s)$
from \cite[\S~1]{Kato-Crelle} (this exists even if $d \neq 2$):
\begin{equation}\label{eqn:GR-Surface-3}
  H^{r+d}_\et(K, {\Z}/m(s+d))  \xrightarrow{\partial_{\ov{X}}} \cdots \to
  {\underset{x \in \ov{X}^{(d-1)}}\bigoplus} H^{r+1}_\et(k(x), {\Z}/m(s+1)) \to
  {\underset{x \in \ov{X}^{(d)}}\bigoplus} H^r_\et(k(x), {\Z}/m(s)),
\end{equation}
which is defined when $r, s > 0$. This complex is defined for $\ov{X}_s$ for
every $r > 0$ and $s \ge 0$. We denote it by $K(\ov{X}_s, r,s)$.
The differentials of both complexes are the residue maps.
The Kato homology $H^{r,s}_i(\ov{X}, {\Z}/m)$ is the $i$-th homology of
$K(\ov{X}, r,s)$. We let $H^{r,s}_i(\ov{X}_s, {\Z}/m)$ denote the $i$-th homology of
$K(\ov{X}_s, r,s)$. 
Our problem now is reduced to showing that $H^{2,1}_2(\ov{X}, {\Z}/m) = 0$.
For this, it suffices to show that $H^{2,1}_2(\ov{X}, {\Z}/{\ell^n}) = 0$
for every prime $\ell \neq \Char(k)$ and $n \ge 1$.
For proving this, we let $H^{r,s}_i(\ov{X}, {\Q_\ell}/{\Z_\ell})
= \varinjlim_n H^{r,s}_i(\ov{X}, {\Z}/{\ell^n})$ (cf. \cite[Defn.~1.2]{JS-Doc}).

Since $H^{2,1}_3(\ov{X}, {\Q_\ell}/{\Z_\ell}) = 0$ as $d \le 2$, the exact sequence
  (cf. \cite[Lem.~7.3]{JS-Doc}) 
  \begin{equation}\label{eqn:GR-Surface-4}
    0 \to {H^{2,1}_3(\ov{X}, {\Q_\ell}/{\Z_\ell})}/{\ell^n} \to
      H^{2,1}_2(\ov{X}, {\Z}/{\ell^n}) \to {}_{\ell^n}
      H^{2,1}_2(\ov{X}, {\Q_\ell}/{\Z_\ell}) \to 0
      \end{equation}
      reduces us to showing that $H^{2,1}_2(\ov{X}, {\Q_\ell}/{\Z_\ell}) =0$.
To that end, we note that the residue map
$H^{2,1}_2(\ov{X}, {\Q_\ell}/{\Z_\ell}) \to  H^{1,0}_2(\ov{X}_s, {\Q_\ell}/{\Z_\ell})$ is 
bijective (even if $\ell = \Char(\ff)$)
by \cite[Thm.~1.6]{JS-Doc}. But the latter group is zero by
\cite[Thm.~1.4]{JS-Doc} because $\ov{X}_s$ is smooth over $\ff$.
\end{proof}

\subsection{A surjectivity of {\'e}tale realization}\label{sec:Surj-et}
We next prove the following general result.

\begin{lem}\label{lem:GR-Surface}
Assume that $\ov{X}$ admits a good reduction.  Then the {\'e}tale realization map
  \[
    \epsilon^*_X \colon H^{2d}(\ov{X}, {\Z}/m(d+1)) \to  H^{2d}_\et(\ov{X}, {\Z}/m(d+1))
  \]
is surjective.
\end{lem}
\begin{proof}
We shall prove the lemma by induction on $d$. The case $d \le 1$ follows directly
  from \corref{cor:Real-iso-sing-0}. To prove $d = 2$ case, we consider the commutative
  diagram of the Zariski descent spectral sequences
  \begin{equation}\label{eqn:ZDSS}
    \xymatrix@C1pc{
      E^{i,j} = H^i_\zar(\ov{X}, \sH^j({\Z}/m)(n)) \ar[d]_-{\epsilon^*_{\ov{X}}}
      & \Rightarrow & H^{i+j}(\ov{X}, {\Z}/m(n)) \ar[d]^-{\epsilon^*_{\ov{X}}} \\
      E^{i,j}_\et = H^i_\zar(\ov{X}, \sH^j_\et({\Z}/m)(n))  & \Rightarrow &
      H^{i+j}_\et(\ov{X}, {\Z}/m(n)).}
    \end{equation}

The above diagram of spectral sequences gives rise to the following
    two commutative diagrams of exact sequences.
    \begin{equation}\label{eqn:GR-Surface-0}
      \xymatrix@C1pc{
        0 \ar[r] & F^1 \ar[r] \ar[d] & H^4(\ov{X}, {\Z}/m(3)) \ar[r] \ar[d] &
        H^0_\zar(\ov{X}, \sH^4({\Z}/m)(3)) \ar[d] \\
        0 \ar[r] & F^1_\et \ar[r] & H^4_\et(\ov{X}, {\Z}/m(3)) \ar[r] &
        H^0_\zar(\ov{X}, \sH^4_\et({\Z}/m)(3)).}
    \end{equation}

\begin{equation}\label{eqn:GR-Surface-1}
      \xymatrix@C.6pc{
H^0_\zar(\ov{X},  \sH^3({\Z}/m)(3)) \ar[r] \ar[d] & H^2_\zar(\ov{X},  \sH^2({\Z}/m)(3)) 
\ar[r] \ar[d] &  F^1 \ar[r] \ar[d] & H^1_\zar(\ov{X},  \sH^3({\Z}/m)(3)) 
\ar[r] \ar[d] & 0 \\
H^0_\zar(\ov{X},  \sH^3_\et({\Z}/m)(3)) \ar[r] & H^2_\zar(\ov{X},  \sH^2_\et({\Z}/m)(3)) 
\ar[r] &  F^1_\et \ar[r] & H^1_\zar(\ov{X},  \sH^3_\et({\Z}/m)(3)) 
\ar[r] & 0.}
\end{equation}
In the above two diagrams, the vertical arrows are the {\'e}tale realization maps.

It follows from \lemref{lem:Real-iso} that all vertical arrows (except possibly the third vertical arrow from the left) in ~\eqref{eqn:GR-Surface-1} are isomorphisms. It follows that the left vertical arrow in 
~\eqref{eqn:GR-Surface-0} is an isomorphism.
We are therefore left with the task of showing that
$H^0_\zar(\ov{X}, \sH^4_\et({\Z}/m)(3)) = 0$. But this follows from
\lemref{lem:KH-Surface}. We have thus proven the lemma when $d = 2$.

We now assume that $d > 2$ and that the assertion of the lemma
holds in dimensions smaller than $d$.
By \cite[Lem.~9.1, Thm.~9.6]{Ghosh-Krishna-Bertini}
(see the proof of Thm.~9.6 of op. cit.)
% see also \cite[Cor.~03]{JS-JAG})
we can find a very ample smooth
divisor $\ov{Y} \inj \ov{X}$ such that $\ov{Y}$ also admits a good reduction.
Let $\iota \colon \ov{Y} \inj \ov{X}$ be the inclusion.
By \corref{cor:Gysin-Nis-etale-1}, we have a commutative diagram
\begin{equation}\label{eqn:GR-Surface-5}
  \xymatrix@C1pc{
    H^{2d-2}(\ov{Y}, {\Z}/m(d)) \ar[r]^-{\iota_*} \ar[d]_-{\epsilon^*_{\ov{Y}}} &
      H^{2d}(\ov{X}, {\Z}/m(d+1)) \ar[d]^-{\epsilon^*_{\ov{X}}} \\
     H^{2d-2}_\et(\ov{Y}, {\Z}/m(d)) \ar[r]^-{\iota_*} &
      H^{2d}_\et(\ov{X}, {\Z}/m(d+1)).}   
  \end{equation}
  Since $\ov{X}$ is of pure dimension $d$ and $\ov{Y}$ is very ample,
  it follows that $\ov{Y}$ is of pure dimension $d-1$.
  In particular, the left vertical arrow in ~\eqref{eqn:GR-Surface-5} is
  surjective by induction. To conclude the proof, it remains to show that
  the bottom horizontal arrow in ~\eqref{eqn:GR-Surface-5} is surjective.

To prove the above, we recall that the map $\iota_*$ is the composition
  \[
    H^{2d-2}_\et(\ov{Y}, {\Z}/m(d)) \xrightarrow{\cong}
    H^{2d}_{\et, \ov{Y}}(\ov{X}, {\Z}/m(d+1)) \to H^{2d}_\et(\ov{X}, {\Z}/m(d+1)),
  \]
  where the first arrow is Gabber's purity isomorphism and the second is the forget
  support map. Hence, we are reduced to showing that the second arrow
  is surjective. To this end, we let $U = \ov{X} \setminus \ov{Y}$ and
  look at the exact localization sequence 
  \[
    H^{2d}_{\et, \ov{Y}}(\ov{X}, {\Z}/m(d+1)) \to H^{2d}_\et(\ov{X}, {\Z}/m(d+1)) \to
    H^{2d}_\et(U, {\Z}/m(d+1)).
  \]
  Since $\ov{Y}$ is a very ample divisor of the projective $k$-scheme $\ov{X}$, it
  follows that $U$ is an affine $k$-scheme of dimension $d$.
  In particular, $cd(U) \le d + 2$. Since $d \ge 3$, it follows that
  $H^{2d}_\et(U, {\Z}/m(d+1)) = 0$. This concludes the proof.
\end{proof}

\subsection{Proof of \thmref{thm:Main-4}}\label{sec:Good-red-main}
We now prove \thmref{thm:Main-4}. We state it here for reader's convenience.
A special case of this, when $\Char(k) = 0$ and $\ov{X}$ is a geometrically connected
rational surface over $k$, was earlier shown by Yamazaki \cite[Thm.~6.5]{Yamazaki}.

\begin{thm}\label{thm:Main-4-1}
  Let $k$ be a local field and $j \colon X \inj \ov{X}$ an open
  immersion in $\Sm_k$ with dense image such that $\ov{X}$ is 
  projective over $k$. Assume that ${X}$ is birational to a smooth projective $k$-scheme
  which admits a good reduction. Then
  \[
    \rho_X \colon {C(X)}/m \to {\pi^{\ab}_1(X)}/m
  \]
  is an isomorphism for every integer $m \in k^\times$.
  In particular, $\Ker(\rho_X)$ is divisible away from $\Char(k)$.
  The same also holds for $\rho^\tm_X$.
\end{thm}
\begin{proof}
By \corref{cor:Pi-decom-0}, \lemref{lem:Idele-Tame-prime-to-p} and
 \propref{prop:REC-T}, the assertions of the
theorem for $\rho_X$ and $\rho^\tm_X$ are equivalent.
By \propref{prop:Tame-MCCS}, \thmref{thm:Saito-D} and \propref{prop:Rec-Real},
the first part of the theorem is equivalent to the statement that the
{\'e}tale realization map
\begin{equation}\label{eqn:Et*-0}
  \epsilon^*_X \colon H^{2d+1}_c(X, {\Z}/m(d+1)) \to H^{2d+1}_{\et,c}(X, {\Z}/m(d+1))
\end{equation}
is an isomorphism.

By our assumption, there are dense open immersions
      $\ov{X}' \hookleftarrow X'
      \hookrightarrow \ov{X}$, where $\ov{X}'$ is a smooth connected
      projective $k$-scheme which admits a good reduction.
      We let $Z' = \ov{X} \setminus X'$ and $W' = \ov{X}' \setminus X'$ with their
      reduced subscheme structures.
      In the argument below, we shall write $H^{i}_c(Y, {\Z}/m(d+1))$ simply
as $H^i_c(Y)$ for $Y \in \Sch_k$.
We shall follow the same convention for the {\'e}tale cohomology
(with and without compact support).

We now consider the commutative diagram with exact rows (cf. \lemref{lem:Ex-seq} and
~\eqref{eqn:Ker-bound-0})
\begin{equation}\label{eqn:Rational-var-0}
  \xymatrix@C.8pc{
    H^{2d-1}(W') \ar[r] \ar[d]_-{\epsilon^*_{W'}} &
    H^{2d}_c(X') \ar[r] \ar[d]^-{\epsilon^*_{X'}} &
    H^{2d}(\ov{X}') \ar[r] \ar[d]^-{\epsilon^*_{\ov{X}'}} &
  H^{2d}(W') \ar[r] \ar[d]^-{\epsilon^*_{W'}} & H^{2d+1}_c(X')
\ar[r] \ar[d]^-{\epsilon^*_{X'}} & H^{2d+1}(\ov{X}')
\ar[d]^-{\epsilon^*_{\ov{X}'}} \ar[r] & 0 \\
H^{2d-1}_\et(W') \ar[r]  & H^{2d}_{\et,c}(X') \ar[r] &
H^{2d}_\et(\ov{X}') \ar[r] & H^{2d}_\et(W') \ar[r] &
H^{2d+1}_{\et, c}(X') \ar[r] & H^{2d+1}_\et(\ov{X}') \ar[r] & 0.}
\end{equation}

The arrows $\epsilon^*_{W'}$ are isomorphisms by \corref{cor:Real-iso-sing-0}.
Since $\ov{X}'$ admits a good reduction, it follows from \lemref{lem:GR-Surface} that
the third vertical arrow from the left in ~\eqref{eqn:Rational-var-0} is surjective.
For the same reason, the vertical arrow on the extreme right is an isomorphism
by \cite[Thm.~4.1]{JS-JAG} in view of \propref{prop:Rec-Real}.
An easy diagram chase shows that the second vertical
arrow from the left is surjective and the second vertical arrow from the right is
an isomorphism. If we now replace $\ov{X}'$ (resp. $W'$) by $\ov{X}$
(resp. $Z'$) in the above diagram and repeat the argument, we conclude that the map
\begin{equation}\label{eqn:Rational-var-1}
  \epsilon^*_{\ov{X}} \colon H^{i}(\ov{X}) \to H^{i}_\et(\ov{X})
\end{equation}
is surjective for $i = 2d$ and isomorphism for $i = 2d+1$.

We let $Z = \ov{X} \setminus X$ with the reduced subscheme structure
and look at the commutative diagram of exact sequences
\begin{equation}\label{eqn:Et*-1}
  \xymatrix@C.8pc{
  H^{2d}(\ov{X}) \ar[r] \ar[d]_-{\epsilon^*_{\ov{X}}} &
  H^{2d}(Z) \ar[r] \ar[d]^-{\epsilon^*_Z} & H^{2d+1}_c(X)
\ar[r] \ar[d]^-{\epsilon^*_X} & H^{2d+1}(\ov{X})
\ar[d]^-{\epsilon^*_{\ov{X}}} \ar[r] & 0 \\
H^{2d}_\et(\ov{X}) \ar[r] & H^{2d}_\et(Z) \ar[r] &
H^{2d+1}_{\et, c}(X)
 \ar[r] & H^{2d+1}_\et(\ov{X}) \ar[r] & 0.}
\end{equation}
We showed above that $\epsilon^*_{\ov{X}}$ is surjective on the left end and
an isomorphism on the right end. The arrow $\epsilon^*_Z$ is an isomorphism by
\corref{cor:Real-iso-sing-0}. An easy diagram chase shows that $\epsilon^*_X$ is
an isomorphism.

To show that $\Ker(\rho_X)$ is divisible away from $\Char(k)$, we
simply repeat the argument for proving \thmref{thm:Main-1-pf} with
$\L$ replaced by $\P$ and note that $T = 0$ in the present case by
what we have shown above. The same argument also applies to $\Ker(\rho^\tm_X)$.
This concludes the proof.
\end{proof}

\begin{cor}\label{cor:GD**}
  Let the assumptions be as in \thmref{thm:Main-4-1} and assume furthermore
  that $X$ is geometrically connected. Then there is an
  exact sequence
  \[
    0 \to \Ker(\rho^\tm_{X,0}) \to C^\tm(X)_0 \xrightarrow{\rho^\tm_{X,0}}
    \pi^{\ab, \tm}_1(X)_0 \to F \to 0,
  \]
  where $\Ker(\rho^\tm_{X,0})$ is divisible away from $\Char(k)$ and $F$ is
  zero (resp. a finite $p$-group) if $\Char(k) = 0$ (resp. $\Char(k) =
  \Char({\ff}) = p $).
\end{cor}
\begin{proof} Recall that we have an exact sequence $0 \to k^{\times}\to G_k \to \wh{\Z}/\Z\to 0$. Using this, the corollary is a direct consequence of
  Theorems~\ref{thm:Main-2}, ~\ref{thm:Main-3} and ~\ref{thm:Main-4-1} if we know that the norm map $C^t(X) \to k^{\times}$ is surjective.  Using the technique of the proof of \thmref{thm:Main-4-1} and the surjection $C^t(X) \surj C(\ov{X})$, we are reduced to show that the norm map $C(\ov{X}) \to k^{\times}$ is surjective when $\ov{X}$ admits a good reduction. But this follows from \cite[Lem.~8]{Kato-Saito-1}. 
  \end{proof}

\vskip .4cm

 \noindent\emph{Acknowledgements.}
R.G. was supported by 
SFB 1085 \emph{Higher Invariants} (Universit\"at Regensburg) and J.R.
was supported by IISc, Bangalore for parts of this project.
The authors thank the referee for providing helpful comments and suggestions to
improve the exposition of the paper.


\begin{thebibliography}{99}
\bibitem{Abbes-Saito} A. Abbes, T. Saito, {\sl Ramification of local fields 
with imperfect residue fields\/}, Amer. J. Math., {\bf 124}, (2002), no. 5, 
879--920. \

%\bibitem{Abbes-Saito-1} A. Abbes, T. Saito, {\sl Analyse micro-locale 
%$\ell$-adique 
%en caract{\'e}ristique $p > 0$: le cas d'un trait\/}, Publ. Res. Inst. Math. 
%Sci., {\bf 45}, (2009), no. 1, 25--74. \

\bibitem{Akhtar} R. Akhtar, {\sl Zero-cycles on varieties over finite fields\/},
Comm. Algebra, {\bf 32}, (2004), 274--294. \


\bibitem{AK} A. Altman, S. Kleiman, {\sl Bertini theorems for hypersurface
sections containing a subscheme\/}, Comm. Alg., {\bf 7}, (1979), no. 8,
775--790. \



\bibitem{SGA4-Tome2} M. Artin, A. Grothendieck, J.-L. Verdier,
  {\sl Theorie des topos et cohomologie \'etale des schemas\/}, SGA~4, Tome 2,
  Lecture Notes in Math., Springer, {\bf 270}, (1972). \




\bibitem{SGA4} M. Artin, A. Grothendieck, J.-L. Verdier,
  {\sl Theorie des topos et cohomologie \'etale des schemas\/}, SGA~4, Tome 3,
  Lecture Notes in Math., Springer, {\bf 305}, (1973). \

\bibitem{Ayoub-1} J. Ayoub, {\sl Les six op{\'e}rations de Grothendieck et le
    formalisme des cycles {\'e}vanescents dans le monde motivique. I\/},
  Ast{\'e}risque, {\bf 314},  477 pp., (2007). \

\bibitem{Ayoub} J. Ayoub, {\sl Les six op{\'e}rations de Grothendieck et le
    formalisme des cycles {\'e}vanescents dans le monde motivique. II\/},
  Ast{\'e}risque, {\bf 315},  364 pp., (2007). \

\bibitem{Bass-Tate} H. Bass, J. Tate, {\sl Milnor $K$-theory of a global field\/},
  in `Algebraic $K$theory II', Lecture Notes in Math.,  Springer,
  {\bf 342}, (1973), p.~ 349--428.\

\bibitem{Binda-Krishna-JEP} F. Binda, A. Krishna,
  {\sl Zero-cycle groups on algebraic varieties\/},
Journal de l'{\'E}cole polytechnique, {\bf 9}, (2022), 281--325. \
  
%
%\bibitem{BKS} F. Binda, A. Krishna, S. Saito, {\sl
%Bloch's formula for 0-cycles with modulus and the higher dimensional
%class field theory\/}, J. Algebraic Geom., {\bf 32}, (2023), 323--384. \
%

\bibitem{Bloch-cft} S. Bloch, {\sl Algebraic $K$-theory and class field theory
    for arithmetic surfaces\/}, Ann. Math., {\bf 114}, (1981), 229--265. \ 

%
%\bibitem{Bloch-Adv} S. Bloch, {\sl Algebraic Cycles and Higher K-theory\/},
%Adv. Math., {\bf 61}, (1986), 267--304. \



\bibitem{Bloch-JAG} S. Bloch, {\sl The moving lemma for higher Chow groups\/,}
  J. Alg. Geom., {\bf 3}, no. 3, (1994), 537--568. \

%
%\bibitem{Bloch-Ogus} S. Bloch, A. Ogus, {\sl Gersten’s conjecture and the homology of
%    schemes\/}, Ann. Scient. \'Ec. Norm. Sup. {\bf 4} (1974), no. 7, 181--202.\

\bibitem{CD-Doc} D.-C. Cisinski, F. D\'eglise, {\sl Integral mixed motives in equal
    characteristic\/}, Doc. Math. Extra volume:
Alexander S. Merkurjev’s sixtieth birthday, (2015), 145--194.\


\bibitem{CD-Comp} D.-C. Cisinski, F. D\'eglise, {\sl \'Etale motives\/},
  Compos. Math., {\bf 152}, (2016), 556--666.\

\bibitem{CD-Springer} D.-C. Cisinski, F. D\'eglise, {\sl Triangulated categories of
    mixed motives\/}, Springer Monographs in Mathematics, Springer, (2019).\


\bibitem{CHK} J.-L. Colliot-Th\'el\`ene, R. Hoobler, B. Kahn,
  {\sl The Bloch-Ogus-Gabber theorem\/}, Algebraic $K$-theory (Toronto, ON, 1996), 
  Fields Inst. Commun., {\bf 16}, 31--94, Amer. Math. Soc., Providence, RI, 
  1997. \ 


\bibitem{CSS} J.-L. Colliot-Th\'el\`ene, J.-J. Sansuc, C. Soul\'e, {\sl  Torsion dans le
    groupe de Chow de codimension 2\/}, Duke Math. J. {\bf 50} (1983) 763--801.\
  
\bibitem{CTS} J.-L. Colliot-Th{\'e}l{\`e}ne,  A. Skorobogatov,
  {\sl The Brauer-Grothendieck group\/}, Ergebnisse der Mathematik und ihrer
  Grenzgebiete 3, Folge. A Series of Modern Surveys in Mathematics, 
{\bf 71}, Springer, Berlin, (2021). \

\bibitem{Cossart-J-S} V. Cossart, U. Jannsen, S. Saito, {\sl Canonical embedded and
    non-embedded resolution of singularities for excellent two-dimensional schemes\/},
  arXiv:0905.2191v2 [math.AG], (2009). \
  
  
\bibitem{Cossart-J-S-2} V. Cossart, U. Jannsen, S. Saito, {\sl Desingularization:
    invariants and strategy-application to dimension 2\/}, Lecture Notes in Math.,
  {\bf 2270}, Springer, Berlin, (2020). \ 

\bibitem{Deglise-Oriented} F. D{\'e}glise, {\sl Orientation theory in arithmetic
    geometry\/}, In:  K-Theory-Proceedings of the International Colloquium, Mumbai,
  (2016), 239--347, Hindustan Book Agency, New Delhi, 2018. \ 
  
\bibitem{DJK} F. D{\'e}glise, F. Jin, A. Khan, {\sl Fundamental classes in motivic
    homotopy theory\/}, J. Eur. Math. Soc., {\bf 23}, no.~ 12, (2021), 3935--3993. \

\bibitem{Deligne-Mumford} P. Deligne, D. Mumford, {\sl The irreducibility of the space of
    curves of given genus\/}, Publications math{\'e}matiques de l'IHES, {\bf 43}, (1980),
  137--252. \
  
%\bibitem{SP} A. J. de Jong et al., {\sl The Stacks Project\/}, Available at http://stacks.math.columbia.edu. (2020). \


\bibitem{Elmanto-Khan} E. Elmanto, A. Khan, {\sl Perfection in motivic homotopy
    theory\/},  Proc. Lond. Math. Soc., {\bf 120}, (2020), no. 1, 28--38. \ 


%\bibitem{ELSO} E. Elmanto, M. Levine, M. Spitzweck, P. {\O}stv{\ae}r, {\sl Algebraic cobordism and {\'e}tale cohomology\/}, Geom. Topol., {\bf 26}, (2022), no.~ 2, 477--586.
  
  



%\bibitem{EVMS}  P. Elbaz-Vincent, S. M{\"u}ller-Stach, {\sl Milnor $K$-theory, higher
%Chow groups and applications\/}, Invent. Math., {\bf 148}, (2002), 177--206. \

\bibitem{Fesenko-Vostokov} I. Fesenko, S. Vostokov, {\sl Local fields and their
    extensions\/}, Translations of Mathematical Monographs, American Math. Society,
  Providence, {\bf 121}, 2002. \

\bibitem{Forre-Crelle} P. Forr{\'e}, {\sl The kernel of the reciprocity map of varieties
    over local fields\/},\
 J. Reine Angew. Math., {\bf 698}, (2015), p. 55--69.\
 
%\bibitem{Forre-Thesis} P. Forr\'e, {\sl On the kernel of the reciprocity map of varieties
%    over local fields\/}, Thesis\

\bibitem{Fujiwara} K. Fujiwara, {\sl A proof of the absolute purity conjecture (after
    Gabber)\/}, Algebraic Geometry 2000, Azumino, Math. Soc. Japan, Tokyo, Adv. Stud. Pure
Math. {\bf 36} (2002), 153--183.\

%\bibitem{FV-Annals} E. Friedlander, V. Voevodsky, {\sl Bivariant cycle cohomology\/}, in ‘Cycles, transfers, and motivic homology theories’, Ann. of Math. Stud., {\bf 143}, Princeton University Press, (2000), 138--187.\


\bibitem {Fulton} W. Fulton, {\sl Intersection theory\/}, 2nd ed.,
Ergebnisse der Mathematik und ihrer Grenzgebiete 3,
Folge. A Series of Modern Surveys in Mathematics, 
{\bf 2}, Springer, Berlin, (1998). \


%\bibitem{Gupta-Krishna-JAG} R. Gupta, A. Krishna, {\sl $K$-theory and 0-cycles on
   % schemes\/}, J. Alg. Geom., {\bf 29}, (2020), 547--601. \
   
%\bibitem{GL-Bloch-Kato}   T. Geisser, M. Levine, {\sl The Bloch-Kato conjecture and a
%    theorem of Suslin- Voevodsky\/}, J. Reine Angew. Math. {\bf 530}, (2001), 55--103. \

%\bibitem{Geisser-Schmidt} T. Geisser, A, Schmidt, {\sl Tame class field theory of
%singular varieties over algebraically closed fields\/},
%Doc. Math., {\bf 21}, (2016), 91--123. \
  
\bibitem{Ghosh-Krishna-Bertini} M. Ghosh, A. Krishna, {\sl Bertini theorems
    revisited\/}, J. London Math. Soc., {\bf 108}, (2023), 1163-1192. \

\bibitem{Gille-Szamuely} P. Gille, T. Szamuely, {\sl Central Simple Algebras and
    Galois Cohomology\/}, Cambridge studies in advanced mathematics, {\bf 101}, 
Cambridge university press, (2006). \


\bibitem{GT} S. Gorchinsky, D. Tyurin, {\sl Relative Milnor-groups and differential
forms of split nilpotent extensions\/}, Izvestia Math., {\bf 82}(5), (2018), 880--913. \

\bibitem{SGA-1} A. Grothendieck, {\sl Rev{\'e}tements {\'e}tales et groupe 
fondamental (SGA 1)\/}, In: S{\'e}minaire de G{\'e}om{\'e}trie Alg{\'e}brique 
du Bois-Marie 1960--1961, Lecture Notes in Math., {\bf 224}, 
Springer-Verlag, Berlin, (1971). \

%
%\bibitem{Grothendieck-Deligne} A. Grothendieck, P. Deligne, {\sl Le classe de 
%cohomologie associ{\'e}e {\`a} cycle\/}, in `Seminaire de G{\'e}ometrie Alg{\'e}brique
%du Bois-Marie SGA $4 \frac{1}{2}$, Cohomologie Etale',
%Lecture Notes in Math. {\bf 569}, Springer Berlin, (1977), pp.~129--153. \
%
%
%\bibitem{Grothendieck-Murre}  A. Grothendieck, J. Murre, 
%{\sl The Tame Fundamental Group of a Formal Neighbourhood of a Divisor with 
%Normal Crossings on a Scheme\/},
%Lecture Notes in Math., {\bf 208}, Springer-Verlag, Berlin, (1971). \
%

\bibitem{Gruson-Raynaud} L. Gruson, M. Raynaud, {\sl Crit\`eres de platitude et de
    projectivit\'e\/}, Invent. Math., {\bf 13}, (1971), 1--89. \

\bibitem{Gupta-Krishna-AKT-1} R. Gupta, A. Krishna, {\sl Zero cycles with modulus and
    relative $K$-theory\/}, Ann. K-Theory,  {\bf 5}, (2020), 757--819. \
    
%\bibitem{Gupta-Krishna-AKT-2} R. Gupta, A. Krishna, {\sl Relative $K$-theory via
%    0-cycles
%    in finite characteristic\/}, Ann. K-Theory, {\bf 6}, (2021), 673--712. \

  \bibitem{Gupta-Krishna-Duality} R. Gupta, A. Krishna, {\sl Ramified class field
    theory and duality over finite fields\/},  arXiv:2104.03029v1  [math.AG], (2021). \

    
\bibitem{Gupta-Krishna-BF} R. Gupta, A. Krishna, {\sl Idele class groups
  with modulus\/}, Adv. Math., {\bf 404}, (2022), 1--75. \

\bibitem{Gupta-Krishna-REC} R. Gupta, A. Krishna, {\sl Reciprocity for 
    Kato-Saito idele class group with modulus\/}, J. Algebra., {\bf 608}, (2022),
  487--552. \

  

  
    
%\bibitem{GKR-arxiv} R. Gupta, A. Krishna, J. Rathore, {\sl Tame class field theory
   % over local fields\/}, arXiv:2209.02953, [math.AG], (2022). \
  
%\bibitem{Gupta-Krishna-Duality} R. Gupta, A. Krishna, {\sl Ramified class field theory
%    and duality over finite fields\/},  arXiv:2104.03029v1  [math.AG], (2021). \

%\bibitem{Hartshorne} R. Hartshorne, {\sl \/}Algebraic Geometry, Graduate Text in Mathematics, {\bf 52}, (1997), Springer-Verlag. \

% \bibitem{Gupta-Krishna-Lef} R. Gupta, A. Krishna, {\sl On the Lefschetz theorem of Kerz-Saito\/}, Preprint, (2021). \


\bibitem{Hiranouchi-1} T. Hiranouchi, {\sl Class field theory of open curves
    over $p$-adic fields\/}, Math. Z., {\bf 266}, (2010), 107--113. \
  
\bibitem{Hiranouchi-2} T. Hiranouchi, {\sl Class field theory for open curves over local
    fields\/}, J. Théor. Nombres Bordeaux, {\bf 30}, no 2, (2018), p. 501--524.\

%\bibitem{Hubner-Schmidt} K. H{\"u}bner, A. Schmidt, {\sl The tame site of a scheme\/},
%Invent. Math., {\bf 233}, (2021), 379--443. \

\bibitem{Jannsen-MA} U. Jannsen, {\sl Continuous {\'E}tale Cohomology\/},
  Math. Ann., {\bf 280}, (1988), 207--245. \

\bibitem{JS-Doc} U. Jannsen,  S. Saito, {\sl Kato homology of arithmetic schemes and
    higher class field theory over local fields\/}, Doc. Math., (2003), Extra Vol., Kazuya
  Kato’s Fiftieth Birthday, 479--538. \

\bibitem{JS-JAG} U. Jannsen, S. Saito, {\sl Bertini theorems and
    Lefschetz pencils over discrete valuation rings, with applications to
    higher class field theory\/}, J. Algebraic Geom., {\bf 21} (2012),
  683--705. \ 


  \bibitem{Kahn-03} B. Kahn, {\sl Some finiteness result for \'etale cohomology\/},
  Journal of Number Theory {\bf 99} (2003), 57--73. \

%\bibitem{Kaplan} S. Kaplan, {\sl Extensions of the Pontryagin duality I:
%    infinite products\/}, Duke Math. J., {\bf 15}, (1948), 649--658. \


\bibitem{Kato-cft-1} K. Kato, {\sl A generalization of local class field theory
    by using $K$-groups. I\/}, J. Fac. Sci., Univ. Tokyo, {\bf 27}, (1980), 303--376. \ 

  
\bibitem{Kato80} K. Kato, {\sl A generalization of local class field theory by using
    $K$-groups. II\/}, J. Fac. Sci., Univ. Tokyo, Sect. I A {\bf 27}, (1980), no. 3,
  p. 603--683. \

\bibitem{Kato-Crelle} K. Kato, {\sl A Hasse principle for two dimensional global
  ﬁelds\/}, J. Reine Angew. Math.,  {\bf 366}, (1986), 142--183.

\bibitem{Kato86} K. Kato, {\sl Milnor $K$-theory and Chow group of zero 
cycles\/}, In: Applications of algebraic K-theory to algebraic 
geometry and number theory, Contemp. Math., {\bf 55}, 
Amer. Math. Soc, Providence, RI, (1986), 241--255. \



  
%\bibitem{Kato89} K. Kato, {\sl Swan conductors for characters of degree one 
%in the imperfect residue field case\/}, in Algebraic $K$-Theory and Algebraic 
%Number Theory (Honolulu, Hawaii, 1987), Contemp. Math. {\bf 83}, 
%Amer. Math. Soc., Providence, (1989), 101--131. \ 
%

\bibitem{Kato-Saito-83} K. Kato, S. Saito, {\sl Two dimensional class field theory\/},
  Advanced Studies in Pure Mathematics, Galois Groups and their Representations,
  {\bf 2}, (1983), 103--152. \

\bibitem{Kato-Saito-1} K. Kato, S. Saito, {\sl Unramified class field theory of
arithmetic surfaces\/}, Ann. of Math., {\bf 118}, No. 2, (1983), 241--275. \





\bibitem{Kato-Saito-2} K. Kato, S. Saito, {\em Global class field theory of 
arithmetic schemes\/}, In: Applications of algebraic K-theory to algebraic 
geometry and number theory, Contemp. Math., {\bf 55}, 
Amer. Math. Soc, Providence, RI, (1986), 255--331. \


\bibitem{Kelly} S. Kelly, {\sl Voevodsky motives and $l$dh-descent\/}, Ast\'erisque, {\bf 391}, (2017), pp. 125. \ 

%\bibitem{Kerz-Invent} M. Kerz, {\em The Gersten conjecture for Milnor K-theory\/}, 
%Invent. Math., {\bf 175}, (2009),  1--33. \

%
\bibitem{Kerz-JAG} M. Kerz, {\sl Milnor K-theory of local rings with finite residue fields\/}, J. Alg. Geom., 19, (2010), 173--191. \


%\bibitem{Kerz-MRL} M. Kerz, {\sl Ideles in higher dimensions\/}, 
%Math. Res. Lett., {\bf 18}, (2011), 699--713. \
  

%\bibitem{Kerz-Saito-Hasse} M. Kerz, S. Saito, {\sl Cohomological Hasse principle and
%motivic cohomology for arithmetic schemes\/},
%Publ. Math. Inst. Hautes {\'E}tudes Sci., {\bf 115}, (2012), 123--183. \


%\bibitem{Kerz-Saito-Duke} M. Kerz, S. Saito, {\sl Chow group of 0-cycles with modulus
%and higher dimensional class field theory\/}, Duke Math. J., {\bf 165},
%(2016), no. 15, 2811--2897. \

\bibitem{Kerz-Schmidt-JNT} M. Kerz, A. Schmidt, {\sl Covering data and higher
    dimensional global class field theory\/}, J. of Number Theory, {\bf 129}, (2009),
  2569--2599. \

\bibitem{Kerz-Schmidt-MA} M. Kerz, A. Schmidt, {\sl On different notions of tameness in
    arithmetic geometry\/}, Math. Ann., {\bf 346}, (2010),  641--668.\

\bibitem{Krishna-Levine} A. Krishna, M. Levine, {\sl Additive higher Chow groups of
    schemes\/}, J. reine angew. Math., {\bf 619}, (2008), 75--140. \

%\bibitem{Krishna-Park-AGT} A. Krishna, J. Park, {\sl Semitopologization in motivic
%    homotopy theory and applications\/}, Algebraic \& Geometric Topology,
%  {\bf 15}, (2015), 823--861. \
 
%\bibitem{Krishna-Park-MRL} A. Krishna, J. Park, {\sl A module structure and a vanishing 
%theorem for cycles with modulus\/}, Math. Res. Lett., {\bf 24}, (2017),
%1147--1176. \

\bibitem{Krishna-Pelaez-AKT} A. Krishna, P. Pelaez, {\sl Slice spectral sequence for
singular schemes and applications\/}, Ann. K-Theory, {\bf 3}, (2018), 657--708. \
  
\bibitem{Landsburg} S. Landsburg, {\sl Relative Chow groups\/}, Illinois J. Math.,
  {\bf 35}, (1991), 618--641. \

%\bibitem{Laumon} G. Laumon, {\sl Semi-continuit{\'e} du conducteur de Swan 
%(d'apr{\'e}s P. Deligne)\/}, In: `The Euler-Poincar{\'e} characteristic',
%edited by J.-L. Verdier, Soc. Math.
%France, Paris, Ast{\'e}risque, {\bf 83}, (1981),  173--219. \

\bibitem{Levine-K} M. Levine, {\sl Relative Milnor $K$-Theory\/}, $K$-Theory, {\bf 6},
  (1992), 113--175. \

%\bibitem{Lipman69} J. Lipman, {\sl Rational singularities, with applications to
%    algebraic
%surfaces and unique factorization\/}, Publ. Math. IHES No. {\bf 36} (1969), 195--279.\
%
%
%\bibitem{Lipman78} J. Lipman, {\sl Desingularization of two-dimensional schemes\/}, Ann.
%  Math. (2) {\bf107} (1978), no. 1, 151--207.\

%
%\bibitem{Liu02} Q. Liu, {\sl Algebraic geometry and arithmetic curves\/}, Oxford
%  Graduate Texts in Mathematics {\bf 6}, Oxford University Press, 2002.\


%
%\bibitem{Matsuda} S. Matsuda, {\sl On the Swan conductors in positive 
%characteristic\/}, Amer. J. Math., {\bf 119}, (1997), 705--739. \

\bibitem{Matsumura} H. Matsumura, {\sl Commutative ring theory\/},
Cambridge studies in advanced mathematics, {\bf 8}, (1997),
Cambridge University Press, Cambridge. \

\bibitem{MVW} C. Mazza, V. Voevodsky, C. Weibel, 
{\sl Lecture notes on motivic cohomology,\/} Clay Mathematics
Monographs, 2, American Mathematical Society, Providence, (2006).\


\bibitem{Merkurjev} A. Merkurjev, {\sl On torsion in $K_2$ of local fields\/},
  Ann. Math., {\bf 118}, (1983), 375--381. \ 

\bibitem{Milne-etale} J. Milne, {\sl \'Etale cohomology\/}, Princeton Mathematical
  Series, Princeton University Press, {\bf 33}, (1980). \


%\bibitem{Miyazaki} H. Miyazaki, {\sl Cube invariance of higher Chow groups
%with modulus\/}, J. Alg. Geom., {\bf 28}, (2019), 339--390. \
  
\bibitem{Nagata} M. Nagata, {\sl Embedding of an abstract variety in a complete
variety\/}, J. Math. Kyoto univ., {\bf 2}, (1962), 2--10. \
  

\bibitem{Navarro} A. Navarro, {\sl Riemann-Roch for homotopy invariant $K$-theory and
    Gysin morphisms\/}, Adv. Math., {\bf 328}, (2018), 501--554. \

\bibitem{Nest-Suslin} Y. Nesterenko, A. Suslin, {\sl Homology of the full linear
    group over a local ring, and Milnor's $K$-theory\/},
  Math. USSR Izv., {\bf 34}, no.~1, (1990), 121--145. \

\bibitem{Neukirch} J. Neukirch, {\sl Algebraic number theory\/}, Grundlehren der
  Mathematischen Wissenschaften, vol. {\bf 322}, Springer, (1999), xvii+571.\

\bibitem{Quillen} D. Quillen, {\sl Higher  algebraic $K$-theory: I\/},
  Lecture Notes in Math., {\bf 341}, (1973), Springer Berlin, 85--147. \

\bibitem{Pro-fin} L. Ribes, P. Zalesskii, {\sl Profinite Groups\/},
2nd,  Ergebnisse der Mathematik und iher Grenzgibiete, {\bf 40},
Springer, (2010). \ 
%
%\bibitem{Raskind} W. Raskind, {\sl Abelian class field theory of arithmetic 
%schemes\/}, `(Santa Barbara, CA, 1992)', Proc. Sympos. Pure Math., {\bf 58}, 
%Part-1, (1995), 85--187, Amer. Math. Soc., Providence, RI.  \

\bibitem{Riou} J. Riou, {\sl Classes de Chern, morphismes de Gysin,
    puret{\'e} absolue\/}, In: Travaux de Gabber
    sur l'uniformisation locale et la cohomologie des sch{\'e}emas quasi-excellents,
Ast{\'e}risque, {\bf 363, 364}, (2014), 652 p., 
arXiv:1207.3648 [math.AG], (2007). \
  


%\bibitem{Saito-Annals} S. Saito, {\sl Unramified class field theory of arithmetical
%schemes\/}, Ann. of Math., {\bf 121}, no.~2, (1985), 251--281. \

\bibitem{Saito-JNT} S. Saito, {\sl Class field theory for curves over local fields\/}, 
J. Number Theory {\bf 21}, (1985), no. 1, p. 44--80. \

\bibitem{Saito-Duality} S. Saito, {\sl A global duality theorem for varieties over
    global fields\/}, In: ‘Algebraic K-theory: connections
  with geometry and topology’ (Lake Louise, 1987), NATO ASI Series, Series C:
  Mathematical and
Physical Sciences, Kluwer Academic Publishers, {\bf 279}, (1987), 425--444.\

\bibitem{Saito20} T. Saito, {\sl A characterization of ramification groups
  of local fields with imperfect residue fields\/},  In:  `Arithmetic Geometry',
  Proceedings of International conference on arithmetic geometry, Mumbai, (2020),
  421--433. \

%\bibitem{Saito-Sato-Ann} S. Saito, K. Sato, {\sl A finiteness theorem for 
%zero-cycles over $p$-adic fields\/}, Ann. of Math., {\bf 172}, (2010),
%1593--1639. \

%\bibitem{Sato-JNT} K. Sato, {\sl Non-divisible cycles on surfaces over local
%fields\/}, J. Number Theory, {\bf 114}, (2005), 272--297. \

%\bibitem{Schmidt-ANT} A. Schmidt, {\sl Singular homology of arithmetic schemes\/},
%Algebra \& Number Theory, {\bf 1}, no. 2, (2007), 183--222. \
  
%\bibitem{Schmidt-Spiess} A. Schmidt, M. Spie{\ss}, {\sl Singular homology and
%class field theory of varieties over finite fields\/},
%J. Reine Angew. Math., {\bf 527}, (2000), 13-36. \

%\bibitem{Serre-AGCF} J.-P. Serre, {\sl Algebraic Groups and Class Fields\/},
%  Graduate Texts in Mathematics, {\bf  117}, Springer-Verlag, 1975. \

\bibitem{Serre-LF} J.-P. Serre, {\sl Local Fields\/},
  Graduate Texts in Mathematics, {\bf  67}, Springer-Verlag, 1979. \

\bibitem{Serre-GC} J.-P. Serre, {\sl Galois Cohomology\/},
  Springer Monographs in Mathematics, Springer-Verlag, 1997. \


\bibitem{SP} Stacks project authors, {\sl The Stacks project\/},
  https://stacks.math.columbia.edu, (2025). \

\bibitem{Suslin-Izv} A. Suslin, {\sl Reciprocity laws and the stable rank of polynomial
    rings\/}, Math. USSR-Izv., {\bf 15}, no.~ 3, (1980), 589--623. \

 \bibitem{Suslin-AKT} A. Suslin, {\sl Motivic complexes over nonperfect fields\/},
  Ann. $K$-Theory, {\bf 2}, (2017), 277--302. \

%\bibitem{SV-Invent} A. Suslin, V. Voevodsky, {\sl Singular homology of abstract
%    algebraic varieties\/}, Invent. Math., {\bf 123}, (1996), 61--94. \
  
%\bibitem{SV} A. Suslin, V. Voevodsky, {\sl Relative cycles and Chow sheaves\/},
%  In Cycles,
%  transfers, and motivic homology theories, volume {\bf 143} of Ann. of Math. Stud.,
%  pages 10--86. Princeton Univ. Press, Princeton, NJ, 2000. \

\bibitem{Szamuely} T. Szamuely, {\sl Galois Groups and Fundamental Groups\/},
Cambridge studies in advanced mathematics, {\bf 117}, 
Cambridge university press, (2009). \

\bibitem{Tate} J. Tate, {\sl Duality theorems in Galois cohomology over number
fields\/}, Proc. Intern. Congress Math. 1962, Stockholm, (1963), 288--295. \


\bibitem{Tate-Kyoto} J. Tate, {\sl On the torsion in $K_2$ of fields\/},
  In: `Algebraic number theory', Kyoto Internat. Sympos.,
  Res. Inst. Math. Sci., Univ. Kyoto, Kyoto, (1976), pp. 243--261,
  Japan Soc. Promotion Sci., Tokyo, 1977. \ 

  
\bibitem{Temkin} M. Temkin, {\sl Tame distillation and desingularization by
    p-alterations\/}, Ann. of Math., {\bf 186}, (2017), 97--126. \

%\bibitem{TT} R. Thomason, T.Trobaugh, {\sl Higher algebraic $K$-theory of
%    schemes and of derived categories\/}, In: `The Grothendieck Festschrift III',
%Progress in Math., {\bf 88}, Birkh{\"a}user, Boston, (1990), 247--435.


\bibitem{Totaro} B. Totaro, {\sl Milnor $K$-theory is the simplest part of algebraic
    $K$-theory\/}, K-Theory, {\bf 6}, no.~1, (1992), 177--189. \

  
\bibitem{Uzun} M. Uzun, {\sl Motivic homology and class field theory over $p-$adic
    fields\/}, Journal of Number Theory, {\bf 160}, (2016), 566--585. \


\bibitem{Voev-Selecta} V. Voevodsky, {\sl Homology of Schemes\/}, Selecta Math.,
  New Ser., {\bf 2}, (1996), 111--153. \

  
\bibitem{Voe00} V. Voevodsky, {\sl Triangulated categories of motives over a
    field\/}. In: Cycles, transfers, and motivic homology theories, {\bf 143},
  Ann. of Math. Stud., {188--238}, Princeton Univ. Press, Princeton, NJ, 2000. \

 \bibitem{Voe-imrn} V. Voevodsky, {\sl Motivic cohomology groups are isomorphic to higher
    chow groups in any characteristic\/},  Int. Math. Res. Not. (IMRN), {\bf 7}, (2002),
  351--355. \ 

%\bibitem{Voe-Canc} V. Voevodsky, {\sl Cancellation theorem\/}, Doc. Math. Extra
%  volume: Andrei A. Suslin sixtieth birthday (2010), 671–685.\

\bibitem{Voe-BK} V. Voevodsky, {\sl On motivic cohomology with ${\Z}/{l}$-coeffcients\/},
  Ann. of Math., {\bf 174}, (2011), 401--438. \
  
%\bibitem{Weibel-80} C. Weibel, {\sl $K_2, \ K_3$ and nilpotent ideals\/},
%  J. Pure Appl. Algebra, {\bf 18}, (1080), 333--345. \

%\bibitem{Weibel-09} C. Weibel, {\sl The norm residue isomorphism theorem\/}, J. Topol.
%  {\bf 2}, (2009), 346--372.\
  
\bibitem{K-book} C. Weibel, {\sl The $K$-book, An introduction to algebraic
$K$-theory\/}, Graduate studies in Mathematics, {\bf 145},
American Mathematical Soc., Providence, (2013). \


%\bibitem{Wiesend} G. Wiesend, {\sl Tamely ramified covers of varieties and arithmetic
%schemes\/},  Forum Math., {\bf 20},  no. 3, (2008), 515--522. \

\bibitem{Yamazaki} T. Yamazaki, {\sl The Brauer-Manin pairing, class field theory, and
    motivic homology\/}, Nagoya Math.
J., {\bf 210}, (2013), 29--58.\

%
%\bibitem{Yatagawa} Y. Yatagawa, {\sl Equality of two non-logarithmic 
%ramification filtrations of abelianized Galois group in positive 
%characteristic\/}, Doc. Math., {\bf 22}, (2017), 917--952. \

\bibitem{Yoshida03}  T. Yoshida, {\sl Finiteness theorems in the class field theory of
    varieties over local fields\/}, J. Number Theory {\bf 101}, (2003), no. 1,
  p. 138--150.\
\end{thebibliography}
                       \end{document}